\documentclass[11pt,a4paper]{article}
    
\usepackage{graphicx}
\usepackage{amsmath,amsfonts,amssymb}
\usepackage{lscape}
\usepackage{url}
\usepackage{hyperref}
\usepackage{tikz}
\usetikzlibrary{arrows.meta}
\usepackage{booktabs, caption}
\usepackage{float}
\usepackage{subcaption} 
\usepackage{natbib} 
\usepackage[ruled,vlined,linesnumbered]{algorithm2e}
    
\usepackage[utf8]{inputenc}
\usepackage{pgfplots}
\pgfplotsset{compat=newest}
\usepgfplotslibrary{groupplots}
\usepgfplotslibrary{dateplot}
\graphicspath{{Images/}}
\usepackage{xcolor,colortbl}  
\usepackage{threeparttable}

\usepgfplotslibrary{statistics}
\usetikzlibrary{pgfplots.statistics} 
\pgfplotsset{width=8cm,compat=1.15}

\newtheorem{proposition}{Proposition}[section]

\newenvironment{proof}[1][Proof]{\textbf{#1.} }{\hfill$\Box$}
    
\begin{document}
    
\title{Finding $K$ dissimilar paths using integer linear formulations}
    
\author{Ali Moghanni%
	\thanks{Universidade de Coimbra, CMUC, Departmento de Matemática,
		3001-501 Coimbra, Portugal,
		(\texttt{ali.moghanni@mat.uc.pt})}
	\and
	Marta Pascoal%
	\thanks{Universidade de Coimbra, CMUC, Departamento de Matemática,
		3001-501 Coimbra, Institute for Systems Engineering and Computers -- Coimbra,
		rua S\'\i lvio Lima, Pólo II, 3030-290 Coimbra, Portugal,
        Dipartimento di Elettronica, Informazione e Bioingegneria, Politecnico di Milano, Piazza Leonardo da Vinci, 32, Milan, 20133, Italy
		(\texttt{marta.brazpascoal@polimi.it})}
		\footnote{Corresponding author}
	\and
	Maria Teresa Godinho%
	\thanks{Department of  Mathematics and Physical Sciences, Polytechnic Institute of Beja, Campus do Instituto Politécnico de Beja, rua Pedro Soares, 7800-295 Beja, Portugal, and CMAFcIO, Faculdade de Ciências da Universidade de Lisboa, Campo Grande, 1749-016 Lisboa 
		(\texttt{mtgodinho@ipbeja.pt})}
}

\date{}

\maketitle
\begin{abstract}
While finding a path between two nodes is the basis for several applications, the need for alternative paths also may have various motivations.
For instance, this can be of interest for ensuring reliability in a
telecommunications network, for reducing the consequences of possible accidents in the transportation of hazardous materials, or to decrease the risk of robberies in money distribution.
Each of these applications has particular characteristics, but they all have the common purpose of searching for a set of paths which are as dissimilar as possible with respect to the nodes/arcs that compose them.

In this work we present linear integer programming formulations for finding $K$ dissimilar paths, with the main goal of preventing the overlap of arcs in the paths for a given integer $K$.
The different formulations are tested for randomly generated general networks and for grid networks.
The obtained results are compared in terms of the solutions' dissimilarity and of the run time.
Two of the new formulations are able to find 10 paths with better average and minimum dissimilarity values than an iterative approach in the literature, in less than 20 seconds, for random networks with 500 nodes and 5000 arcs.

\end{abstract}

\baselineskip15pt
\section{Introduction\label{sec:1}}  
Let $(N,A)$ denote a given directed network consisting of a set
$N=\{1,\dots,n\}$ of nodes and a set $A\subseteq N\times N$ of $m$
arcs. Let $s$ and $t$ be two different nodes of $N$, called the
source node and the target node, respectively.
Finding a path in the network $(N,A)$ between the nodes $s$ and $t$
is one of the most classical and widely used network optimization
problems, and the basis for several applications in operations research.
Studying the determination of alternative paths, on the other hand, is an interesting problem by itself that stems from different real-life problems but has been considerably less studied than the former.
For instance, in a modern and industrialized society, routing hazardous
materials like poisonous gases or radioactive materials is an important issue, so the need for alternative safe routes is crucial for reducing the risk of disaster in case of accidents or if the best route becomes infeasible due to road construction \cite{Akgun_2000,CaramiaJMM_2010,DellOlmo_2005,GopalanOR_1990}. Repeating paths is also avoided in money collection, where having alternative paths/routes decreases the risk of robberies and can be used in case of danger of robberies
\cite{CalvoCOR_2003, Constantino_2017}. Additionally, in
telecommunications, a backup path is often replaced by a primary one if a failure occurs along it or if it can be used simultaneously to spread information transmitted at a specific time \cite{GomesJTIT_2010,GomesPNC_2016}.
    
Let $K\in\mathbb{N}$ be a given number of alternative paths to be found.
The definition of alternative paths may vary depending on the application, the common denominator being that the paths in the solution should share the least possible network resources.
Several works use dissimilarity measures between two paths as the metric for achieving this purpose, nevertheless, also this notion is not uniquely defined, nor would that be desirable provided that the metrics are often tailored to the application.
For instance, \cite{Erkut_1998} developed four indices for measuring the similarity between two paths, defined as follows:
\begin{quote}
\begin{tabular}{lll}
Index 1:
$S_1(p,q)=\dfrac{1}{2}
\left(\dfrac{L(p\cap q)}{L(p)}+\dfrac{L(p\cap q)}{L(q)}\right)$\\\\

Index 2: $S_2(p,q)=\sqrt{\dfrac{L^2(p\cap q)}{L(p)L(q)}}$\\\\

Index 3: $S_3(p,q)=\dfrac{L(p\cap q)}{\max\{L(p),L(q)\}}$\\\\

Index 4: $S_4(p,q)=\dfrac{L(p\cap q)}{L(p\cup q)}$
\end{tabular}
\end{quote}
where $p$ and $q$ are two paths and $L(p)$ denotes the length of path $p$, that is, its number of arcs. The dissimilarity between $p$ and $q$ is then given by $D_i(p,q)=1-S_i(p,q)$, for $i=1,2,3,4$.
The dissimilarities vary between 0 and 1, the first when the two paths coincide and the latter when they are arc disjoint.
The authors also showed that there is a strong correlation between these indices.
Other works have extended these concepts by including information about the underlying area affected by the paths or the distance between them, once again depending on the problem \cite{DellOlmo_2005,MARTI_2009}.
    
Additionally, in concrete applications the problem has frequently been handled from a bi-objective point of view, having the goals of optimizing both the total paths length/cost as well as the dissimilarity of the set of paths.
Now, while the shortest path problem or the ranking of $K$ shortest paths problem are well-known and well-studied problems, the same is not true when the objective function represents paths dissimilarity.
Many of these bi-objective problems have been addressed from an algorithmic approach whose primary goal is not to optimize the dissimilarity and, to our knowledge, there are no published studies considering the paths dissimilarity problem from an integer programming point of view.

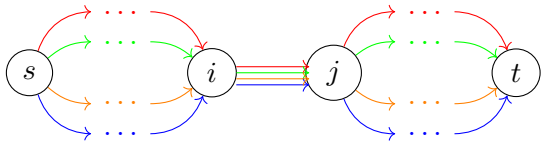
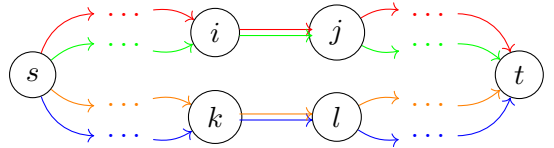
\begin{figure}[htb]
\centering
\begin{subfigure}[h]{0.45\textwidth}
\centering
    \begin{tikzpicture}
        [scale=.8,every node/.style={circle,text=black,draw,minimum size=0.5cm}]
        \node (a) at (0,0) [] {$s$};
        \node (b) at (3,0) [] {$i$};
        \node (c) at (5,0) [] {$j$};
        \node (d) at (8,0) [] {$t$};
    
        \node[draw=none,green] (ellipsis1) at (1.5,0.5) {$\ldots$};
        \node[draw=none,red] (ellipsis2) at (1.5,1) {$\ldots$};
        \node[draw=none,blue] (ellipsis3) at (1.5,-1) {$\ldots$};
        \node[draw=none,orange] (ellipsis4) at (1.5,-0.5) {$\ldots$};
    
        \node[draw=none,green] (ellipsis5) at (6.5,0.5) {$\ldots$};
        \node[draw=none,red] (ellipsis6) at (6.5,1) {$\ldots$};
        \node[draw=none,blue] (ellipsis7) at (6.5,-1) {$\ldots$};
        \node[draw=none,orange] (ellipsis8) at (6.5,-0.5) {$\ldots$};
    
        \draw (c) edge[->,bend left=35,red] (ellipsis6);
        \draw (ellipsis6) edge[->,bend left=35,red] (d);
        \draw (c) edge[->,bend right=35,blue] (ellipsis7);
        \draw (ellipsis7) edge[->,bend right=35,blue] (d);
        \draw (c) edge[->,bend left=20,green] (ellipsis5);
        \draw (ellipsis5) edge[->,bend left=20,green] (d);
        \draw (c) edge[->,bend right=20,orange] (ellipsis8);
        \draw (ellipsis8) edge[->,bend right=20,orange] (d);

        \draw (a) edge[->,bend left=35,red] (ellipsis2);
        \draw (ellipsis2) edge[->,bend left=35,red] (b);
        \draw (a) edge[->,bend right=35,blue] (ellipsis3);
        \draw (ellipsis3) edge[->,bend right=35,blue] (b);
        \draw (a) edge[->,bend left=20,green] (ellipsis1);
        \draw (ellipsis1) edge[->,bend left=20,green] (b);
        \draw (a) edge[->,bend right=20,orange] (ellipsis4);
        \draw (ellipsis4) edge[->,bend right=20,orange] (b);
    
        \draw[->,red] (3.4,0.1)--(4.6,0.1);
        \draw[->,green] (3.4,0)--(4.6,0);
        \draw[->,blue] (3.4,-0.2)--(4.6,-0.2);
        \draw[->,orange] (3.4,-0.1)--(4.6,-0.1);
    \end{tikzpicture}
\caption{Solution with one repeated arc, shared by several paths\label{fig:fig2newa}}
\end{subfigure}\hfill
\begin{subfigure}[h]{0.45\textwidth}
\centering
    \begin{tikzpicture}
        [scale=.8,every node/.style={circle,text=black,draw,minimum size=0.5cm}]
        \node (a) at (0,0) [] {$s$};
        \node (b) at (3,0.7) [] {$i$};
        \node (c) at (5,0.7) [] {$j$};
        \node (d) at (8,0) [] {$t$};
        \node (e) at (3,-0.7) [] {$k$};
        \node (f) at (5,-0.7) [] {$l$};
    
        \node[draw=none,green] (ellipsis1) at (1.5,0.5) {$\ldots$};
        \node[draw=none,red] (ellipsis2) at (1.5,1) {$\ldots$};
        \node[draw=none,blue] (ellipsis3) at (1.5,-1) {$\ldots$};
        \node[draw=none,orange] (ellipsis4) at (1.5,-0.5) {$\ldots$};
    
        \node[draw=none,green] (ellipsis5) at (6.5,0.5) {$\ldots$};
        \node[draw=none,red] (ellipsis6) at (6.5,1) {$\ldots$};
        \node[draw=none,blue] (ellipsis7) at (6.5,-1) {$\ldots$};
        \node[draw=none,orange] (ellipsis8) at (6.5,-0.5) {$\ldots$};
    
        \draw (c) edge[->,bend left=20,red] (ellipsis6);
        \draw (ellipsis6) edge[->,bend left=35,red] (d);
        \draw (f) edge[->,bend right=20,blue] (ellipsis7);
        \draw (ellipsis7) edge[->,bend right=35,blue] (d);
        \draw (c) edge[->,bend right=20,green] (ellipsis5);
        \draw (ellipsis5) edge[->,bend left=20,green] (d);
        \draw (f) edge[->,bend left=20,orange] (ellipsis8);
        \draw (ellipsis8) edge[->,bend right=20,orange] (d);

        \draw (a) edge[->,bend left=35,red] (ellipsis2);
        \draw (ellipsis2) edge[->,bend left=20,red] (b);
        \draw (a) edge[->,bend right=35,blue] (ellipsis3);
        \draw (ellipsis3) edge[->,bend right=20,blue] (e);
        \draw (a) edge[->,bend left=20,green] (ellipsis1);
        \draw (ellipsis1) edge[->,bend right=20,green] (b);
        \draw (a) edge[->,bend right=20,orange] (ellipsis4);
        \draw (ellipsis4) edge[->,bend left=20,orange] (e);
    
        \draw[->,red] (3.4,0.75)--(4.6,0.75);
        \draw[->,green] (3.4,0.65)--(4.6,0.65);
        \draw[->,blue] (3.4,-0.75)--(4.6,-0.75);
        \draw[->,orange] (3.4,-0.65)--(4.6,-0.65);
    \end{tikzpicture}
\caption{Solution with many repeated arcs, shared by few paths\label{fig:fig2newb}}
\end{subfigure}
\caption{Different sets of $K=4$ paths\label{fig:fig2new}}
\end{figure}
    
Filling this gap and deepening the understanding of such problems are the main motivations for the present work, also because this may be a relevant contribution for an efficient treatment of a bi-objective problem involving paths dissimilarity.

As mentioned earlier, it is not uncommon to find different understandings of the term ``dissimilarity'' in the literature. In this work, we focus on the conception of dissimilarity as defined by $D_1$. Thus, the presented models aim at producing sets of $K$ paths with good scores in terms of $D_1$.

However, modeling $D_1$ as the objective function of an integer linear programming (ILP) model, presents some difficulties as this is a non linear metric involving an underlying combinatorial problem.
On the other hand, it is intuitive that  minimizing the number of arcs shared by the $K$ paths or minimizing the number of paths that share a common arc in the $K$ paths, favor the dissimilarity of the solution. Furthermore, these problems can be modeled quite easily, obviating the above mentioned difficulties. Thus these alternative ways of looking into the $K$ dissimilar paths problem seem promising and are worth exploring.

In the present work we introduce and compare three families of ILP formulations, each one addressing one of the strategies mentioned before. Figure \ref{fig:fig2new} illustrates their differences and emphasizes some of their limitations. 

For simplicity, assume that all the paths represented in Figure~\ref{fig:fig2new} have the same length.
There are 6 arc overlaps for every pair of paths in Figure \ref{fig:fig2newa} (that is, the sum of $L(p\cap q)$ for every pair of paths $p,q$ in the solution is equal to 6), while  in the solution depicted in Figure~\ref{fig:fig2newb} there are only 2.
Therefore, the second is a better solution than the first with regard to their dissimilarity.
However, if one chooses to count the arcs shared by more than one path, there is only 1 in the solution in Figure \ref{fig:fig2newa} and there are 2 in the solution in Figure \ref{fig:fig2newb}.
Thus, the first solution is the best with regard to this metric.
This situation illustrates the major drawback of the strategy devised by this approach -- the lack of control over the number of overlaps associated to arcs that are used by more than one path.
By contrast, if counting the number of arc repetitions in the solution (given by the number of times an arc is present in the solution, besides its first use), there are 3 in the solution depicted in Figure~\ref{fig:fig2newa} and 2 in the solution depicted in Figure \ref{fig:fig2newb}, favoring the solution that seems to be the most dissimilar one.
Still, this third approach is not exempted of drawbacks, as it does not distinguish between alternative solutions with different dissimilarities.

To (partially) overcome the drawbacks of the two approaches a set of additional constraints is proposed, which imposes a bound on the number of paths that use each arc in the solution.
Formal definitions of each of these problems and concepts are introduced later on.

The remainder of this work is structured as follows.
Section~\ref{sec:2} gives a literature overview of problems related
with finding alternative paths.
In Section~\ref{sec:3} the problem of finding $K$ dissimilar paths between two nodes is presented, and in the next three sections ILP formulations are introduced to provide solutions based on the approaches described above. Computational results for different variants of each formulation are presented at the end of each section.
A set of constraints which help the new formulations to obtain more dissimilar solutions is introduced in Section~\ref{sec:7}.
Overall computational experiments for all approaches are presented in Section~\ref{sec:8}.
The performance of the proposed formulations is analyzed and this is compared to an intuitive and classical approach in the literature of finding alternative paths, the iterative penalty method (IPM) \cite{Johnson_1992}.
The formulations are compared both in terms of the run time and of the dissimilarity of the output solutions, derived from $D_1$.
Some conclusions and future directions for research are outlined in Section~\ref{sec:9}.

\section{Literature review\label{sec:2}}    
Several problems focus on finding a single path between two nodes in a network, which optimizes either a certain criterion or several criteria simultaneously, while others aim at finding a set with a given number of paths, again with respect to one, or several, criteria.
The single path problems have practical interest by themselves, but finding a set of paths may still be relevant, for instance to ensure reliability and having a replacement path in case of failure in the primary one, or simply if several alternatives should avoid sharing resources with other paths.
In this case ranking paths provides a pre-defined number of paths from the source node to the target node by increasing order of the objective function.

Several network optimization problems search for a path between two nodes which optimizes a certain criterion. In many cases it is of interest to extend this problem by searching not only for the best solution but also for the second best, the third best and so on, that is, to rank paths by increasing order of the objective function.
In practice this is useful, for instance, when the paths need to satisfy additional constraints, which can be checked as new solutions are found.
This problem, known as the $K$ shortest paths problem, was first
proposed in 1959, by \citet{hoffman_1959}, and is usually
classified into two variants, one that aims at the determination of unconstrained paths and another one for which the nodes in each solution cannot appear more than once.
Despite the first being easier to solve than the second, both can be solved in polynomial space and time, depending on the number $K$ of paths to be found and on the size of the network. See for example \citet{eppstein_1998,jimenez_1999,martins_1984,Martins_1999} for works on ranking unconstrained paths and \citet{katoh_1982,martins_2003,yen_1971} for works on ranking loopless paths.

The search for solutions when ranking paths is guided by the objective function, therefore, very often the solutions which are close in terms of the cost are also similar in terms of their composition.
In the $K$ disjoint paths problem, a cost objective function of $K$ paths is minimized, while the overlaps between them are forbidden.
The problem can be classified into arc disjoint or node disjoint, the second one being a particular case of the first (for instance, if every node is split into two nodes that are linked by an arc).
The disjointness of the paths is often a requirement in
telecommunications, in order to ensure the reliability of communications.
In practice this is managed by the computation of a pair of paths connecting two given nodes, a primary path to be used as a first option and a backup path to replace the first one if there is a failure along its arcs or nodes.
The determination of $K$ disjoint simple paths has been studied by~\citet{Bhandari_1998,Suurballe_1974,Suurballe_1984}.
Their approaches consist of formulating the problem as a minimum cost flow problem and propose the application of a labeling algorithm, changing the given network.
The arc disjoint version of the problem has been studied in~\citet{frank_1985,guruswami_2003,vygen_1995}.
A review on disjoint path problems can be found in~\citet{Iqbal_2015}.

A handicap of the $K$ disjoint paths problem is that the disjointness condition may be too demanding for some instances and no solution is returned in those cases.
The dissimilar paths problem has been studied in the context of hazardous materials transportation, where the alternative paths should not share a large number of arcs and they should be relatively short in length.
Different methods have been proposed for approaching this problem.
The IPM \citet{Johnson_1992} is one of the most intuitive methods, based on the iterative application of shortest path algorithms.
At each iteration, a cost penalty is associated to each selected arc to discourage them of appearing in the forthcoming iteration;
hence, generating dissimilar paths.
Another proposed method is the Gateway Shortest Path, \citet{Lombard_1993}.
In this case, the generated shortest paths should go through a given set of nodes called a ``gateway''.
Additionally, the concept of ``area under a path'' is used to evaluate the similarity between two paths. The minimax method, by \citet{Kuby_1997}, selects paths starting from $K$ assigned paths using some dissimilarity indices.
\citet{Akgun_2000} reviews the three mentioned methods for generating dissimilar paths, and proposes another dissimilar path model that makes use of a $p$-dispersion location model, \citet{Erkut_1990}. \citet{Erkut_1998} presents four indices to measure the dissimilarity among two paths, one of which will be used later. In~\citet{Carotenuto_2007a}, the authors introduce a model for generating dissimilar paths that takes into account also the risk induced on the arcs in the neighborhood of a selected path. Another work, \citet{Carotenuto_2007b}, also considers the need to distribute the risk of the paths in an equitable way with respect to both the space and the time, avoiding as much as possible the presence of more than one hazardous vehicle at the same time on the same zone.
Later on \citet{DellOlmo_2005} study the problem from a multi-objective perspective. They introduce the concept of ``buffer zone'' in the measure of similarity.
\citet{MARTI_2009} choose approaches different from the previous and consider a spatial point of view in their dissimilarity index.
More recently, \citet{Zajac_2018} works on the bi-objective $K$ dissimilar vehicle routing problem ($k$d-VRP). The work considers two dissimilarity indices: the ``grid metric'', which treats spatial dissimilarity, as well as the ``edge metric'', which defines dissimilarity via shared arcs between different routes.

\section{The $K$ dissimilar paths problem\label{sec:3}}    
Let $(N,A)$ be a directed graph with $|N|=n$ nodes and $|A|=m$ arcs, where $s$ denotes a given source node and $t$ denotes a given target node, $s,t\in N$.
Let also $P$ denote the set of paths in $(N,A)$ from node $s$ to node $t$ and $K$ be a given positive integer.
The goal of the $K$ dissimilar paths problem from $s$ to $t$ is to find a set of $K$ paths in $P$, such that the paths in the set are fairly distributed throughout the network.
Considering dissimilarity based in the index $S_1$ introduced in Section~\ref{sec:1}, the problem can be defined as
$$\begin{array}{ll}
\max & \displaystyle\sum_{i=1}^{K-1}\sum_{j=i+1}^K D_1(p_i,p_j) / \binom{K}{2}\\
\text{subject to} & p_1,p_2,\ldots,p_K \in P
\end{array}$$
where
$$D_1(p,q)=1-S_1(p,q)=1-\dfrac{1}{2}
\left(\dfrac{L(p\cap q)}{L(p)}+\dfrac{L(p\cap q)}{L(q)}\right),$$
for any paths $p,q\in P$, and this is equivalent to minimizing the similarity of the set of $K$ paths,
$$\sum_{i=1}^{K-1}\sum_{j=i+1}^K \left(\dfrac{L(p_i\cap p_j)}{L(p_i)}+\dfrac{L(p_i\cap p_j)}{L(p_j)}\right). $$
This objective function is fractional and difficult to handle directly.
Therefore, we will consider simplifications of the problem.

The next sections introduce three formulations, simpler than this one, but which try to capture the main characteristic of this problem based on different assumptions.

\section{Minimization of the number of arc overlaps for each pair of paths\label{sec:4}}   
    
The length of the paths overlap is a common term to the several similarity indices described in Section \ref{sec:1}.
For now we focus on that length and, by doing so, the objective function becomes linear and simpler to handle.
Thus, the current goal is to find solutions for
\begin{equation}\label{prob1}
\begin{array}{ll}
\min & \displaystyle\sum_{i=1}^{K-1}\sum_{j=i+1}^K L(p_i\cap p_j)\\
\text{subject to} & p_1,p_2,\ldots,p_K \in P
\end{array}
\end{equation}

Given two paths $p,q\in P$, it is said that there is an overlap whenever there is an arc $(i,j)\in A$ that belongs to both $p$ and $q$, that is, if $(i,j)\in p$ and $(i,j)\in q$.
Thus, the number of overlaps between those paths is the number of arcs that appear in both, $OL(p,q)=|\{(i,j)\in A:(i,j)\in p \wedge (i,j)\in q\}|$, which coincides with $L(p\cap q)$.
The number of overlaps in a given set of $K$ paths is the total number of overlaps for each pair of paths, that is,
$$OL(\{p_1,p_2,\ldots,p_K\})=\sum_{i=1}^{K-1}\sum_{j=i+1}^K OL(p_i,p_j),$$
which is the objective function in problem (\ref{prob1}).
To illustrate these concepts, we recall the example in Figure \ref{fig:fig2new}.
The arc $(i,j)$ belongs to all four paths in Figure \ref{fig:fig2newa}. Therefore, the number of overlaps for those paths is 6.
In Figure \ref{fig:fig2newb} the arc $(i,j)$ appears in the paths in red and green, while the arc $(k,l)$ appears in the blue and orange paths. So, there are 2 arc overlaps in the 4 paths.
This latter solution is better than the first with respect to problem (\ref{prob1}).
    
The first formulation intends to model problem (\ref{prob1}) as an integer linear program.
Considering the decision variables $x_{ij}^k$ equal to 1 if the arc
$(i,j)$ lies in the $k$-th path from $s$ to $t$, or 0 otherwise, for any
$(i,j)\in A$ and $k=1,\ldots,K$, the problem is formulated as follows:
\begin{subequations}\label{eqf3}
\begin{eqnarray}
\min && f_1(x,z,v)=\sum_{(i,j)\in A}v_{ij}\label{eqf3of}\\
\mbox{subject to} && \sum_{j\in N:(i,j)\in A}x_{ij}^k-\sum_{j\in N:(j,i)\in A}x_{ji}^k=\left\{
    \begin{array}{rl}
     1 & i=s\\
     0 & i\ne s,t\\
    -1 & i=t
    \end{array}\right.,\ \ \ k=1,\ldots,K\label{eqf3c1}\\
	&& z_{ij}^{kl}\le x_{ij}^k,\ \ \ (i,j)\in A, \ \ \ k=1,\ldots,K-1, \ \ \ l=k+1,\ldots,K\label{eqf3c2}\\
	&& z_{ij}^{kl}\le x_{ij}^l,\ \ \ (i,j)\in A, \ \ \ k=1,\ldots,K-1, \ \ \ l=k+1,\ldots,K\label{eqf3c3}\\
	&& z_{ij}^{kl}\ge x_{ij}^k+x_{ij}^l-1,\ \ \ (i,j)\in A, \ \ \ k=1,\ldots,K-1, \ \ \ l=k+1,\ldots,K\label{eqf3c4}\\
	&& v_{ij}=\sum_{k=1}^{K-1}\sum_{l=k+1}^Kz_{ij}^{kl},\ \ \ (i,j)\in A\label{eqf3c5}\\
	&& x_{ij}^k\in\{0,1\}, \ z_{ij}^{kl}\in\{0,1\}, \ \ \ (i,j)\in A,\ \ \ k=1,\ldots,K,\ \ \ l=k+1,\ldots,K\label{eqf3c6}
\end{eqnarray}
\end{subequations}
This formulation may considerably big for problems of modest size, given that the variables $z$ are related with any pair of paths from $s$ to $t$.
Thus, it has $O(mK^2)$ variables and $O(mK^2+nK)$ constraints.
    
The  flow conservation constraints (\ref{eqf3c1}) ensure the existence of $K$ paths from node $s$ to node $t$;
the constraints (\ref{eqf3c2}) -- (\ref{eqf3c4}) guarantee that variables $z_{ij}^{kl}\in\{0,1\}$ are equal to 1 if and only if the arc $(i,j)$ is used both in the paths defined by
the variables $x_{ij}^k$ and $x_{ij}^l$, for $(i,j)\in A$, $k\in\{1,\ldots,K-1\}$, $l\in\{k+1,\ldots,K\}$. In fact:
\begin{itemize}
\item
If $x_{ij}^k=0$ or $x_{ij}^l=0$, then (\ref{eqf3c2}) or (\ref{eqf3c3}) imply that $z_{ij}^{kl}=0$, for any $(i,j)\in A$, $k\in\{1,\ldots,K-1\}$, $l\in\{k+1,\ldots,K\}$.
\item
If $x_{ij}^k=x_{ij}^l=1$, then both (\ref{eqf3c2}) and (\ref{eqf3c3}) imply that $z_{ij}^{kl}\le1$, whereas (\ref{eqf3c4}) imply that $z_{ij}^{kl}\ge1$.
Thus, $z_{ij}^{kl}=1$ for $(i,j)\in A$, $k\in\{1,\ldots,K-1\}$, $l\in\{k+1,\ldots,K\}$.
\end{itemize}
The constraints (\ref{eqf3c5}) state that the auxiliary variables $v_{ij}$ correspond to the number of paths that use arc $(i,j)\in A$ and constraints (\ref{eqf3c6}) define the variables. Observe that because the $x_{ij}^{k}$ are binary variables, then by  (\ref{eqf3c2}) -- (\ref{eqf3c4}), the variables $z_{ij}^{kl}$ are binary as well and, consequently, by (\ref{eqf3c5}), the variables $v_{ij}$ are implicitly defined as non-negative integers, for any $(i,j)\in A$, $k\in\{1,\ldots,K-1\}$, $l\in\{k+1,\ldots,K\}$.
The objective function (\ref{eqf3of}) corresponds to $OL(\{p_1,p_2,\ldots,p_K\})$, the total number of arcs that are shared by at least two paths from $s$ to $t$.

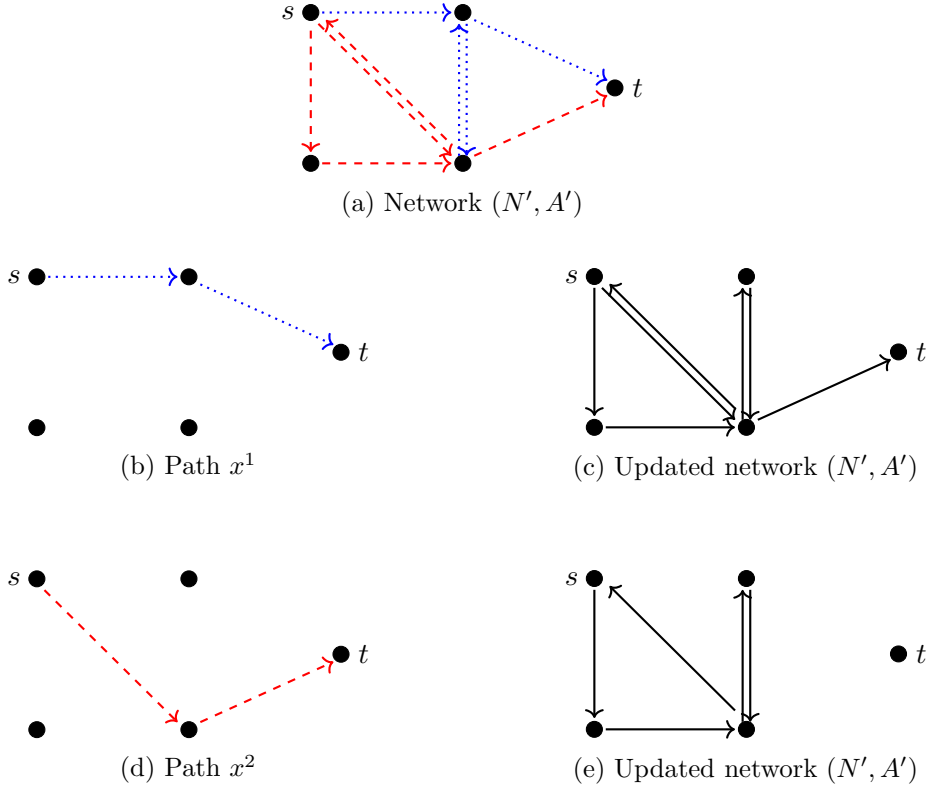
\begin{figure}[htb]
\centering
\begin{subfigure}[h]{0.45\textwidth}
\centering
\begin{tikzpicture}
\filldraw [black]
(0,0) circle [radius=3pt]
(0,2) circle [radius=3pt]
(2,0) circle [radius=3pt]
(2,2) circle [radius=3pt]
(4,1) circle [radius=3pt];
\draw (-0.3,2) node {$s$};
\draw (4.3,1) node {$t$};

\draw[->,red,dashed,thick] (0.15,0)--(1.85,0);
\draw[<-,red,dashed,thick] (0,0.15)--(0,1.85);
\draw[->,red,dashed,thick] (2.15,0.1)--(3.9,0.9);
\draw[->,red,dashed,thick] (0.1,1.85)--(1.85,0.1);
\draw[<-,red,dashed,thick] (0.2,1.9)--(1.85,0.25);
\draw[->,blue,dotted,thick] (0.15,2)--(1.85,2);
\draw[->,blue,dotted,thick] (1.95,0.1)--(1.95,1.85);
\draw[->,blue,dotted,thick] (2.05,1.85)--(2.05,.1);
\draw[->,blue,dotted,thick] (2.15,1.9)--(3.9,1.1);
\end{tikzpicture}
\caption{Network $(N',A')$\label{fig:fig1a}}
\end{subfigure}\vspace{0.5cm}
\begin{subfigure}[h]{0.45\textwidth}
\centering
\end{subfigure}

\begin{subfigure}[h]{0.45\textwidth}
\centering
\begin{tikzpicture}
\filldraw [black]
(0,0) circle [radius=3pt]
(0,2) circle [radius=3pt]
(2,0) circle [radius=3pt]
(2,2) circle [radius=3pt]
(4,1) circle [radius=3pt];
\draw (-0.3,2) node {$s$};
\draw (4.3,1) node {$t$};

\draw[->,blue,dotted,thick] (0.15,2)--(1.85,2);
\draw[->,blue,dotted,thick] (2.15,1.9)--(3.9,1.1);
\end{tikzpicture}
\caption{Path $x^1$\label{fig:fig1b}}
\end{subfigure}\vspace{0.5cm}
\begin{subfigure}[h]{0.45\textwidth}
\centering
\begin{tikzpicture}
	\filldraw [black]
(0,0) circle [radius=3pt]
(0,2) circle [radius=3pt]
(2,0) circle [radius=3pt]
(2,2) circle [radius=3pt]
(4,1) circle [radius=3pt];
\draw (-0.3,2) node {$s$};
\draw (4.3,1) node {$t$};
\draw[->,thick] (0.15,0)--(1.85,0);
\draw[<-,thick] (0,0.15)--(0,1.85);
\draw[->,thick] (1.95,0.1)--(1.95,1.85);
\draw[->,thick] (2.05,1.85)--(2.05,.1);
\draw[->,thick] (2.15,0.1)--(3.9,0.9);
\draw[->,thick] (0.1,1.85)--(1.85,0.1);
\draw[<-,thick] (0.2,1.9)--(1.85,0.25);
\end{tikzpicture}
\caption{Updated network $(N',A')$\label{fig:fig1c}}
\end{subfigure}\vspace{0.5cm}

\begin{subfigure}[h]{0.45\textwidth}
\centering
\begin{tikzpicture}
\filldraw [black]
(0,0) circle [radius=3pt]
(0,2) circle [radius=3pt]
(2,0) circle [radius=3pt]
(2,2) circle [radius=3pt]
(4,1) circle [radius=3pt];
\draw (-0.3,2) node {$s$};
\draw (4.3,1) node {$t$};

\draw[->,red,dashed,thick] (2.15,0.1)--(3.9,0.9);
\draw[->,red,dashed,thick] (0.1,1.85)--(1.85,0.1);
\end{tikzpicture}
\caption{Path $x^2$\label{fig:fig1d}}
\end{subfigure}
\begin{subfigure}[h]{0.45\textwidth}
\centering
\begin{tikzpicture}
\filldraw [black]
(0,0) circle [radius=3pt]
(0,2) circle [radius=3pt]
(2,0) circle [radius=3pt]
(2,2) circle [radius=3pt]
(4,1) circle [radius=3pt];
\draw (-0.3,2) node {$s$};
\draw (4.3,1) node {$t$};

\draw[->,thick] (0.15,0)--(1.85,0);
\draw[<-,thick] (0,0.15)--(0,1.85);
\draw[->,thick] (1.95,0.1)--(1.95,1.85);
\draw[->,thick] (2.05,1.85)--(2.05,.1);
\draw[<-,thick] (0.2,1.9)--(1.85,0.25);
\end{tikzpicture}
\caption{Updated network $(N',A')$\label{fig:fig1e}}
\end{subfigure}
\caption{Application of Algorithm \ref{alg:alg1} to a solution with loops\label{fig:fig1}}
\end{figure}
    
Formulation (\ref{eqf3}) may admit optimal solutions that contain subtours. For instance, the network depicted in Figure~\ref{fig:fig1a} shows a possible solution for finding $K=2$ paths from node $s$ to node $t$.
In this case, the solution has objective value 0, given that the paths represented in red and in blue do not have any arcs in common. However, each of those paths contains one subtour.
Nevertheless, if the problem is feasible, there always exists a loopless optimal solution to the problem, and this is easy to compute after a first optimal solution has been found.
Solutions with loops can be avoided by adding subtour elimination constraints to the formulation or by adding a term to the objective function that penalizes the utilization of arcs.
Another alternative is to apply a post processing algorithm that allows to extract one loopless optimal solution from a given optimal solution.
Algorithm~\ref{alg:alg1} outlines the procedure to obtain such a loopless solution, when given a solution $x$.
    
\begin{algorithm}[htb]
\DontPrintSemicolon
$N'\leftarrow \{i\in N: x_{ij}^k=1\mbox{ for some }j\in N \wedge k=1,\ldots,K\}\cup\{t\}$\;
$A'\leftarrow \{(i,j)\in A: x_{ij}^k=1\mbox{ for some }k=1,\ldots,K\}$\;
\lFor{$(i,j)\in A'$}{
   $u_{ij}\leftarrow \sum_{k=1}^K x_{ij}^k$
}
\For{$k=1,\ldots,K$}{
   \lFor{$(i,j)\in A$}{
      $\bar{x}_{ij}\leftarrow 0$
   }
   $\bar{x}^k\leftarrow$ shortest path from $s$ to $t$ in terms of the
   number of arcs in network $(N',A')$\;
   \For{$(i,j)\in A'$ such that $\bar{x}^k_{ij}=1$}{
      $u_{ij}\leftarrow u_{ij}-1$\;
      \lIf{$u_{ij}=0$}{
 Delete arc $(i,j)$ from $A'$
      }
   }
}
\caption{Algorithm for removing loops from a given solution $x$\label{alg:alg1}}
\end{algorithm}
    
In lines 1 and 2 of Algorithm~\ref{alg:alg1} the network corresponding
to the arcs in the given solution $x$ is built.
Line 3 assigns each arc $(i,j)\in A'$ with the number of times it appears
in $x$, $u_{ij}$. This value works like the arc $(i,j)$'s capacity.
Then, in each iteration of the loop in lines 4 to 9, one path is determined
and the capacity of an arc is updated every time it is used.
The loop runs $K$ times exactly, so that $K$ paths are found. Moreover,
the paths determined on line 5 are always loopless. These paths can be found by
means of breadth-first search, \citet{Cormen_2009}, and therefore the algorithm
runs in $O(Km)$ time. Also, the empirical tests reported later showed that its run time is very small when compared to solving any of the formulations here presented.
    
\begin{proposition}\label{prop1}
Let $(x,z,v)$ be a feasible solution for formulation (\ref{eqf3}) and $\bar{x}$ be
the corresponding vector output by Algorithm \ref{alg:alg1}.
Then, $\bar{x}$ defines $K$ loopless paths from $s$ to $t$.
\end{proposition}
\begin{proof}
The remarks above show that the result holds.
Shortly, every path defined by $\bar{x}^k$ is loopless, for any $k=1,\ldots,K$, because it is the solution of a shortest path problem with all costs positive (unitary).
Additionally, $x$ is the characteristic vector of $K$ paths from $s$ to $t$, given that it satisfies the constraints (\ref{eqf1c1}). Therefore, $\bar{x}^k$ defines $K$ loopless paths from $s$ to $t$.
\end{proof}
    
\begin{proposition}\label{prop4}
Let $(x,z,v)$ be an optimal solution for formulation (\ref{eqf3}).
Let $\bar{x}$ be the vector output by Algorithm \ref{alg:alg1} when applied to $x$,
and $\bar{z}\in\{0,1\}^{mK^2}$ and $\bar{v}\in\mathbb{N}_0^m$ be vectors which
satisfy the constraints (\ref{eqf3c2}) -- (\ref{eqf3c5}).
Then, $(\bar{x},\bar{z},\bar{v})$ is a loopless optimal solution for problem (\ref{eqf3}).
\end{proposition}
\begin{proof}
Suppose $(\bar{x},\bar{z},\bar{v})$ is obtained from $(x,z,v)$ according to Algorithm
\ref{alg:alg1} and the directions above. Then, $x_{ij}^k,\bar{x}_{ij}^k\in\{0,1\}$
and
\begin{equation}\label{eq3}
x_{ij}^k\ge\bar{x}_{ij}^k,\ \ (i,j)\in A,\ \  k = 1,\ldots,K.
\end{equation}
Two aspects need to be considered:
\begin{enumerate}
\item
According to Proposition \ref{prop1}, $\bar{x}$ corresponds to $K$ paths,
so the constraints (\ref{eqf3c1}) hold.
Moreover, $\bar{z}$ and $\bar{v}$ satisfy (\ref{eqf3c2}) and (\ref{eqf3c5}), therefore
$(\bar{x},\bar{z},\bar{v})$ is a feasible solution of (\ref{eqf3}).
\item
Because of condition (\ref{eq3}), it also holds that
$$z_{ij}^{kl}\ge\bar{z}_{ij}^{kl},\ \ (i,j)\in A,\ \ k=1,\ldots,K-1,\ \ l=k+1,\ldots,K,$$
$$v_{ij}\ge \bar{v}_{ij},\ \ (i,j)\in A,$$
therefore $f_1(x,z,v)\ge f_1(\bar{x},\bar{z},\bar{v})$, which shows that the new solution is optimal.
\end{enumerate}
It can then be concluded that $(\bar{x},\bar{z},\bar{v})$ is an optimal loopless
solution of (\ref{eqf3}).
\end{proof}
    
In the next section we describe some computational experiments performed on formulation (\ref{eqf3}).
Results will show that formulation (\ref{eqf3}) outputs solutions with very good dissimilarity scores, indicating that problem (\ref{prob1}) might be a good approach to the $K$ dissimilar paths problem, even though it neglects the role of the length of the paths in the dissimilarity of the $K$ paths.
However, the high run times required to solve even problems of modest dimension compromise its usability in practical applications.
These results are not at all unexpected, due to the combinatorial nature of the model, but they reinforce the need for recurring to alternative models, as mentioned in Section \ref{sec:1}.

\subsection{Computational experiments\label{sec:4.1}}
The purpose of the tests presented in the following is to study the behavior of formulation (\ref{eqf3}) in terms of the solutions it outputs, its run times and its integer programming gaps.

The code designated by \texttt{MAO}, standing for the implementation for minimizing the number of arcs overlaps for each pair of paths, formulation (\ref{eqf3}), was implemented in C. This code uses IBM ILOG CPLEX version 12.7 as the ILP solver.
The Algorithm \ref{alg:alg1}, also coded in C, was applied to the result of this implementation, in order to remove the loops from the obtained solutions.
The tests were carried out on a 64-bit PC with an Intel\textregistered Core\texttrademark\ i7-6700 Quad core at 3.40GHz with 64GB of RAM.

\begin{figure}[htb]
\centering
\begin{tikzpicture}[scale=1.2]
	\filldraw [black]
	(0,0) circle [radius=2pt]
	(1,0) circle [radius=2pt]
	(2,0) circle [radius=2pt]
	(3,0) circle [radius=2pt]
	(4,0) circle [radius=2pt]
	(0,1) circle [radius=2pt]
	(1,1) circle [radius=2pt]
	(2,1) circle [radius=2pt]
	(3,1) circle [radius=2pt]
	(4,1) circle [radius=2pt]
	(0,2) circle [radius=2pt]
	(1,2) circle [radius=2pt]
	(2,2) circle [radius=2pt]
	(3,2) circle [radius=2pt]
	(4,2) circle [radius=2pt]
	(0,3) circle [radius=2pt]
	(1,3) circle [radius=2pt]
	(2,3) circle [radius=2pt]
	(3,3) circle [radius=2pt]
	(4,3) circle [radius=2pt];
    	
	\draw[->,thick] (0.1,0)--(.9,0);
	\draw[->,thick] (1.1,0)--(1.9,0);
	\draw[->,thick] (3.1,0)--(3.9,0);
	\draw[->,thick] (0.1,1)--(.9,1);
	\draw[->,thick] (1.1,1)--(1.9,1);
	\draw[->,thick] (1.1,2)--(1.9,2);
	\draw[->,thick] (1.1,3)--(1.9,3);
	\draw[->,thick] (3.1,1)--(3.9,1);
	\draw[->,thick] (3.1,2)--(3.9,2);
	\draw[->,thick] (3.1,3)--(3.9,3);
	\draw[->,thick] (0,0.9)--(0,0.1);
	\draw[->,thick] (1,0.9)--(1,0.1);
	\draw[->,thick] (2,0.9)--(2,0.1);
	\draw[->,thick] (3,0.9)--(3,0.1);
	\draw[->,thick] (4,0.9)--(4,0.1);
	\draw[->,thick] (4,2.9)--(4,2.1);
	\draw[->,thick] (0,2.9)--(0,2.1);
	\draw[->,thick] (1,2.9)--(1,2.1);
	\draw[->,thick] (2,2.9)--(2,2.1);
	\draw[->,thick] (3,2.9)--(3,2.1);
	\draw[->,thick] (0.1,2)--(.9,2);	
	\draw[->,thick] (0.1,3)--(.9,3);
	\node[draw=none] (ellipsis1) at (2.5,0) {$\ldots$};

	\node[draw=none] (ellipsis1) at (2.5,1) {$\ldots$};	
	\node[draw=none] (ellipsis1) at (2.5,2) {$\ldots$};
	\node[draw=none] (ellipsis1) at (2.5,3) {$\ldots$};		
	\node[draw=none] (ellipsis1) at (0,1.6) {$\vdots$};
	\node[draw=none] (ellipsis1) at (1,1.6) {$\vdots$};	
	\node[draw=none] (ellipsis1) at (2,1.6) {$\vdots$};
	\node[draw=none] (ellipsis1) at (3,1.6) {$\vdots$};
	\node[draw=none] (ellipsis1) at (4,1.6) {$\vdots$};		
	\draw (-0.3,3) node {$s$};
	\draw (4.3,0) node {$t$};			
\end{tikzpicture}
\caption{Grid network\label{fig:Grid}}
\end{figure}
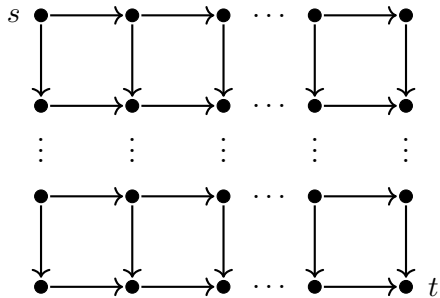

In the performed tests $K=3,4,\ldots,10$ paths were computed in directed networks.
Two types of instances were considered:
\begin{itemize}
\item
Random networks, denoted by $R_{n,m}$, with $n=100,300,500$ nodes, and $m=dn$ arcs, for average degrees $d=5,10$. A Hamiltonian cycle is created for all the nodes in the network, and afterwards the remaining arcs are generated randomly. This Hamiltonian cycle is directed, therefore the strong connectivity of the graph is not fully ensured.
There are no parallel arcs between pairs of nodes.
For each size of these networks, 30 instances were generated based on different seeds. The reported results are based only on the 22, out of the 30, that originated feasible problems.
\item
Grid networks, denoted by $G_{p,q}$, with $p=3,4,6,12$ rows, $q=6,12,36$ columns and $n=pq=36,144$ nodes, arranged in a planar grid, numbered consecutively from left to right and top to bottom -- as shown in Figure \ref{fig:Grid}.
Any pair of adjacent nodes is connected by an arc, thus $m=2pq-p-q=57,60,248,264$.
\end{itemize}
In both cases the source and the target nodes are $s=1$ and $t=n$, respectively, with no loss of generality.
It is worth noting that the considered grid networks are acyclic, therefore Algorithm~\ref{alg:alg1} was not applied for such instances.
Moreover, the run times of Algorithm \ref{alg:alg1} were small (at most 4\% of the total run times, that is, up to 50 milliseconds), therefore they are not included in the results reported in the following.

We begin by observing that in a non negligible number of the random instances, $K$ disjoint paths can be found. This happens for 274 of the 1056 instances. Table~\ref{tab:disj} exhibits the distribution of those instances.
   
\begin{table}[H]
\centering
\caption{Number of instances with $K$ disjoint paths in random networks\label{tab:disj}}
{\small
\begin{tabular}{@{}lccccccccl@{}}\toprule
	 &\multicolumn{8}{c}{$K$} \\\cmidrule(lr){2-9} 
	$R_{n,m}$& 3 & 4 & 5 & 6 & 7 & 8 & 9 & 10\\\midrule
    $R_{100,500}$  & 12 & 2 & 1 & 0  & 0  &  0  & 0 & 0  \\
    $R_{100,1000}$ & 22 & 18 & 14 & 12  & 7 & 4  & 2 & 1\\\midrule
    $R_{300,1500}$ & 14 & 7 & 1 & 0 & 0 & 0 & 0 & 0\\
    $R_{300,3000}$ & 21 & 18 & 13 & 9 & 7 & 4 & 2 & 1\\\midrule
    $R_{500,2500}$ & 14 & 4 & 3 & 1 & 1 & 0 & 0 & 0\\
    $R_{500,5000}$ & 17 & 14 & 12 & 8 & 5 & 3 & 0 & 0\\\bottomrule
\end{tabular}}
\end{table}
   
Unless otherwise stated, hereafter only the instances with no disjoint solutions are considered. There are two reasons for this:
\begin{itemize}
\item
the main goal of this work is to find good methods to obtain dissimilar paths in networks where finding those paths is not easy (finding disjoint paths can be formulated as a minimum cost flow problem \cite{Suurballe_1974,Suurballe_1984});
\item
such instances are not evenly distributed throughout the set of instances, thus their presence in different numbers in each group could skew the results.
\end{itemize}
The drawback of this decision is that each set $R_{n,m}$ now has a smaller and different number of instances. Table \ref{tab:rem} indicates the final number for each type of instances.
It should be remarked that there are no disjoint solutions in the grid instances.

\begin{table}[H]
\centering
\caption{Final number of random network instances\label{tab:rem}}
{\small
\begin{tabular}{@{}lccccccccl@{}}\toprule
		 &\multicolumn{8}{c}{$K$} \\\cmidrule(lr){2-9} 
	$R_{n,m}$	& 3 & 4 & 5 & 6 & 7 & 8 & 9 & 10\\\midrule
     $R_{100,500}$ & 10 & 20 & 21 & 22 & 22 & 22 & 22 & 22\\
    $R_{100,1000}$ & 0 & 4 & 8 & 10 & 15 & 18 & 20 & 21 \\\midrule
    $R_{300,1500}$ & 8 & 15 & 21 & 22 & 22 & 22 & 22 & 22\\
    $R_{300,3000}$ & 1 & 4 & 9 & 13 & 15 & 18 & 20 & 21 \\\midrule
    $R_{500,2500}$ & 8 & 18 & 19 & 21 & 21 & 22 & 22 & 22\\
    $R_{500,5000}$ & 5 & 8 & 10 & 14 & 17 & 19 & 22 & 22\\\bottomrule
\end{tabular}}
\end{table}

An upper bound of 300 seconds was set for the elapsed time when running the solver.
The results presented in the following for random networks are average values for solving the instances with the same characteristics except for a seed, and with no disjoint solutions.
The results shown for grid networks are based on a single instance, again with no disjoint solutions.

Approximately 84 \% of the random instances were solved to optimality (corresponding to 659 out of 782 instances) and 44 \% of the grid instances (corresponding to 14 out of 32 instances).

\begin{table}[htb]
\centering
\caption{Number of instances solved to optimality by \texttt{MAO} (\%)\label{tab:mao_opt}}
{\small
\begin{tabular}{@{}lccccccccl@{}}\toprule
		 &\multicolumn{8}{c}{$K$} \\\cmidrule(lr){2-9} 
	$R_{n,m}$	& 3 & 4 & 5 & 6 & 7 & 8 & 9 & 10\\\midrule
    $R_{100,500}$ & 100 & 100 & 100 & \colorbox{lightgray}{96} & \colorbox{lightgray}{77} & \colorbox{lightgray}{50} & \colorbox{lightgray}{27} & \colorbox{lightgray}{23} \\
    $R_{100,1000}$   &  -- & 100 & 100 & 100 & 100 & 100 & 100 & \colorbox{lightgray}{95} \\\midrule
    $R_{300,1500}$   & 100  & 100 & 100   & 100  & \colorbox{lightgray}{95} & \colorbox{lightgray}{80} & \colorbox{lightgray}{59} & \colorbox{lightgray}{41}\\
    $R_{300,3000}$ & 100 & 100 & 100 & 100  & 100 & 100 & 100 & \colorbox{lightgray}{86} \\\midrule
    $R_{500,2500}$ & 100 & 100 & 100 & 100 & \colorbox{lightgray}{86} & \colorbox{lightgray}{73} & \colorbox{lightgray}{55} & \colorbox{lightgray}{32} \\
    $R_{500,5000}$ & 100 & 100 & 100 & 100  & 100 & 100  & \colorbox{lightgray}{86} & \colorbox{lightgray}{82}\\\bottomrule
\end{tabular}\hfill
\begin{tabular}{@{}lccccccccl@{}}\toprule
		 &\multicolumn{8}{c}{$K$} \\\cmidrule(lr){2-9} 
	$G_{p,q}$ & 3 & 4 & 5 & 6 & 7 & 8 & 9 & 10\\\midrule
    $G_{3,12}$ & 100  & 100    & 100   & 100  & \colorbox{lightgray}{0} & \colorbox{lightgray}{0} & \colorbox{lightgray}{0} & \colorbox{lightgray}{0}\\
	$G_{4,36}$ & 100 & 100 & \colorbox{lightgray}{0} &  \colorbox{lightgray}{0} & \colorbox{lightgray}{0} &  \colorbox{lightgray}{0} & \colorbox{lightgray}{0} &  \colorbox{lightgray}{0}  \\
	$G_{6,6}$ & 100 & 100 & 100 & 100  & \colorbox{lightgray}{0} & \colorbox{lightgray}{0} & \colorbox{lightgray}{0} & \colorbox{lightgray}{0}\\
	$G_{12,12}$ & 100 & 100 & 100 & 100 & \colorbox{lightgray}{0} & \colorbox{lightgray}{0}  &  \colorbox{lightgray}{0} & \colorbox{lightgray}{0}\\\bottomrule
\end{tabular}}
\end{table}

Due to the combinatorial nature of the model, it would be expected that the results of its application depended heavily on the size of the network. However, according to Tables~\ref{tab:mao_opt}\footnote{There is only one grid instance of each kind, therefore, the listed values are either 0\% or 100\%.} and~\ref{tab:mao_cpu} many of the instances that were not solved within the time limit are actually associated to the smaller networks.
In fact, the results behaved as expected only for the first values of $K\leq 4$, for both the random and the grid networks. As $K$ increases, other features seem to have a more decisive influence. A finer analysis of the results allows to enumerate three possible explanations for that situation:
the layout of the network;
the sparsity of the network; and 
the relation between the number of paths $K$ and the size of the network.

The differences in the layout of the network account for the fact that grid instances are, in general, harder to solve than random instances, even though they are much smaller (recall that $36\leq n\le144$ and $57\leq m\le264$ for the grid instances, whereas $100\leq n\le500$ and $500\leq m\le5000$ for the random instances). Furthermore, even among grid networks, there are differences related to the layout, as the problem seems to be more difficult to solve in rectangular grids rather than in square grids -- see Table \ref{tab:mao_cpu} and Figure \ref{fig:mao_cpu}.

\begin{table}[htb]
\centering
\caption{Run times of \texttt{MAO} (seconds)\label{tab:mao_cpu}}
{\small
\begin{tabular}{@{}lccccccccl@{}}\toprule
		 &\multicolumn{8}{c}{$K$} \\\cmidrule(lr){2-9} 
	$R_{n,m}$	& 3 & 4 & 5 & 6 & 7 & 8 & 9 & 10\\\cmidrule(lr){1-9} 
$R_{100,500}$   & 0.124  & 0.490    & 1.135   & \colorbox{lightgray}{32.149} & \colorbox{lightgray}{76.488} & \colorbox{lightgray}{175.405}  & \colorbox{lightgray}{222.937}  & \colorbox{lightgray}{249.626}  \\
    $R_{100,1000}$   & -- & 0.854    & 3.245   & 2.910  & 5.067 & 10.116  & 18.692 & \colorbox{lightgray}{48.099}  \\\cmidrule(lr){1-9} 
    $R_{300,1500}$   & 0.391  & 1.611    & 2.637  & 10.127  & \colorbox{lightgray}{37.499} & \colorbox{lightgray}{111.854} & \colorbox{lightgray}{174.106}  & \colorbox{lightgray}{212.429}  \\
    $R_{300,3000}$   & 0.745  & 1.934    & 5.063   & 11.508  & 23.537 & 35.912  & 63.022  & \colorbox{lightgray}{132.342}  \\\cmidrule(lr){1-9} 
    $R_{500,2500}$   & 0.793  & 2.179    & 4.579  & 12.781  & \colorbox{lightgray}{66.327}  & \colorbox{lightgray}{121.559}   & \colorbox{lightgray}{196.307}  & \colorbox{lightgray}{238.856}  \\
    $R_{500,5000}$   & 1.888  & 1.491   & 7.253   & 17.309  & 34.013 & 57.627  & \colorbox{lightgray}{118.488}  & \colorbox{lightgray}{196.436}  \\\midrule
	Average & 0.788  & 1.426 & 3.985   & 14.464  & 40.488 & 85.412  & 132.258  & 179.631 & \fbox{57.306}\\\bottomrule
\end{tabular}\hfill
 \begin{tabular}{@{}lccccccccl@{}}\toprule
		 &\multicolumn{8}{c}{$K$} \\\cmidrule(lr){2-9} 
	$G_{p,q}$ & 3 & 4 & 5 & 6 & 7 & 8 & 9 & 10\\\cmidrule(lr){1-9} 
$G_{3,12}$   &  0.072  & 2.665    & 10.150 & 52.917  & \colorbox{lightgray}{300} & \colorbox{lightgray}{300}   & \colorbox{lightgray}{300}  & \colorbox{lightgray}{300}  \\
    $G_{4,36}$  & 0.218  & 0.684    &  \colorbox{lightgray}{300}   & \colorbox{lightgray}{300}  &  \colorbox{lightgray}{300} &  \colorbox{lightgray}{300}  &  \colorbox{lightgray}{300} &  \colorbox{lightgray}{300}  \\
    $G_{6,6}$   & 0.096  & 0.265 & 2.776 & 11.783  & \colorbox{lightgray}{300}  & \colorbox{lightgray}{300}   & \colorbox{lightgray}{300}  & \colorbox{lightgray}{300}  \\
    $G_{12,12}$  & 0.240  & 0.464    & 5.212   & 19.772  &  \colorbox{lightgray}{300} & \colorbox{lightgray}{300}  &  \colorbox{lightgray}{300} & \colorbox{lightgray}{300}  \\\midrule
	Average & 0.156  & 1.019  & 79.534 & 96.118 & 300   & 300   & 300   & 300 & \fbox{172.103}\\\bottomrule
\end{tabular}}
\end{table}

Moreover, the effect of the sparseness of the  network is particularly clear in the results of the random instances.
According to Figure \ref{fig:mao_cpu}, the run times decrease as the density of the network\footnote{The density of a network is defined as $m/(n(n-1))$.} increases when $K\ge5$, in all sets of instances but one, and that this factor overcomes the one of the size of the network.
The only exception is the smallest of the instances. In spite of its density, second only to $R_{100,1000}$, $R_{100,500}$ instances have proven to be, in average, the hardest of the random instances to solve. Another interesting conclusion is that the impact of sparseness in grid networks is not so significant as it is in random networks. This finding becomes evident when comparing the run times for the $G_{4,36}$ and $G_{12,12}$ networks, which have the lowest of the densities of the grid networks.

As for the third of the reasons presented: the proportion between the value of $K$ and the size of the network, finding $K$ dissimilar paths in a given network becomes more difficult as the value of $K$ grows and this difficulty is increased in very small networks. This is a plausible reason why the results obtained for the $R_{100,500}$ cases do not follow the pattern of the remaining random instances.
Moreover, it is also expected that this factor has an impact on the results for the grid instances. However, results for more grid networks would be recommended for a sound conclusion.

\begin{figure}[htp]
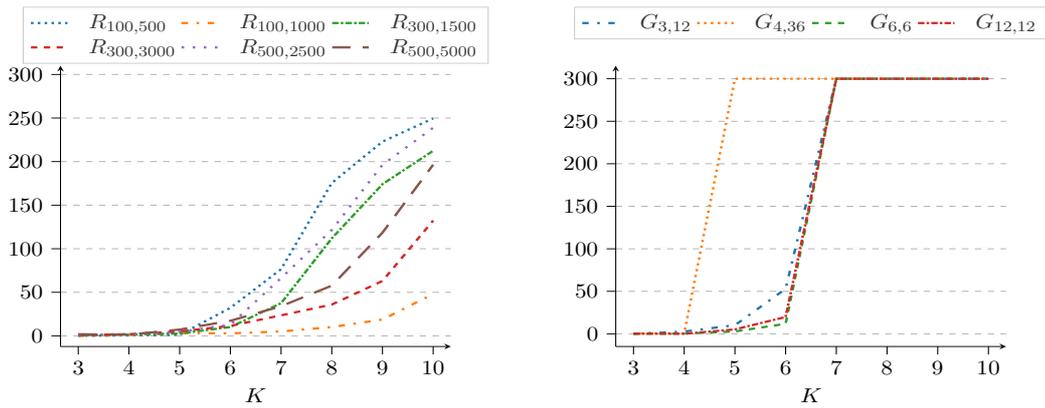

\centering
\begin{subfigure}[h]{0.45\textwidth}
\centering
	\input{Images/mao_cpu_rn.tex}	
\end{subfigure}
\begin{subfigure}[h]{0.45\textwidth}
\centering
	\input{Images/mao_cpu_gr.tex}
\end{subfigure}	
\caption{Run times of \texttt{MAO} (seconds)\label{fig:mao_cpu}}
\end{figure}
 
For the sake of completeness, the lower bounds obtained by solving the linear programming relaxation of the model are also presented. Table~\ref{tab:mao_IG} presents the average integer programming gaps, as well as the run times for solving the corresponding linear programming relaxations, for each group of instances.
The integer programming gaps are computed as
$100(f_1^*-f_{LR_1}^*)/|f_1^*|\ \%$,
where $f_1^*$ denotes the optimum value of (\ref{eqf3}) and $f_{LR_1}^*$ denotes the optimum value of its linear programming relaxation.
Whenever the optimum value $f_1^*$ is unknown, the best known integer is used to compute these gaps.
The gaps associated with the grid instances were all 100\%, therefore no such table is presented.
   
\begin{table}[H]
\centering
\caption{Average integer programming gaps of \texttt{MAO$_L$} (\%)\label{tab:mao_IG}}
{\small
\begin{tabular}{@{}lccccccccl@{}}\toprule
	&\multicolumn{8}{c}{$K$} \\\cmidrule(lr){2-9} 
	$R_{n,m}$ & 3 & 4 & 5 & 6 & 7 & 8 & 9 & 10\\\midrule
	$R_{100,500}$ & 43 & 73 & 75 & 77 & 78 & 78 & 79  & 79\\
	$R_{100,1000}$ & -- & 100 & 100 & 100 & 100 & 100 & 100  & 100\\\midrule
	$R_{300,1500}$ & 75 & 89 & 92 & 93 & 93  & 93 & 93 & 93\\
	$R_{300,3000}$ & 100 & 100 & 100 & 100 & 100 & 100 & 100 & 100\\\midrule
	$R_{500,2500}$ & 91  & 96 & 96 & 97 & 97 & 97 & 97 & 97\\
	$R_{500,5000}$ & 60 & 75 & 80 & 86 & 87 & 90 & 91 & 91\\\bottomrule
\end{tabular}}
\end{table}

As it can be observed from Table \ref{tab:mao_IG}, the linear programming relaxation of formulation (\ref{eqf3}) produces very weak lower bounds.
This often happens in models that use the same type of linking constraints (\ref{eqf3c4}) used in this formulation.

\section{Minimization of the number of repeated arcs\label{sec:5}}  
While taking into account all the overlaps in the pairs of paths in the solutions, the formulation (\ref{eqf3}) is not easy to handle from a practical point of view, as shown by the experiments reported in Section~\ref{sec:4.1}.
In the following an alternative, and simpler, approach to problem (\ref{prob1}) is introduced.

A given arc $(i,j)\in A$ is said to be repeated in a set of $K$ paths if it belongs to more than one of them.
A way to ensure that the $K$ paths are sufficiently different from each other is to consider a relaxed version of the $K$ disjoint paths problem, where instead of forbidding the occurrence of repeated arcs, the number of arcs in those conditions is minimized.
This concept was explored in the example illustrated in Figure~\ref{fig:fig2new}, given in the introductory section.
Because this approach simply privileges the number of repeated paths, it favors the solution depicted in Figure~\ref{fig:fig2newa}, rather than the one in Figure~\ref{fig:fig2newb}.

In order to model such a problem, let us consider, as before, the decision variables $x_{ij}^k$ equal to 1 if the arc $(i,j)$ lies in the $k$-th path from $s$ to $t$, or 0 otherwise, for any $(i,j)\in A$ and $k=1,\ldots,K$.
The problem of minimizing the number of repeated arcs can then be formulated as follows:
\begin{subequations}\label{eqf1}
\begin{eqnarray}
     \min && f_2(x,y)=\sum_{(i,j)\in A}y_{ij}\label{eqf1of}\\
\mbox{subject to} && \sum_{j\in N:(i,j)\in A}x_{ij}^k-\sum_{j\in N:(j,i)\in A}x_{ji}^k=\left\{
     \begin{array}{rl}
      1 & i=s\\
      0 & i\ne s,t\\
     -1 & i=t
     \end{array}\right.,\ \ \ k=1,\ldots,K\label{eqf1c1}\\
  && y_{ij}\le\sum_{k=1}^Kx_{ij}^k,\ \ \ (i,j)\in A\label{eqf1c2}\\
  && (K-1)y_{ij}\ge\sum_{k=1}^Kx_{ij}^k-1,\ \ \ (i,j)\in A\label{eqf1c3}\\
  && x_{ij}^k\in\{0,1\}, \ y_{ij} \in\{0,1\}, \ \ \ (i,j)\in A,\ \ \ k=1,\ldots,K
\end{eqnarray}
\end{subequations}
This formulation has $O(Km)$ binary variables and $O(Kn+m)$ constraints.
The constraints (\ref{eqf1c1}) are flow conservation constraints that model
$K$ paths from node $s$ to node $t$.
The constraints (\ref{eqf1c2}) and (\ref{eqf1c3}) relate the $x$ and the
$y$ variables, in a way that $y_{ij}$ is 1 if and only if the arc $(i,j)$
is used in more than one path, that is, if this arc is repeated, and 0 otherwise.
In fact, given the arc $(i,j)\in A$:
\begin{enumerate}
\item
If $x_{ij}^k=0$ for every $k=1,\ldots,K$, then by (\ref{eqf1c2}) we have $y_{ij}=0$,
whereas (\ref{eqf1c3}) has no implications on the value of $y_{ij}$.
\item
If $x_{ij}^k=1$ for exactly one $k\in\{1,\ldots,K\}$, then neither (\ref{eqf1c2}) nor (\ref{eqf1c3}) have implications on the value of $y_{ij}$.
\item
If $x_{ij}^k=1$ for more than one $k\in\{1,\ldots,K\}$, then (\ref{eqf1c2}) has
no implications on the value of $y_{ij}$, whereas by (\ref{eqf1c3}) we have
$y_{ij}=1$.
\end{enumerate}
When in situation 2., that is, if only one variable $x_{ij}^k$ has value 1, the value of $y_{ij}$ can be arbitrary.
However, the objective function minimizes the sum of all these variables and this minimization is achieved if the arbitrary $y_{ij}$'s are equal to 0, as intended.
Therefore, the objective function counts the number of repeated arcs, that is the number of arcs that are used more than once.

The same reasoning can be used to prove that the set of constraints (\ref{eqf1c2}) can be dropped.
In fact, because the goal of this formulation is to minimize $\sum_Ay_{ij}$, the values of $y_{ij}$ are 0 by default.

Just like for the formulation presented in the previous section, an optimal solution of formulation (\ref{eqf1}) may contain loops, as long as they do not include any arc that is common to several paths.
However, given any such optimal solution, Algorithm~\ref{alg:alg1} can be applied in order to remove its loops.
The arguments for proving Proposition \ref{prop2} are similar to those used for Proposition~\ref{prop4}, therefore its proof is omitted.

\begin{proposition}\label{prop2}
Let $(x,y)$ be an optimal solution for problem (\ref{eqf1}).
Let $\bar{x}\in\{0,1\}^{Km}$ be the corresponding vector output by Algorithm \ref{alg:alg1} when applied to $x$, and $\bar{y}\in\{0,1\}^m$ be such that the constraints (\ref{eqf1c2}) and (\ref{eqf1c3}) are satisfied.
Then, $(\bar{x},\bar{y})$ is a loopless optimal solution for problem (\ref{eqf1}).
\end{proposition}

\subsection{Computational experiments\label{sec:5.1}}
In the following, formulation (\ref{eqf1}), which minimizes the number of repeated arcs, is analyzed empirically.
The code that implements this formulation is designated by \texttt{MRA}.
It was written in C and uses IBM ILOG CPLEX version 12.7 to solve the integer programs.
The variant of the same formulation obtained by removing the constraints (\ref{eqf1c2}) was also implemented. Note that both models are valid formulations of the problem, the later being a weaker (in terms of its linear programming bound)  but smaller variant of the first. Because the differences in the run times obtained with both variants were not significant, only the original one is presented below.

The experimental setup was as described in Section~\ref{sec:4.1}.
Algorithm~\ref{alg:alg1} was applied to the results of the code \texttt{MRA} to remove the loops from the obtained solutions in the random instances.

\begin{table}[htb]
\centering
\caption{Number of instances solved to optimality by \texttt{MRA} (\%)\label{tab:MRAopt}}
\small{
\begin{tabular}{@{}lccccccccl@{}}\toprule
   			&\multicolumn{8}{c}{$K$} \\\cmidrule(lr){2-9}
  $G_{p,q}$ & 3 & 4 & 5 & 6 & 7 & 8 & 9 & 10\\\midrule
 $G_{3,12}$ & 100 & 100 & 100 & 100 & 100 & 100 & 100 & 100\\
 $G_{4,36}$ & 100 & 100 & 100 & 100 & 100 & \colorbox{lightgray}{0} & \colorbox{lightgray}{0} & \colorbox{lightgray}{0}\\
  $G_{6,6}$ & 100 & 100 & 100 & 100 & 100 & 100 & 100 & 100\\
$G_{12,12}$ & 100 & 100 & 100 & 100 & 100 & \colorbox{lightgray}{0} & \colorbox{lightgray}{0} & \colorbox{lightgray}{0}\\\bottomrule
\end{tabular}}
\end{table}

\texttt{MRA} was able to solve to optimality all the random network instances, within the 300 seconds time window.
Additionally, Table~\ref{tab:MRAopt} shows the same number but for grids.
It was possible to solve 81\% of these instances, corresponding to 26 out of the 32 instances.
The interruptions after 300 seconds occurred in problems of finding 8 or more paths in the 144 node grids.
Finally, it is worth noting that even though formulations (\ref{eqf3}) and (\ref{eqf1}) model different problems it is still of interest to compare them, since they both are relaxations of the $K$ dissimilar paths problem, and that far more instances were solved by \texttt{MRA} than by \texttt{MAO}.

\begin{table}[htb]
\centering
\caption{Run times of \texttt{MRA} (seconds)\label{tab:MRAcpu0}}
\small{
\begin{tabular}{@{}lccccccccl@{}}\toprule
       &\multicolumn{8}{c}{$K$} \\\cmidrule(lr){2-9}
     $R_{n,m}$ & 3 & 4 & 5 & 6 & 7 & 8 & 9 & 10\\\cmidrule(lr){1-9}
 $R_{100,500}$ & 0.051 & 0.061 & 0.098 & 0.173 & 0.254 & 0.657 & 0.981 & 2.092\\
$R_{100,1000}$ & -- & 0.108 & 0.145 & 0.189 & 0.206 & 0.351 & 0.373 & 0.385\\\cmidrule(lr){1-9}
$R_{300,1500}$ & 0.138 & 0.204 & 0.279 & 0.312 & 0.432 & 0.539 & 1.057 & 2.144\\
$R_{300,3000}$ & 0.280 & 0.347 & 0.411 & 0.485 & 0.614 & 0.723 & 0.855 & 0.949\\\cmidrule(lr){1-9}
$R_{500,2500}$ & 0.256 & 0.323 & 0.457 & 0.557 & 0.723 & 0.964 & 1.458 & 1.586\\
$R_{500,5000}$ & 0.394 & 0.542 & 0.770 & 0.944 & 1.074 & 1.271 & 1.488 & 1.836\\\midrule
       Average & 0.223  & 0.261 & 0.360 & 0.443  & 0.550 & 0.750  & 1.053  & 1.498 & \fbox{0.640}\\\bottomrule
\end{tabular}\hfill
\begin{tabular}{@{}lccccccccl@{}}\toprule
    &\multicolumn{8}{c}{$K$} \\\cmidrule(lr){2-9}
  $G_{p,q}$ & 3 & 4 & 5 & 6 & 7 & 8 & 9 & 10\\\cmidrule(lr){1-9}
 $G_{3,12}$ & 0.054 & 0.154 & 0.455 & 0.859 &   0.696 &   0.902 &   1.078 & 1.235\\
 $G_{4,36}$ & 0.103 & 0.107 & 1.833 & 5.047 & 221.466 & \colorbox{lightgray}{300} & \colorbox{lightgray}{300} & \colorbox{lightgray}{300}\\
  $G_{6,6}$ & 0.019 & 0.093 & 0.327 & 1.200 &  11.392 &  3.075 &   2.176 & 1.441\\
$G_{12,12}$ & 0.166 & 0.239 & 1.070 & 4.359 &  47.101 & \colorbox{lightgray}{300} & \colorbox{lightgray}{300} & \colorbox{lightgray}{300}\\\midrule
    Average & 0.085  & 0.148 & 0.921 & 2.866  & 70.163 & 150.994  & 150.813  & 150.669 & \fbox{65.832}\\\bottomrule
\end{tabular}}
\end{table}

The run times of code \texttt{MRA} are summarized in Table~\ref{tab:MRAcpu0} and depicted in Figure~\ref{fig:mra_cpu}.
The results were particularly sensitive to the number of paths. However, other features have negative repercussions as well.

First, the results depend greatly on the layout of the network. In fact, unsolved instances were only found solo in the case of grids. Furthermore, the magnitude of the run times associated to the solved grid instances is several times higher than the run times for solving the random instances (up to 221.5 seconds in the former case and to 2.2 seconds in the latter).
Another evidence of the difficulty associated to solving the grid instances, is the fact that between 4 and 256 arcs have to be repeated in these instances, against between 1 and 24 for random networks.
In this way, both formulations (\ref{eqf3}) and (\ref{eqf1}) are very sensitive to the layout of the network and both work worse in the case of grid networks.
In the case of the grid networks, the run times of \texttt{MRA} vary both with the shape of the grid and with $K$. Further conclusions would require more exhaustive tests.

Second, as made evident by analyzing Table~\ref{tab:MRAcpu0} and Figure~\ref{fig:mra_cpu}, the patter behaviour of the \texttt{MRA} run times for random networks changes for $K\ge 7$. Prior to that value, the run times vary with the size of the networks; afterwards other factors become dominant. This type of behavior was identified in Section \ref{sec:4.1}, when analyzing \texttt{MAO} results. Then, the density of the network and the proportion between the value of $K$ and the size of the network were identified as determinant factors  of the observed deviation. The corresponding \texttt{MRA} results, seem to be affected in a similar way -- the density of the network smooths the growth of the run times, whereas the increase of $K$ causes abrupt increases from a certain threshold for the $R_{100,500}$ and $R_{300,1500}$ instances.

\begin{figure}[htb]
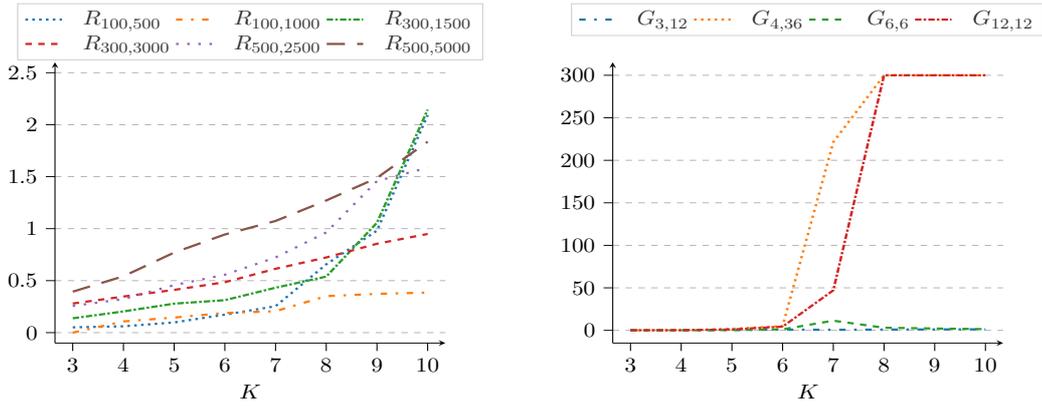

\centering
\begin{subfigure}[h]{0.45\textwidth}
	\centering
	\input{Images/mra_cpu_rn.tex}	
\end{subfigure}
\begin{subfigure}[h]{0.45\textwidth}
	\centering
	\input{Images/mra_cpu_gr.tex}
\end{subfigure}	
\caption{Run times of \texttt{MRA} (seconds)\label{fig:mra_cpu}}
\end{figure}

\begin{table}[htb]
\centering
\caption{Average integer programming gaps of \texttt{MRA$_L$} (\%)\label{tab:MRAgap}}
\small{
\begin{tabular}{@{}lccccccccl@{}}\toprule
       &\multicolumn{8}{c}{$K$} \\\cmidrule(lr){2-9}
     $R_{n,m}$ & 3 & 4 & 5 & 6 & 7 & 8 & 9 & 10\\\midrule
 $R_{100,500}$ & 10 & 40 & 36 & 38 & 36 & 42 & 43 & 43\\
$R_{100,1000}$ & -- & 12 & 23 & 25 & 40 & 46 & 48 & 49\\\midrule
$R_{300,1500}$ & 14 & 33 & 45 & 41 & 39 & 39 & 47 & 47\\
$R_{300,3000}$ &  2 & 11 & 25 & 35 & 38 & 44 & 49 & 47\\\midrule
$R_{500,2500}$ & 17 & 42 & 35 & 37 & 35 & 42 & 45 & 45\\
$R_{500,5000}$ &  7 & 14 & 17 & 30 & 38 & 41 & 48 & 42\\\bottomrule
\end{tabular}\hfill
\begin{tabular}{@{}lccccccccl@{}}\toprule
    &\multicolumn{8}{c}{$K$} \\\cmidrule(lr){2-9}
  $G_{p,q}$ &  3 &  4 & 5 & 6 & 7 & 8 & 9 & 10\\\midrule
 $G_{3,12}$ & 50 & 61 & 50 & 46 & 41 & 37 & 35 & 32\\
 $G_{4,36}$ & 50 & 67 & 72 & 59 & 51 & 46 & 41 & 37\\
  $G_{6,6}$ & 50 & 67 & 67 & 70 & 70 & 66 & 60 & 56\\
$G_{12,12}$ & 50 & 67 & 67 & 70 & 70 & 71 & 71 & 72\\\bottomrule
\end{tabular}}
\end{table}

Finally, Table \ref{tab:MRAgap} presents the average integer programming gaps determined by the lower bound produced by the linear relaxation of the formulation (\ref{eqf1}). These gaps are computed as explained in Section~\ref{sec:4.1}.
In the random instances the gap values are at most 49\%.
The gaps are even bigger for the grid networks, between 32\% ad 72\%.

\section{Minimization of the number of arc repetitions\label{sec:6}}  
The goal of formulation (\ref{eqf1}), presented in the previous section, is to minimize the number of arcs which are repeated in the solutions.
The undesired effect of this single objective may be that few arcs appear in many different paths. This situation is illustrated in Figure~\ref{fig:fig2new}.
Another example is depicted by the two sets of $K=3$ paths in Figure~\ref{fig:fig3new}. Both the solutions in Figure~\ref{fig:fig3new} have a single repeated arc. However, the paths in the solution in Figure~\ref{fig:fig3newb} are more dissimilar than the paths in Figure~\ref{fig:fig3newa}, the reason being that in the first case the arc $(i,j)$, which is repeated, appears only twice, while it appears in all the three paths in the latter case.
In the following two approaches are presented which intend to model a more complete understanding of how dissimilar paths should look like.

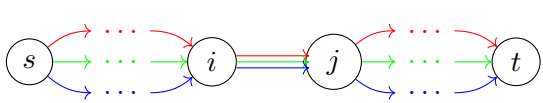
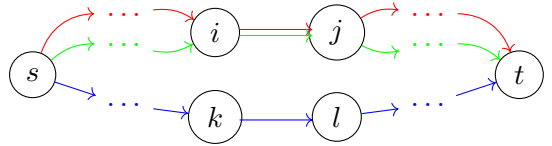
\begin{figure}[htb]
\centering
\begin{subfigure}[h]{0.45\textwidth}
\centering
\begin{tikzpicture}
     [scale=.8,every node/.style={circle,text=black,draw,minimum size=0.5cm}]
     \node (a) at (0,0) [] {$s$};
     \node (b) at (3,0) [] {$i$};
     \node (c) at (5,0) [] {$j$};
     \node (d) at (8,0) [] {$t$};

     \node[draw=none,red] (ellipsis1) at (1.5,0.5) {$\ldots$};
     \node[draw=none,green] (ellipsis2) at (1.5,0) {$\ldots$};
     \node[draw=none,blue] (ellipsis3) at (1.5,-0.5) {$\ldots$};

     \node[draw=none,red] (ellipsis4) at (6.5,0.5) {$\ldots$};
     \node[draw=none,green] (ellipsis5) at (6.5,0) {$\ldots$};
     \node[draw=none,blue] (ellipsis6) at (6.5,-0.5) {$\ldots$};

     \draw (c) edge[->,bend left=20,red] (ellipsis4);
     \draw (ellipsis4) edge[->,bend left=20,red] (d);
     \draw (c) edge[->,green] (ellipsis5);
     \draw (ellipsis5) edge[->,green] (d);
     \draw (c) edge[->,bend right=20,blue] (ellipsis6);
     \draw (ellipsis6) edge[->,bend right=20,blue] (d);
     \draw (a) edge[->,bend left=20,red] (ellipsis1);
     \draw (ellipsis1) edge[->,bend left=20,red] (b);
     \draw (a) edge[->,bend right=20,blue] (ellipsis3);
     \draw (ellipsis3) edge[->,bend right=20,blue] (b);
     \draw (a) edge[->,green] (ellipsis2);
     \draw (ellipsis2) edge[->,green] (b);

     \draw[->,red] (3.4,0.1)--(4.6,0.1);
     \draw[->,green] (3.4,0)--(4.6,0);
     \draw[->,blue] (3.4,-0.1)--(4.6,-0.1);
\end{tikzpicture}
\caption{Solution with one repeated arc, shared by several paths\label{fig:fig3newa}}
\end{subfigure}\hfill
\begin{subfigure}[h]{0.45\textwidth}
\centering
\begin{tikzpicture}
     [scale=.8,every node/.style={circle,text=black,draw,minimum size=0.5cm}]
     \node (a) at (0,0) [] {$s$};
     \node (b) at (3,0.7) [] {$i$};
     \node (c) at (5,0.7) [] {$j$};
     \node (d) at (8,0) [] {$t$};
     \node (e) at (3,-0.7) [] {$k$};
     \node (f) at (5,-0.7) [] {$l$};

     \node[draw=none,green] (ellipsis1) at (1.5,0.5) {$\ldots$};
     \node[draw=none,red] (ellipsis2) at (1.5,1) {$\ldots$};
     \node[draw=none,blue] (ellipsis3) at (1.5,-0.5) {$\ldots$};

     \node[draw=none,green] (ellipsis5) at (6.5,0.5) {$\ldots$};
     \node[draw=none,red] (ellipsis6) at (6.5,1) {$\ldots$};
     \node[draw=none,blue] (ellipsis7) at (6.5,-0.5) {$\ldots$};

     \draw (c) edge[->,bend left=20,red] (ellipsis6);
     \draw (ellipsis6) edge[->,bend left=35,red] (d);
     \draw (f) edge[->,blue] (ellipsis7);
     \draw (ellipsis7) edge[->,blue] (d);
     \draw (c) edge[->,bend right=20,green] (ellipsis5);
     \draw (ellipsis5) edge[->,bend left=20,green] (d);
     \draw (a) edge[->,bend left=35,red] (ellipsis2);
     \draw (ellipsis2) edge[->,bend left=20,red] (b);
     \draw (a) edge[->,blue] (ellipsis3);
     \draw (ellipsis3) edge[->,blue] (e);
     \draw (a) edge[->,bend left=20,green] (ellipsis1);
     \draw (ellipsis1) edge[->,bend right=20,green] (b);

     \draw[->,red] (3.4,0.75)--(4.6,0.75);
     \draw[->,green] (3.4,0.65)--(4.6,0.65);
     \draw[->,blue] (3.4,-0.75)--(4.6,-0.75);
\end{tikzpicture}
\caption{Solution with two repeated arcs, shared by few paths\label{fig:fig3newb}}
\end{subfigure}
\caption{Different sets of $K=3$ paths\label{fig:fig3new}}
\end{figure}

According to the example above, the way how paths spread in a network is affected by the number of repeated arcs as well as by the number of times that the repeated arcs appear in the paths.
This issue will be addressed with two approaches in the following.
First, by minimizing the number of times the repeated arcs appear in the set of the paths.
Later, by minimizing the number of times they are repeated, which accounts also for the number of paths where they appear.

The number of occurrences of the arc $(i,j)\in A$ in a set of $K$ paths $P_K$ is defined as:
$$Occ(i,j;P_K)=\left\{\begin{array}{cl}
0 & \text{if } |\{p\in P_k: (i,j)\in p\}|\le1\\
|\{p\in P_k: (i,j)\in p\}| & \text{otherwise}\
\end{array}
\right.$$
and the number of repeated arc occurrences in $P_k$ is given by
$$RO(P_K)=\sum_{(i,j)\in A}Occ(i,j;P_K).$$
This value is $RO(P_3)=3$ in Figure~\ref{fig:fig3newa}, and $RO(P'_3)=2$ in Figure~\ref{fig:fig3newb}, which makes the solution in the second plot better than the first with respect to the number of repeated arc occurrences.
The purpose of the next formulation is to find a set of $K$ paths which minimizes the number of repeated arc occurrences, $RO(P_K)$.

Let $x_{ij}^k$ be decision variables defined as before and let us consider the formulation:
\begin{subequations}\label{eqf2}
\begin{eqnarray}
\min && f_3(x,y,u)=\sum_{(i,j)\in A}u_{ij}\\
\mbox{subject to} && \sum_{j\in N:(i,j)\in A}x_{ij}^k-\sum_{j\in N:(j,i)\in A}x_{ji}^k=\left\{
     \begin{array}{rl}
      1 & i=s\\
      0 & i\ne s,t\\
     -1 & i=t
     \end{array}\right.,\ \ \ k=1,\ldots,K\label{eqf2c1}\\
  && y_{ij}\le\sum_{k=1}^Kx_{ij}^k,\ \ \ (i,j)\in A\label{eqf2c2}\\
  && (K-1)y_{ij}\ge\sum_{k=1}^Kx_{ij}^k-1,\ \ \ (i,j)\in A\label{eqf2c3}\\
	&& u_{ij} \le K\,y_{ij},\ \ \ (i,j)\in A\label{eqf2c4}\\
	&& u_{ij} \le \sum_{k=1}^K x_{ij}^k,\ \ \ (i,j)\in A\label{eqf2c5}\\
	&& u_{ij} \ge y_{ij}+\sum_{k=1}^K x_{ij}^k-1,\ \ \ (i,j)\in A\label{eqf2c6}\\
	&& x_{ij}^k\in\{0,1\}, y_{ij}\in\{0,1\}, u_{ij}\in\mathbb{N}_0, \ \ \ (i,j)\in A,\ \ \ k=1,\ldots,K\label{eqf2c7}
\end{eqnarray}
\end{subequations}
This formulation has $O(Km)$ variables and $O(Kn+m)$ constraints.
For any $(i,j)\in A$, the variable $y_{ij}$ is defined as in the previous formulation.
Additionally, for a given $(i,j)\in A$:
\begin{itemize}
\item
If $x_{ij}^k=0$ for any $k=1,\ldots,K$, then the constraints (\ref{eqf2c5}) imply that $u_{ij}=0$.
\item
If $x_{ij}^k=1$ for exactly one $k\in\{1,\ldots,K\}$, then by constraints
(\ref{eqf2c5}), $u_{ij}\le1$, and by constraints (\ref{eqf2c6}), $u_{ij}\ge0$.
Because the goal of the problem is to minimize $\sum_Au_{ij}$, then $u_{ij}=0$.
\item
If $x_{ij}^k=1$ for more than one $k\in\{1,\ldots,K\}$, because in the last section we saw that then $y_{ij}=1$, then by constraints (\ref{eqf2c5}), $u_{ij}\le\sum_{k=1}^Kx_{ij}^k$, and by constraints (\ref{eqf2c6}), $u_{ij}\ge\sum_{k=1}^Kx_{ij}^k$. Combining the two conditions implies that $u_{ij}=\sum_{k=1}^Kx_{ij}^k$.
\end{itemize}
Thus, the variable $u_{ij}$ counts the number of times the arc $(i,j)$
appears in the solution, or is equal to 0 if $(i,j)$ is not repeated in
that solution, for any $(i,j)\in A$.
According to these points it can also be concluded that the variables $u_{ij}$ can be relaxed as $u_{ij}\ge0$, without changing the solution.
The constraints (\ref{eqf2c4}) and (\ref{eqf2c5}) can be dropped, because by default the minimization of the objective function implies that $u_{ij}=0$.
Moreover, also the constraints (\ref{eqf2c2}) can be skipped because the variables $y_{ij}$ are only useful when the arc $(i,j)$ appears in more than one path and this constraint is not affected in that case.

\begin{figure}[htb]
\centering
\begin{subfigure}[h]{\textwidth}
\centering
\begin{tikzpicture}
     [scale=.8,every node/.style={circle,text=black,draw,minimum size=0.5cm}]
     \node (a) at (0,0) [] {$s$};
     \node (b) at (3,0) [] {$i$};
     \node (c) at (5,0) [] {$j$};
     \node (d) at (8,0) [] {$t$};

     \node[draw=none,green] (ellipsis1) at (1.5,0.5) {$\ldots$};
     \node[draw=none,red] (ellipsis2) at (1.5,1) {$\ldots$};
     \node[draw=none,blue] (ellipsis3) at (1.5,-1) {$\ldots$};
     \node[draw=none,orange] (ellipsis4) at (1.5,-0.5) {$\ldots$};

     \node[draw=none,green] (ellipsis5) at (6.5,0.5) {$\ldots$};
     \node[draw=none,red] (ellipsis6) at (6.5,1) {$\ldots$};
     \node[draw=none,blue] (ellipsis7) at (6.5,-1) {$\ldots$};
     \node[draw=none,orange] (ellipsis8) at (6.5,-0.5) {$\ldots$};

     \draw (c) edge[->,bend left=35,red] (ellipsis6);
     \draw (ellipsis6) edge[->,bend left=35,red] (d);
     \draw (c) edge[->,bend right=35,blue] (ellipsis7);
     \draw (ellipsis7) edge[->,bend right=35,blue] (d);
     \draw (c) edge[->,bend left=20,green] (ellipsis5);
     \draw (ellipsis5) edge[->,bend left=20,green] (d);
     \draw (c) edge[->,bend right=20,orange] (ellipsis8);
     \draw (ellipsis8) edge[->,bend right=20,orange] (d);
    \draw (a) edge[->,bend left=35,red] (ellipsis2);
    \draw (ellipsis2) edge[->,bend left=35,red] (b);
    \draw (a) edge[->,bend right=35,blue] (ellipsis3);
    \draw (ellipsis3) edge[->,bend right=35,blue] (b);
    \draw (a) edge[->,bend left=20,green] (ellipsis1);
    \draw (ellipsis1) edge[->,bend left=20,green] (b);
    \draw (a) edge[->,bend right=20,orange] (ellipsis4);
    \draw (ellipsis4) edge[->,bend right=20,orange] (b);

    \draw[->,red] (3.4,0.1)--(4.6,0.1);
    \draw[->,green] (3.4,0)--(4.6,0);
    \draw[->,blue] (3.4,-0.2)--(4.6,-0.2);
    \draw[->,orange] (3.4,-0.1)--(4.6,-0.1);
\end{tikzpicture}
\caption{Solution with one repeated arc, shared by several paths\label{fig:fig4newa}}
\end{subfigure}

\begin{subfigure}[h]{\textwidth}
\centering
\begin{tikzpicture}
    [scale=.8,every node/.style={circle,text=black,draw,minimum size=0.5cm}]
    \node (a) at (0,0) [] {$s$};
    \node (b) at (3,0.7) [] {$i$};
    \node (c) at (5,0.7) [] {$j$};
    \node (e) at (8,0.7) [] {$k$};
    \node (f) at (10,0.7) [] {$l$};
    \node (d) at (13,0) [] {$t$};

    \node[draw=none,red] (ellipsis1) at (1.5,1) {$\ldots$};
    \node[draw=none,red] (ellipsis3) at (6.5,1) {$\ldots$};
    \node[draw=none,red] (ellipsis5) at (11.5,1) {$\ldots$};
    \node[draw=none,green] (ellipsis2) at (1.5,0.5) {$\ldots$};
    \node[draw=none,green] (ellipsis4) at (6.5,0.5) {$\ldots$};
    \node[draw=none,green] (ellipsis6) at (11.5,0.5) {$\ldots$};
    \node[draw=none,orange] (ellipsis7) at (6.5,-0.25) {$\ldots$};
    \node[draw=none,blue] (ellipsis8) at (6.5,-0.75) {$\ldots$};

    \draw (a) edge[->,bend left=35,red] (ellipsis1);
    \draw (ellipsis1) edge[->,bend left=20,red] (b);
    \draw (c) edge[->,bend left=20,red] (ellipsis3);
    \draw (ellipsis3) edge[->,bend left=20,red] (e);
    \draw (f) edge[->,bend left=20,red] (ellipsis5);
    \draw (ellipsis5) edge[->,bend left=20,red] (d);

    \draw (a) edge[->,bend left=20,green] (ellipsis2);
    \draw (ellipsis2) edge[->,bend right=20,green] (b);
    \draw (c) edge[->,bend right=20,green] (ellipsis4);
    \draw (ellipsis4) edge[->,bend right=20,green] (e);
    \draw (f) edge[->,bend right=20,green] (ellipsis6);
    \draw (ellipsis6) edge[->,bend left=20,green] (d);

    \draw (a) edge[->,bend left=-20,orange] (ellipsis7);
    \draw (ellipsis7) edge[->,bend left=-20,orange] (d);

    \draw (a) edge[->,bend left=-35,blue] (ellipsis8);
    \draw (ellipsis8) edge[->,bend left=-35,blue] (d);

    \draw[->,red] (3.4,0.75)--(4.6,0.75);
    \draw[->,red] (8.4,0.75)--(9.6,0.75);
    \draw[->,green] (3.4,0.65)--(4.6,0.65);
    \draw[->,green] (8.4,0.65)--(9.6,0.65);
\end{tikzpicture}
\caption{Solution with two repeated arcs, shared by one path\label{fig:fig4newb}}
\end{subfigure}
\caption{Different sets of $K=4$ paths\label{fig:fig4new}}
\end{figure}
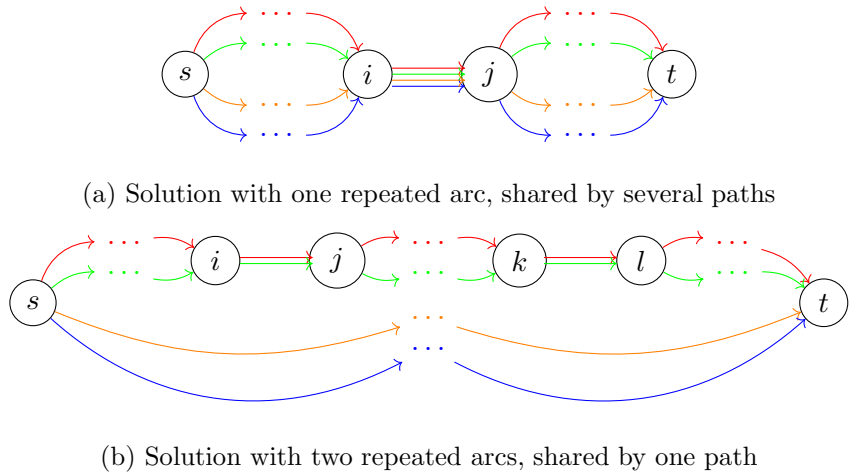

In the case shown in Figure~\ref{fig:fig4new}, counting the number of times that the repeated arcs appear in the solution is not enough to distinguish between the two depicted solutions.
In fact, Figures~\ref{fig:fig4newa} and~\ref{fig:fig4newb} we have $RO(P_4)=RO(P'_4)=4$, where the repeated arcs are $(i,j)$ in the first case, and $(i,j)$ and $(k,l)$ in the second.
Nevertheless, the fact that the repetitions happen in different arcs should be valued, given that this is reflected in the dissimilarity of these two sets of paths.
Therefore, a new concept will be introduced, the number of repetitions for any arc that appears in the solution more than once.

Let the number of repetitions of an arc $(i,j)\in A$ in the set $P_K$ be the number of paths it belongs to, excluding its first utilization, that is,
$$Rep(i,j;P_K)=\left\{\begin{array}{cl}
      0 & \text{if } |\{p\in P_k: (i,j)\in p\}|\le1\\
Occ(i,j;P_K)-1 & \text{otherwise}\
\end{array}
\right.$$
Then, the number of arc repetitions in $P_k$ is given by
$$Rep(P_K)=\sum_{(i,j)\in A}Rep(i,j;P_K).$$

The number of arc repetitions defined above reflects two aspects: the number of arcs shared by more than one paths as well as the number of paths that share them.
Recalling the example in Figure~\ref{fig:fig2new}, for the solution $P_4$ in Figure~\ref{fig:fig2newa} we have $Rep(P_4)=3$, because the arc $(i,j)$ is repeated 3 times, whereas for the solution $P'_4$ in Figure~\ref{fig:fig2newb} we have $Rep(P'_4)=2$, because both the arcs $(i,j)$ and $(k,l)$ are repeated once.
The next formulation aims at minimizing $Rep(P_K)$.

Like before, for modeling the problem of finding $K$ paths from $s$ to $t$ which minimize the number of arc repetitions, the variables $x_{ij}^k$ represent $K$ paths and have value 1 if the arc $(i,j)$ is in the $k$-th path from $s$ to $t$ or are 0 otherwise, for any $(i,j)\in A$. The formulation is as follows:
\begin{subequations}\label{eqf4}
\begin{eqnarray}
     \min && f_4(x,w,u)=\sum_{(i,j)\in A}u_{ij}\label{eqf4of}\\
\mbox{subject to} && \sum_{j\in N:(i,j)\in A}x_{ij}^k-\sum_{j\in N:(j,i)\in A}x_{ji}^k=\left\{
     \begin{array}{rl}
      1 & i=s\\
      0 & i\ne s,t\\
     -1 & i=t
     \end{array}\right.,\ \ \ k=1,\ldots,K\label{eqf4c1}\\
  && x_{ij}^k\le w_{ij},\ \ \ (i,j)\in A,\ \ \ k=1,\ldots,K\label{eqf4c2}\\
  && w_{ij}\le\sum_{k=1}^Kx_{ij}^k,\ \ \ (i,j)\in A\label{eqf4c3}\\
  && u_{ij}=\sum_{k=1}^Kx_{ij}^k-w_{ij},\ \ \ (i,j)\in A\label{eqf4c4}\\
  && x_{ij}^k\in\{0,1\}, \ w_{ij}\in\{0,1\}, \ u_{ij}\ge 0 \ \ \ (i,j)\in A,\ \ \ k=1,\ldots,K\label{eqf4c5}
\end{eqnarray}
\end{subequations}

This formulation has $O(Km)$ variables and $O(K(m+n))$ constraints. The constraints (\ref{eqf4c1}) are flow conservation constraints that define a set of $K$ paths from node $s$ to node $t$. The constraints (\ref{eqf4c2}) and (\ref{eqf4c3}) are used to define the variables $w_{ij}\in \{0, 1\}$, each one equal to 1 if and only if the arc $(i,j)$ is used in at least one path, or 0 otherwise, for any $(i,j)\in A$.
Additionally, the constraints (\ref{eqf4c4}) define the auxiliary variables $u_{ij}$, which corresponds to the number of times that arc $(i,j) \in A$ is repeated in different paths. For a given $(i,j)\in A$:
\begin{itemize}
\item
If $x_{ij}^k=0$ for any $k=1,\ldots,K$, then the constraints (\ref{eqf4c3}) imply that $w_{ij}=0$. Therefore the constraints (\ref{eqf4c4}) imply that $u_{ij}=0$.
\item
If $x_{ij}^k=1$ for exactly one $k\in\{1,\ldots,K\}$, then by constraints (\ref{eqf4c2}), $w_{ij}\ge1$ which together with (\ref{eqf4c5}) imply $w_{ij}=1$, and by constraints (\ref{eqf4c4}), $u_{ij}=0$.
\item
If $x_{ij}^k=1$ for more than one $k\in\{1,\ldots,K\}$, then by constraints (\ref{eqf4c2}), $w_{ij}\ge1$ which together with (\ref{eqf4c5}) imply that $w_{ij}=1$, and by constraints (\ref{eqf4c4}), $u_{ij}=\sum_{k=1}^Kx_{ij}^k-1$. 
\end{itemize}

Because both $x_{ij}^k$ and $w_{ij}$ are binary variables, the variables $u_{ij}$ are implicitly defined as integers for any $(i,j) \in A$, $k=1,\ldots,K$. Additionally, (\ref{eqf4c4}) together with the  non-negative constraints of the variables $u_{ij}$
imply (\ref{eqf4c3}).
Therefore, the constraints (\ref{eqf4c3}) can be skipped from the formulation.

Finally, we observe that constraints (\ref{eqf4c2}) can be aggregated as
\begin{equation}\label{eqf4c2a}
\sum_{k=1}^Kx_{ij}^k\le K\,w_{ij},\ \ \ (i,j)\in A.
\end{equation}

Like for formulations (\ref{eqf3}) and (\ref{eqf1}), both formulations (\ref{eqf2}) and (\ref{eqf4}) admit optimal solutions with loops. However, in such cases Proposition \ref{prop3} holds and the Algorithm~\ref{alg:alg1} can be applied in order to obtain loopless optimal solutions.

\begin{proposition}\label{prop3}
\begin{enumerate}
\item
Let $(x,y,u)$ be an optimal solution for problem (\ref{eqf2}).
Let $\bar{x}\in\{0,1\}^{Km}$ be the vector output by Algorithm~\ref{alg:alg1} when applied
to $x$, $\bar{y}\in\{0,1\}^m$ and $\bar{u}\in\mathbb{N}_0^m$ be such that the constraints (\ref{eqf2c2}) to (\ref{eqf2c7}) are satisfied.
Then, $(\bar{x},\bar{y},\bar{u})$ is a loopless optimal solution for
problem (\ref{eqf2}).
\item
Let $(x,w,u)$ be an optimal solution for problem (\ref{eqf4}).
Let $\bar{x}\in\{0,1\}^{Km}$ be the vector output by Algorithm~\ref{alg:alg1} when applied to $x$, and $\bar{w}\in\{0,1\}^m$ and $\bar{u}$ be such that the constraints (\ref{eqf4c2}) to (\ref{eqf4c5}) are satisfied.
Then, $(\bar{x},\bar{w},\bar{u})$ is a loopless optimal solution for
problem (\ref{eqf4}).
\end{enumerate}
\end{proposition}

\subsection{Computational experiments\label{sec:6.1}}
The tests setup presented in the following for experimentally assessing formulations (\ref{eqf2}) and (\ref{eqf4}) is similar to what was described in Section~\ref{sec:4.1}.

While a number of relaxations of formulation (\ref{eqf2}) were tested, only the results for the one with the best behavior with regard to the run times are reported. These correspond to the variant that omits the set of constraints (\ref{eqf2c2}) from the original model, (\ref{eqf2}). For simplicity, we keep the same designation and in the following refer to the new model as formulation (\ref{eqf2}).
Likewise, several variants of (\ref{eqf4}) were tested. The model obtained from (\ref{eqf4}) by replacing  constraints (\ref{eqf4c2}) with its aggregated version, constraints (\ref{eqf4c2a}), was significantly faster than the others. Thus, only the results for this new model are presented. Again, for simplicity, we keep the same designation and hereafter refer to this new model as formulation (\ref{eqf4}).

The following codes were written in C, while using IBM ILOG CPLEX version 12.7 for solving the integer programs:
\begin{itemize}
\item \texttt{MRO}: Implementation for minimizing the number of repeated arc occurrences, formulation (\ref{eqf2}).
\item
\texttt{MAR}: Implementation for minimizing the number of arc repetitions, formulation (\ref{eqf4}).
\end{itemize}

\begin{table}[htb]
\centering
\caption{Run times of \texttt{MRO} (seconds)\label{tab:MROcpu0}}
\small{
\begin{tabular}{@{}lccccccccl@{}}\toprule
       &\multicolumn{8}{c}{$K$} \\\cmidrule(lr){2-9}
     $R_{n,m}$ & 3 & 4 & 5 & 6 & 7 & 8 & 9 & 10\\\cmidrule(lr){1-9}
 $R_{100,500}$ & 0.051 & 0.067 & 0.108 & 0.183 & 0.274 & 0.572 & 0.836 &	1.322\\
$R_{100,1000}$ & -- & 0.123 & 0.157 & 0.220 & 0.238 & 0.307 & 0.386 &	0.497\\\cmidrule(lr){1-9}
$R_{300,1500}$ & 0.161 & 0.214 & 0.253 & 0.331 & 0.439 & 0.706 & 0.998 &	1.601\\
$R_{300,3000}$ & 0.303 & 0.370 & 0.444 & 0.508 & 0.677 & 0.724 & 0.846 &	1.000\\\cmidrule(lr){1-9}
$R_{500,2500}$ & 0.280 & 0.382 & 0.485 & 0.594 & 0.757 & 1.035 & 1.326 &	1.687\\
$R_{500,5000}$ & 0.486 & 0.603 & 0.746 & 0.906 & 1.083 & 1.285 & 1.500 &	1.781\\\midrule
Average     & 0.256  & 0.293 & 0.365   & 0.457  & 0.578 & 0.771  & 0.892  & 1.314 & \fbox{0.627}\\\bottomrule
\end{tabular}
\begin{tabular}{@{}lccccccccl@{}}\toprule
    &\multicolumn{8}{c}{$K$} \\\cmidrule(lr){2-9}
  $G_{p,q}$ & 3 & 4 & 5 & 6 & 7 & 8 & 9 & 10\\\cmidrule(lr){1-9}
 $G_{3,12}$ & 0.037 & 0.251 & 0.815 &  1.041 &   5.496 &  15.547 &   8.673 &  11.894\\
 $G_{4,36}$ & 0.094 & 0.293 & 2.204 & 43.209 & 140.598 & \colorbox{lightgray}{300} & \colorbox{lightgray}{300} & \colorbox{lightgray}{300}\\
  $G_{6,6}$ & 0.165 & 0.009 & 0.276 &  0.752 &   2.409 &   3.930 &  23.715 &  14.426\\
$G_{12,12}$ & 0.075 & 0.274 & 0.880 &  3.660 &  15.674 &   9.391 &  87.845 & 146.961\\\midrule
Average & 0.092  & 0.206 & 1.043   & 12.165  & 41.044 & 82.217  & 150.058  & 118.320 & \fbox{45.018}\\\bottomrule
\end{tabular}}
\end{table}

\begin{figure}[htb]
\centering
\begin{subfigure}[h]{0.45\textwidth}
	\centering
	\input{Images/mro_cpu_rn.tex}	
\end{subfigure}
\begin{subfigure}[h]{0.45\textwidth}
	\centering
	\input{Images/mro_cpu_gr.tex}
\end{subfigure}	
\caption{Run times of \texttt{MRO} (seconds)\label{fig:mro_cpu}}
\end{figure}

Unlike the code \texttt{MAR}, which was able to find the optimal solution for all the instances, the code \texttt{MRO} resumed after the 300 seconds limit for the $4\times36$ grids when seeking for more than 7 paths.
This can be seen in Table~\ref{tab:MROcpu0}, which reports their average run times, depicted in Figure~\ref{fig:mro_cpu}.

The code \texttt{MAR} outperformed the previous in almost all cases in terms of run time. Figure~\ref{fig:mar_cpu} illustrates the results summarized in Table~\ref{tab:MARcpu0}. The major differences are found in the results associated to the grid networks (where run times fall, in average, by 99.6\%) and in the subset of the random networks previously identified as the harder to solve (which has reductions of 48\% for the $R_{100,500}$ instances, 41\% for the $R_{300,1500}$ instances and 25\% for the $R_{300,1500}$ instances). Still, more than emphasizing the relative reductions towards \texttt{MRO}, it is important to point out that \texttt{MAR} solved all instances, in less than 2 seconds. Moreover, results suggest that \texttt{MAR} is far less susceptible to the effect of the variables that have been identified as inhibiting to other models (the layout of the network, its sparseness and the proportion between the value of $K$ and the size of the network), indicating that this model can be used in a wider range of situations. As a final remark, there was some instability in the run time of the grid networks that should be clarified in a future work.

\begin{table}[htb]
\centering
\caption{Run times of \texttt{MAR} (seconds)\label{tab:MARcpu0}}
\small{
\begin{tabular}{@{}lccccccccl@{}}\toprule
       &\multicolumn{8}{c}{$K$} \\\cmidrule(lr){2-9}
     $R_{n,m}$ & 3 & 4 & 5 & 6 & 7 & 8 & 9 & 10\\\cmidrule(lr){1-9}
 $R_{100,500}$ & 0.051 & 0.078 &	0.116 & 0.172 & 0.181 &	0.301 & 0.342 &	0.519\\
$R_{100,1000}$ & -- & 0.097 & 0.140 & 0.143 & 0.193 &	0.265 & 0.236 &	0.316\\\cmidrule(lr){1-9}
$R_{300,1500}$ & 0.139 & 0.193 & 0.222 & 0.259 & 0.313 &	0.426 & 0.486 &	0.717\\
$R_{300,3000}$ & 0.226 & 0.323 & 0.391 & 0.458 & 0.609 &	0.706 & 0.773 &	0.905\\\cmidrule(lr){1-9}
$R_{500,2500}$ & 0.249 & 0.318 &	0.374 & 0.442 & 0.593 &	0.733 & 1.003 &	1.172\\
$R_{500,5000}$ & 0.384 & 0.571 & 0.825 & 1.024 & 1.211 &	1.369 & 1.576 &	1.730\\\midrule
Average & 0.209  & 0.263 & 0.344 & 0.416  & 0.516 & 0.633  & 0.736  & 0.893 & \fbox{0.501}\\\bottomrule
\end{tabular}
\begin{tabular}{@{}lccccccccl@{}}\toprule
       &\multicolumn{8}{c}{$K$} \\\cmidrule(lr){2-9}
     $G_{p,q}$ & 3 & 4 & 5 & 6 & 7 & 8 & 9 & 10\\\cmidrule(lr){1-9}
 $G_{3,12}$ & 0.010 & 0.016 & 0.012 & 0.017 & 0.020 & 0.021 & 0.023 & 0.054\\
 $G_{4,36}$ & 0.047 & 0.083 & 0.126 & 0.157 & 0.215 & 0.250 & 0.620 & 0.382\\
  $G_{6,6}$ & 0.010 & 0.004 & 0.114 & 0.012 & 0.018 & 0.214 & 0.157 & 0.020\\
$G_{12,12}$ & 0.034 & 0.049 & 0.142 & 0.242 & 0.124 & 0.383 & 0.323 & 1.292\\\midrule
Average & 0.025  & 0.038 & 0.098 & 0.107  & 0.094 & 0.217  & 0.280  & 0.437 & \fbox{0.162}\\\bottomrule
\end{tabular}}
\end{table}

\begin{figure}[htb]
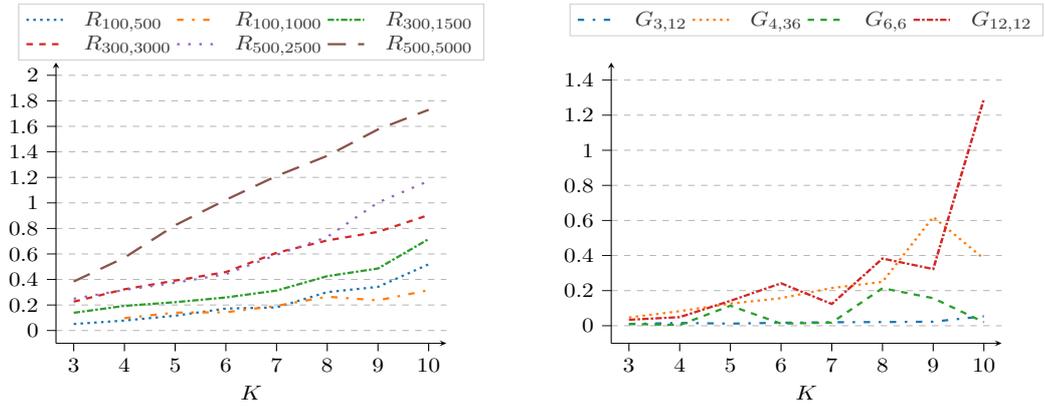

\centering
\begin{subfigure}[h]{0.45\textwidth}
	\centering
	\input{Images/mar_cpu_rn.tex}	
\end{subfigure}
\begin{subfigure}[h]{0.45\textwidth}
	\centering
	\input{Images/mar_cpu_gr.tex}
\end{subfigure}	
\caption{Run times of \texttt{MAR} (seconds)\label{fig:mar_cpu}}
\end{figure}

The average integer programming gaps produced by the linear relaxation of (\ref{eqf2}), in Table~\ref{tab:MROig0}, range between 1\% ad 18\% for the random instances and between 7\% and 33\% for the grid instances.
The same values were all 0 for the linear relaxation of formulation (\ref{eqf4}).

\begin{table}[htb]
\centering
\caption{Average integer programming gaps of \texttt{MRO$_L$} (\%)\label{tab:MROig0}}
\small{
\begin{tabular}{@{}lccccccccl@{}}\toprule
       &\multicolumn{8}{c}{$K$} \\\cmidrule(lr){2-9}
     $R_{n,m}$ & 3 & 4 & 5 & 6 & 7 & 8 & 9 & 10\\\midrule
 $R_{100,500}$ & 5 & 18 & 12 & 11 &  8 &  9 &  9 & 10\\
$R_{100,1000}$ &-- &  6 & 10 &  9 & 15 & 15 & 14 & 12\\\midrule
$R_{300,1500}$ & 7 & 15 & 18 & 12 & 10 &  8 & 11 & 11\\
$R_{300,3000}$ & 1 &  5 & 11 & 13 & 12 & 14 & 14 & 12\\\midrule
$R_{500,2500}$ & 8 & 19 & 11 & 11 &  9 & 11 & 12 & 11\\
$R_{500,5000}$ & 3 &  6 &  7 & 11 & 13 & 12 & 14 & 10\\\bottomrule
\end{tabular}\hfill
\begin{tabular}{@{}lccccccccl@{}}\toprule
       &\multicolumn{8}{c}{$K$} \\\cmidrule(lr){2-9}
  $G_{p,q}$ & 3 & 4 & 5 & 6 & 7 & 8 & 9 & 10\\\midrule
 $G_{3,12}$ & 25 & 28 & 17 & 13 & 11 & 10 &  9 & 9\\
 $G_{4,36}$ & 25 & 33 & 36 & 21 & 14 & 11 &  8 & 7\\
  $G_{6,6}$ & 25 & 33 & 29 & 28 & 25 & 24 & 22 & 21\\
$G_{12,12}$ & 25 & 33 & 29 & 28 & 25 & 24 & 22 & 21\\\bottomrule
\end{tabular}}
\end{table}

\cleardoublepage
\section{Bounding the number of arc presences\label{sec:7}}   

Formulation (\ref{eqf1}) aims at minimizing the number of arcs which appear in more than one path.
The undesired consequence of the simplicity of this objective function may be that few arcs appear in many different paths, a situation which is illustrated in Figure~\ref{fig:fig2new}.
This was the motivation to consider the number of times that each repeated arc appears in the objective function of formulations (\ref{eqf2}) and (\ref{eqf4}), presented in Section~\ref{sec:6}.
In the following we propose an intermediate solution, consisting of overcoming this handicap by adding a constraint over the number of times that each arc is present in the solution. In addition, we found that this approach also had an interesting side effect in the paths dissimilarity of the solutions generated by formulations (\ref{eqf2}) and (\ref{eqf4}): when in the presence of multiple optimal solutions with respect to the number of arc repetitions, bounding the number of times that each arc appears gives an extra condition for untying those solutions, thus increasing their dissimilarity. Therefore, the new set of constraints will also be considered in the context of formulations (\ref{eqf2}) and (\ref{eqf4}).  Naturally, on the downside, the new models may be more time consuming.

The new set of constraints are very similar to the set of constraints proposed by \cite{Constantino_2017} to prevent repetitions of the arcs over the time horizon. However, whereas in \cite{Constantino_2017} the bound is an external parameter, in our case the bound is fixed by solving a simple problem that optimizes the worst case in terms of the number of times that each arc is present in the solution.

The new formulation aims at finding a set of $K$ paths with the minimum maximum number of arc presences.
It is as follows
\begin{subequations}\label{eqf0}
	\begin{eqnarray}
	\min && \max_{(i,j)\in A}\left\{\sum_{k=1}^Kx_{ij}^k\right\}\\
	\mbox{subject to} && \sum_{j\in N:(i,j)\in A}x_{ij}^k-\sum_{j\in N:(j,i)\in A}x_{ji}^k=\left\{
	\begin{array}{rl}
	1 & i=s\\
	0 & i\ne s,t\\
	-1 & i=t
	\end{array}\right.,\ \ \ k=1,\ldots,K\\
	&& x_{ij}^k\in\{0,1\}, \ \ \ (i,j)\in A,\ \ \ k=1,\ldots,K
	\end{eqnarray}
\end{subequations}
where the decision variables are $x_{ij}^k\in\{0,1\}$ equal to 1 if and
only if the arc $(i,j)$ appears in path $k$, $(i,j)\in A$, $k=1,\ldots,K$.
This formulation can be linearized as
\begin{subequations}\label{eqlf0}
	\begin{eqnarray}
	\min && r\\
	\mbox{subject to} && \sum_{j\in N:(i,j)\in A}x_{ij}^k-\sum_{j\in N:(j,i)\in A}x_{ji}^k=\left\{
	\begin{array}{rl}
	1 & i=s\\
	0 & i\ne s,t\\
	-1 & i=t
	\end{array}\right.,\ \ \ k=1,\ldots,K\\
	&& r \ge \sum_{k=1}^Kx_{ij}^k, \ \ \ (i,j)\in A\\
	&& x_{ij}^k\in\{0,1\}, \ \ \ (i,j)\in A,\ \ \ k=1,\ldots,K
	\end{eqnarray}
\end{subequations}
which is equivalent to its linear programming relaxation.

Now, let $R^*$ be the optimal value for problem (\ref{eqlf0}) computed in advance. Then a new constraint can be added to formulations (\ref{eqf1}), (\ref{eqf2}) and (\ref{eqf4}) in order to prevent the number of arc presences from exceeding that value,
\begin{equation}\label{eq1}
\sum_{k=1}^K x_{ij}^k\le R^*,\ \ \ (i,j)\in A.
\end{equation}
The problems modeled by the resulting formulations are constrained and, therefore, different versions of the original ones. To assess whether the new problems produce better solutions to the $K$ dissimilar paths problem, a set of computational experiments was performed. Results are discussed in Sections \ref{sec:7.1} and \ref{sec:8.2}.

\subsection{Computational experiments\label{sec:7.1}}
In this section we analyze the impact of adding the constraints (\ref{eq1}) to formulations (\ref{eqf1}), (\ref{eqf2}) and (\ref{eqf4}) on the run times.  
Considering the experimental setup described in Section~\ref{sec:4}, the following codes were tested:
\begin{itemize}
\item \texttt{MRAA}:
implementation of formulation (\ref{eqf1}) including the constraints (\ref{eq1});
\item \texttt{MROA}:
implementation of formulation (\ref{eqf2}) including the constraints (\ref{eq1});
\item \texttt{MARA}:
implementation of formulation (\ref{eqf4}) including the constraints (\ref{eq1}).
\end{itemize}
Like before, the codes were written in C, calling the integer programming solver IBM ILOG CPLEX version 12.7.
The impact of adding the constraints (\ref{eq1}) to the formulations introduced before on the several parameters is measured as
$100\times(\texttt{MA}-\texttt{M})/\texttt{M}\ \%,$
where \texttt{M} stands for each of the codes listed above.

Of the new codes, \texttt{MARA} was the only one able of finding the optimal solution for all the instances within the time limit of 300 seconds. Both \texttt{MRAA} and \texttt{MROA} resumed after that limit for the $4\times 36$ grids and $K=10$.

According to Table~\ref{tab:MAcpu0}, in most cases the constrained problems require more time to solve when the networks are denser, while the run times do not change much for the sparser instances.
In some of the latter cases there is even a speed up. The speed up happens mostly for \texttt{MRA} and \texttt{MRO}, and is particularly relevant in grids, which are instances where finding solutions is difficult.
It should be added that the run times required for solving the problem (\ref{eqlf0}) are included in the values on Table~\ref{tab:MAcpu0}.
The results regarding the run times for the grid networks were uneven.

The integer programming gap associated with \texttt{MARA} was equal to 0 in all the tested instances.
In general these values increased for the remaining formulations after adding the new constraints, specially in the random instances and in the smaller grids and big $K$'s. However, this variation is not very meaningful, as the unconstrained and the constrained problems are different.

\begin{landscape}
\begin{table}[htb]
\centering
\caption{Run times variation for \texttt{MRAA}, \texttt{MROA} and \texttt{MARA} (\%)\label{tab:MAcpu0}}
\small{
\begin{tabular}{l c@{ }c@{ }c@{ }c@{ }c@{ }c@{ }c@{ }c  c@{ }c@{ }c@{ }c@{ }c@{ }c@{ }c@{ }c  c@{ }c@{ }c@{ }c@{ }c@{ }c@{ }c@{ }c }\toprule
	& \multicolumn{8}{c}{\texttt{MRAA}} & \multicolumn{8}{c}{\texttt{MROA}} & \multicolumn{8}{c}{\texttt{MARA}}\\\cmidrule(lr){2-9} \cmidrule(lr){10-17}\cmidrule(lr){18-25}
	& \multicolumn{8}{c}{$K$} & \multicolumn{8}{c}{$K$} & \multicolumn{8}{c}{$K$}\\\cmidrule(lr){2-9} \cmidrule(lr){10-17}\cmidrule(lr){18-25}
$R_{n,m}$ &  3 &  4 &  5 &  6 &  7 & 8 & 9 & 10   & 3 &  4 &  5 &  6 &  7 & 8 & 9 & 10 & 3 &  4 &  5 &  6 &  7 & 8 & 9 & 10\\\midrule
$R_{100,500}$ & 32 & 56 & 38 & 25 & 19 & -14 & -17 & -50 &
29 & 28 & 17 & -14 &  2 & -29 & -30 & -39 &
42 & 43 & 35 & 20 &  73 &  40 &  48 & 23\\
$R_{100,1000}$ & -- & 64 & 50 & 38 & 71 &  42 &  59 &  56 &
-- & 38 & 22 &  8 & 37 & 15 & 10 &  9 &
-- & 83 & 75 & 87 & 77 & 49 & 97 & 86\\\midrule
$R_{300,1500}$ & 48 & 39 & 34 & 44 & 47 &  37 &  -3 & -21 &
29 & 25 & 40 & 27 & 28 & -2 & -7 & -23 &
58 & 60 & 67 & 80 & 92 & 67 & 66 & 42\\
$R_{300,3000}$ & 48 & 62 & 66 & 66 & 69 &  72 &  70 & 584 &
27 & 40 & 40 &  45 & 41 &  51 &  55 & 532 &
71 & 76 & 74 & 79 &  73 &  76 &  90 & 614\\\midrule
$R_{500,2500}$ & 50 & 49 & 45 & 41 & 56 &  52 &  11 &  33 &
31 & 22 & 25 & 24 & 35 &  43 & 10 &  9 &
60 & 63 & 79 & 78 & 81 & 104 & 56 & 55\\
$R_{500,5000}$ & 44 & 39 & 52 & 52 & 79 &  80 & 253 & 959 &
25 & 34 & 48 & 45 & 62 & 61 & 230 & 982 &
53 & 56 & 55 & 50 & 70 & 79 & 239 & 1031\\\bottomrule
\end{tabular}\hfill
\begin{tabular}{l c@{ }c@{ }c@{ }c@{ }c@{ }c@{ }c@{ }c  c@{ }c@{ }c@{ }c@{ }c@{ }c@{ }c@{ }c  c@{ }c@{ }c@{ }c@{ }c@{ }c@{ }c@{ }c }\toprule
	& \multicolumn{8}{c}{\texttt{MRAA}} & \multicolumn{8}{c}{\texttt{MROA}} & \multicolumn{8}{c}{\texttt{MARA}}\\\cmidrule(lr){2-9} \cmidrule(lr){10-17}\cmidrule(lr){18-25}
	& \multicolumn{8}{c}{$K$} & \multicolumn{8}{c}{$K$} & \multicolumn{8}{c}{$K$}\\\cmidrule(lr){2-9} \cmidrule(lr){10-17}\cmidrule(lr){18-25}
   $G_{p,q}$ &  3 &  4 &  5 &  6 &  7 & 8 & 9 & 10   & 3 &  4 &  5 &  6 &  7 & 8 & 9 & 10 & 3 &  4 &  5 &  6 &  7 & 8 & 9 & 10\\\midrule
  $G_{3,12}$ & $-13$ & $-85$ & $-70$ & $-70$ & $-61$ & $-43$ & $-22$ & $-27$ &
  $-43$ &  $-90$ &  $-81$ & $-80$ & $-95$ & $-93$ & $-90$ & $-41$ &
61 &  42 &  81 &  60 &  47 &  73 &  67 & 72\\
 $G_{4,36}$ &  $-24$ &   9 &   7 & 294 & $-81$ & $-99$ & $-98$ & 0 &
 $-36$ &  $-65$ & 36 & $-90$ & $-74$ & $-97$ & $-98$ &  0 &
57 &  35 &  22 &  20 &   7 &  10 & $-43$ & $-2$\\
 $G_{6,6}$ &  16 & $-87$ & $-68$ & $-92$ & $-97$ & $-84$ &  14 & $-24$ &
 $-90$ & 79 &  $-86$ & $-94$ & $-53$ & $-91$ & $-93$ & $-93$ &
78 & 273 & $-79$ & 103 & 558 & $-86$ & $-47$ & 82\\
 $G_{12,12}$ & $-56$ & $-66$ & $-57$ & $-95$ & $-97$ & $-99$ & $-98$ & $-98$ &
 $-44$ & $-78$ &  $-77$ & $-86$ & $-95$ & $-87$ & $-95$ & $-99$ &
 86 &  31 &  43 & $-46$ & 348 & $-28$ &  73 & $-14$\\\bottomrule
\end{tabular}}
\end{table}
\end{landscape}

\section{Application to finding $K$ dissimilar paths\label{sec:8}} 
In the previous sections, integer formulations were presented for the problems of finding $K$ paths between a given pair of nodes, such that:
\begin{itemize}
\item the number of arc overlaps for each pair of paths is minimized;
\item the number of repeated arcs is minimized;
\item the number of occurrences of repeated arcs, or the number of arc repetitions, is minimized.
\end{itemize}
These problems were suggested with the purpose of capturing characteristics of sets of $K$ dissimilar paths. Therefore in the current section the behavior of the presented approaches is discussed and compared from the perspective of the dissimilarity, based on the metric $D_1$ defined in Section~\ref{sec:1}. In addition, these approaches are also compared with the classical method known as the IPM and proposed in~\citet{Johnson_1992}.
As mentioned in Section~\ref{sec:2}, the idea behind this method is to solve $K$ shortest path problems and penalize the cost of the selected arcs every time one of those paths is computed, in order to prevent their overlap as much as possible.
This approach has often been used for comparisons in the literature due to its flexibility to incorporate features of various problems as well as the simplicity of its implementation.

While a relative comparison of the dissimilarities produced by the different approaches is possible, an absolute assessment of the results requires the optimal value of $D_1$ to be known for each instance. As that is the case only for some of the grid instances, the analysis is based on a description of the results of each model followed by a relative overall comparison. For that purpose, both the average dissimilarity between the pairs of paths in the solution (\texttt{AvDi}) and their minimum dissimilarity (\texttt{MiDi}) are calculated for each instance. Then, the averages of \texttt{AvDi} and of \texttt{MiDi} for each set of instances are calculated.

The test bed used for this study and the testing conditions are the same described in Section~\ref{sec:4}. Furthermore, the \texttt{IPM} code implements the method with the same name, in C and calling the CPLEX solver for solving the shortest path problems. A set of preliminary tests was run, in order to decide how to parametrize the \texttt{IPM}.
\texttt{IPM} works with unitary arc costs and was tested with the additive penalizations $\alpha=0.25,0.50,0.75,1.00$, applied to the cost of the arcs of the most recently found path. Since considering the penalization $\alpha=1.00$ produced better results than the remaining penalties for \texttt{AvDi} and \texttt{MiDi} in all random instances, this was the value fixed in the rest of the experiments.
Tables \ref{tab:ipm_AvDi_MiDi_r} and \ref{tab:IPMcpu0} in Appendix~\ref{app:3} present the dissimilarities and the run times of this implementation, which will be used for comparison with the introduced formulations.
The dissimilarity results are far better for grids than for random networks. This is due to the fact that all paths in the used grids of a certain size have the same length. In effect, \texttt{IPM} naturally avoids the arcs previously used, unless the cost associated to the increase in the number of arcs exceeds the penalty. When applied to the random networks, the length of the paths varies and so does the dissimilarity of the solutions found by the method. 
Because this method essentially solves $K$ shortest path problems, it has polynomial complexity of $O(Km+Kn\log n)$ and it run fast for any of the considered instances: in less than 0.70 seconds in random networks and in less than 0.05 seconds in grid networks.

The rest of the section is organized as follows:
first, the results for the four initial formulations and the \texttt{IPM} are compared;
then the effect of adding the constraints (\ref{eq1}) to the original models is discussed.

\subsection{Unconstrained formulations\label{sec:8.1}}
The average and minimum dissimilarities of \texttt{MAO}, \texttt{MRA}, \texttt{MRO}, and \texttt{MAR} are compared in the following.
The detailed results for each formulation can be found in Appendix \ref{app:1}.

\begin{figure}[htb]
\centering
\begin{subfigure}[h]{0.4\textwidth}
	\centering
	\input{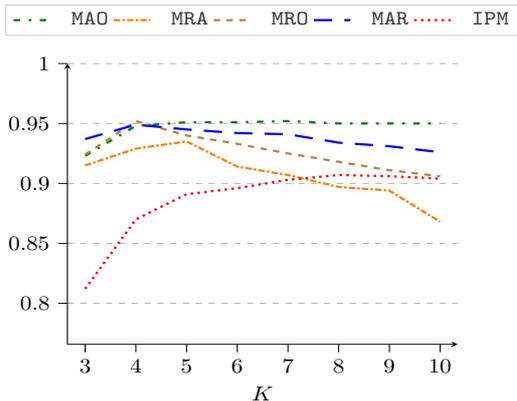}	
\end{subfigure}\hfill
\begin{subfigure}[h]{0.5\textwidth}
\centering
{\small
\begin{tabular}{@{}lc@{ }c@{ }c@{ }c@{ }c@{ }c@{ }c@{ }c@{ }cl@{}}\toprule
 &\multicolumn{8}{c}{$K$}\\\cmidrule(lr){2-9}
 &  3 & 4 & 5 & 6 & 7 & 8 & 9 & 10\\\midrule
\texttt{MAO} & 0.923 & 0.948 & 0.951 & 0.952 & 0.950 & 0.950 & 0.949 & 0.944\\
\texttt{MRA} & 0.915 & 0.929 & 0.935 & 0.914 & 0.907 & 0.897 & 0.894 & 0.860\\
\texttt{MRO} & 0.924 & 0.952 & 0.940 & 0.933 & 0.925 & 0.918 & 0.911 & 0.906\\
\texttt{MAR} & 0.937 & 0.949 & 0.945 & 0.942 & 0.941 & 0.934 & 0.931 & 0.926\\
\texttt{IPM} & 0.812 & 0.870 & 0.891 & 0.896 & 0.903 & 0.907 & 0.906 & 0.904\\\bottomrule
\end{tabular}}
\end{subfigure}
\vspace*{-.7cm}
\caption{Average \texttt{AvDi} values of the unconstrained formulations in random networks\label{fig:avdi_r_all}}
\end{figure}

We begin by analysing the results for the random networks.
Figure~\ref{fig:avdi_r_all} depicts the variation of the average \texttt{AvDi} for each model.
In general, the highest values of the average \texttt{AvDi} are associated to the code \texttt{MAO}, which is followed closely by \texttt{MAR} (the difference between the values associated to the two models does not exceed 2\%). Nevertheless, \texttt{MAR} surmounts \texttt{MAO} for $K=3$ and $K=4$. 
On the other hand, \texttt{IPM} has the worst performance with this regard, except for $K\ge8$, where it outperforms \texttt{MAR}.
In fact, the average dissimilarity obtained by \texttt{IPM} tends to improve when $K$ grows, whereas it tends  to worsen for all the formulations but \texttt{MAO}.

Figure \ref{fig:all_avdi_rn} summarizes the average \texttt{AvDi} results for each random instance.
The variation of the dissimilarities follows closely the pattern identified in Sections~\ref{sec:4.1}, \ref{sec:5.1} and \ref{sec:6.1} for the run times of the formulations: the instances recognised as harder to solve, namely $R_{100,500}$, are associated to the worst \texttt{AvDi} values. 

Figure~\ref{fig:avdi_r_all_q} allows a comparison of the dispersion of the results. The best scores are associated to \texttt{MAO}, followed  by \texttt{MAR} and \texttt{MRO}, whereas the \texttt{IPM} is the code with more disperse values.
In terms of the formulations, \texttt{MRA} was worse than the others.
The box-plots in Figures~\ref{fig:mao_avdi_rg} -- \ref{fig:mar_avdi_rg}, in Appendix \ref{app:1}, allow a thorough analysis of the dispersion of the average \texttt{AvDi} values for each formulation.
In general, the dispersion of the dissimilarities increases with $K$. It can also be concluded that, as expected, the hardest instances have smaller dissimilarities and a bigger dispersion of values.

The values of \texttt{MiDi} allow to study the worst case in terms of dissimilarity.
Figure~\ref{fig:midi_r_all} depicts a summary of the average \texttt{MiDi} for random networks.
In this case, the best results are found for the \texttt{MRO} and the \texttt{MAR} models. \texttt{IPM} is  worse than all formulations with this regard, while all the latter show very similar performances (the inner differences do not exceed 5\%). Also worthy of note, is the significant decrease of the \texttt{MiDi} values observed with the increase of $K$.

\begin{figure}[H]
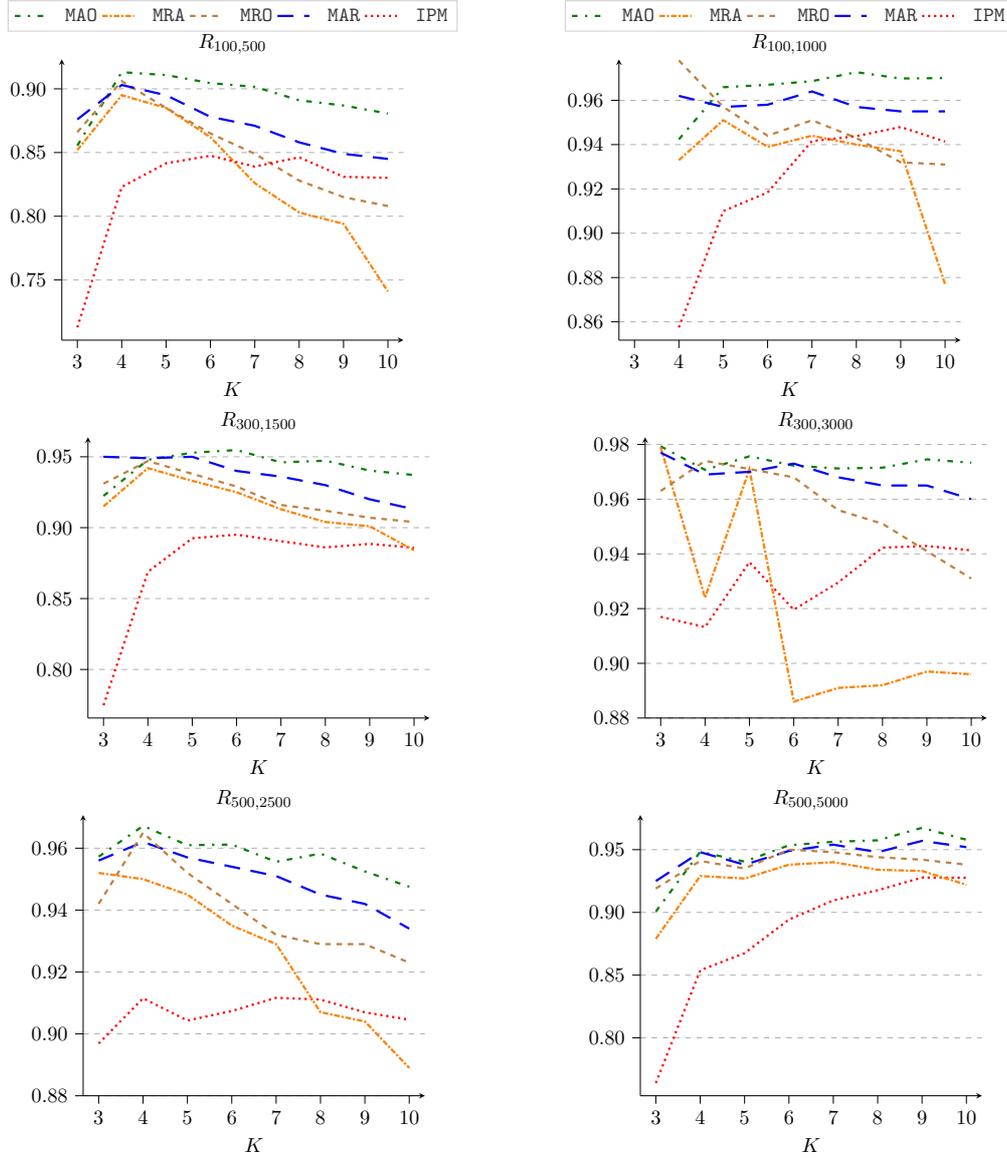

\centering
\begin{subfigure}[h]{0.45\textwidth}
	\centering
	\input{Comp/100-5-AvDi-Rn.tex}		
\end{subfigure}
\begin{subfigure}[h]{0.45\textwidth}
	\centering
	\input{Comp/100-10-AvDi-Rn.tex}		
\end{subfigure}
\begin{subfigure}[h]{0.45\textwidth}
	\centering
	\input{Comp/300-5-AvDi-Rn.tex}
\end{subfigure}
\begin{subfigure}[h]{0.45\textwidth}
	\centering
	\input{Comp/300-10-AvDi-Rn.tex}
\end{subfigure}
\begin{subfigure}[h]{0.45\textwidth}
	\centering
	\input{Comp/500-5-AvDi-Rn.tex}
\end{subfigure}
\begin{subfigure}[h]{0.45\textwidth}
	\centering
	\input{Comp/500-10-AvDi-Rn.tex}
\end{subfigure}
\vspace*{-.3cm}
\caption{Average dissimilarity of the unconstrained formulations in random networks\label{fig:all_avdi_rn}}
\end{figure}

\begin{figure}[htb]
\centering
\begin{subfigure}[h]{0.35\textwidth}
	\centering
	\input{Images/avdi_r_all2.tex}	
\end{subfigure}\hfill
\begin{subfigure}[h]{0.55\textwidth}
\centering
{\small
\begin{tabular}{@{}lccccccl@{}}\toprule
	  & \texttt{MAO} & \texttt{MRA} & \texttt{MRO} & \texttt{MAR}& \texttt{IPM}\\\midrule
     Min. value & 0.855 & 0.741 & 0.808 & 0.845 & 0.712\\
1$^{\rm st}$ quartile & 0.942 & 0.892 & 0.923 & 0.936 & 0.862\\
         Median & 0.956 & 0.922 & 0.938 & 0.950 & 0.904\\
3$^{\rm rd}$ quartile & 0.968 & 0.937 & 0.949 & 0.958 & 0.918\\
     Max. value & 0.979 & 0.979 & 0.978 & 0.977 & 0.947\\\bottomrule
\end{tabular}}
\end{subfigure}	
\vspace*{-.7cm}
\caption{\texttt{AvDi} dispersion of the unconstrained formulations in random networks\label{fig:avdi_r_all_q}}
\end{figure}

\begin{figure}[htb]
\centering
\begin{subfigure}[h]{0.4\textwidth}
\centering
	\input{Images/midi_r_all.tex}
\end{subfigure}\hfill
\begin{subfigure}[h]{0.6\textwidth}
\centering
{\small
\begin{tabular}{@{}lc@{ }c@{ }c@{ }c@{ }c@{ }c@{ }c@{ }c@{ }cl@{}}\toprule
 &\multicolumn{8}{c}{$K$}\\\cmidrule(lr){2-9}
 & 3 & 4 & 5 & 6 & 7 & 8 & 9 & 10\\\midrule
\texttt{MAO} & 0.864 & 0.804 & 0.790 & 0.761 & 0.731 & 0.717 & 0.696 & 0.687\\
\texttt{MRA} & 0.843 & 0.802 & 0.790 & 0.761 & 0.731 & 0.717 & 0.696 & 0.687\\
\texttt{MRO} & 0.836 & 0.851 & 0.803 & 0.785 & 0.752 & 0.731 & 0.706 & 0.689\\
\texttt{MAR} & 0.882 & 0.825 & 0.795 & 0.775 & 0.757 & 0.717 & 0.717 & 0.675\\
\texttt{IPM} & 0.610 & 0.505 & 0.523 & 0.502 & 0.488 & 0.446 & 0.425 & 0.378\\\bottomrule
\end{tabular}}
\end{subfigure}	
\vspace*{-.5cm}
\caption{Average \texttt{MiDi} values of the unconstrained formulations in random networks\label{fig:midi_r_all}}
\end{figure}

To summarize, the best and least disperse average \texttt{AvDi}'s are associated to the \texttt{MAO} model. However, its high run times undermine its application. In contrast, \texttt{MAR} produced solutions with good average and dispersion dissimilarities in less than 2 seconds, for all instances. Furthermore, the \texttt{MiDi} analysis indicates that \texttt{MAR} is less likely to produce solution with very poor dissimilarities.

\begin{figure}[htb]
\centering
\begin{subfigure}[h]{0.35\textwidth}
	\centering
	\input{Images/avdi_g_all.tex}	
\end{subfigure}\hfill
\begin{subfigure}[h]{0.55\textwidth}
\centering
{\small
\begin{tabular}{@{}lc@{ }c@{ }c@{ }c@{ }c@{ }c@{ }c@{ }c@{ }cl@{}}\toprule
 &\multicolumn{8}{c}{$K$}\\\cmidrule(lr){2-9}
 & 3 & 4 & 5 & 6 & 7 & 8 & 9 & 10\\\midrule
\texttt{MAO} & 0.959 & 0.930 & 0.887 & 0.874 & 0.854 & 0.845 & 0.833 & 0.827\\
\texttt{MRA}  & 0.959 & 0.907 & 0.753 & 0.684 & 0.657 & 0.563 & 0.516 & 0.490\\
\texttt{MRO} & 0.959 & 0.923 & 0.864 & 0.822 & 0.783 & 0.753 & 0.732 & 0.716\\
\texttt{MAR} & 0.959 & 0.926 & 0.879 & 0.842 & 0.809 & 0.784 & 0.752 & 0.751\\
\texttt{IPM} & 0.958 & 0.911 & 0.886 & 0.868 & 0.848 & 0.837 & 0.827 & 0.823\\\bottomrule
\end{tabular}}
\end{subfigure}	
\vspace*{-.7cm}
\caption{Average \texttt{AvDi} values of the unconstrained formulations in grid networks\label{fig:avdi_gd_all}}
\end{figure}

Next, the results for the grid networks are analyzed. As already mentioned, for some grid networks it is possible to know the optimal value of $D_1$. In fact, the length of the paths in a grid network $G_{p,q}$ is constant, namely $q+p-2$. Consequently, maximizing $D_1$ is equivalent to minimizing only the numerator of the fractions in its expression, which is precisely the goal of formulation (\ref{eqf3}). Thus, the values of \texttt{AvDi} for all the instances solved to optimality by \texttt{MAO} are optimal dissimilarities.
Figure~\ref{fig:avdi_gd_all} illustrates the main differences between the five codes.
Like what happened for the random networks, \texttt{MAO} produces the best results. However, this model is now followed by \texttt{IMP}, and the differences towards \texttt{MAR} have become wider and go as high as 10\% whereas they do not exceed 2\% for \texttt{IPM}. \texttt{MRA} still provides the worst results.
Despite the good dissimilarities obtained by \texttt{MAO}, it is worth noting that these are affected by the fact that most of these instances were not solved to optimality within the time limit.

Figure \ref{fig:all_avdi_gd} allows a more comprehensive comparison of the \texttt{AvDi} values produced by each code. The results highlight the greater difficulty of solving the square instances, as pointed out in sections \ref{sec:4.1}, \ref{sec:5.1} and \ref{sec:6.1}, and also some inconsistency in the values associated to \texttt{MRA}.
Figure~\ref{fig:avdi_gd_all_q} allows the analysis of the dispersion of these results.
\texttt{MAO} is the formulation with the more uniform values, followed by \texttt{IPM}. On the other hand, the models \texttt{MAR}, \texttt{MRO} and \texttt{MRA} present the smallest dispersion (in this order). Overall, the differences between the models lay in the low quartiles, thus \texttt{MAO} and \texttt{IPM} offer less chances of obtaining solutions with low values \texttt{AvDi} for this type of networks. The figures and the tables in Appendix \ref{app:1}, show that the trends discussed above regarding Figure \ref{fig:all_avdi_gd} are also present when studying the dissimilarity. 

\begin{figure}[htb]
\centering
\begin{subfigure}[h]{0.45\textwidth}
	\centering
	\input{Comp/6-6-AvDi-Gr.tex}		
\end{subfigure}
\begin{subfigure}[h]{0.45\textwidth}
	\centering
	\input{Comp/3-12-AvDi-Gr.tex}		
\end{subfigure}
\begin{subfigure}[h]{0.45\textwidth}
	\centering
	\input{Comp/12-12-AvDi-Gr.tex}		
\end{subfigure}
\begin{subfigure}[h]{0.45\textwidth}
	\centering
	\input{Comp/4-36-AvDi-Gr.tex}		
\end{subfigure}	
\vspace*{-.3cm}
\caption{Average dissimilarity of the unconstrained formulations in grid networks\label{fig:all_avdi_gd}}
\end{figure}

\begin{figure}[htb]
\centering
\begin{subfigure}[h]{0.35\textwidth}
	\centering
	\input{Images/avdi_g_all2.tex}	
\end{subfigure}\hfill
\begin{subfigure}[h]{0.55\textwidth}
\centering
{\small
\begin{tabular}{@{}lccccccl@{}}\toprule
	 & \texttt{MAO} & \texttt{MRA} & \texttt{MRO} & \texttt{MAR} & \texttt{IPM}\\\midrule
       Min. value & 0.730 & 0.185 & 0.542 & 0.607 & 0.728\\
1$^{st}$ quartile & 0.830 & 0.517 & 0.693 & 0.746 & 0.814\\
           Median & 0.879 & 0.779 & 0.879 & 0.879 & 0.879\\
3$^{rd}$ quartile & 0.936 & 0.935 & 0.936 & 0.936 & 0.933\\
       Max. value & 0.982 & 0.982 & 0.982 & 0.982 & 0.982\\\bottomrule
\end{tabular}}\end{subfigure}	
\vspace*{-.7cm}
\caption{\texttt{AvDi} dispersion of the unconstrained formulations in grid networks\label{fig:avdi_gd_all_q}}
\end{figure}

As shown in Figure~\ref{fig:midi_g_all}, there is no clear dominance regarding the values of \texttt{MiDi}. Nevertheless, the worst results are associated to the \texttt{IPM} for $K=3,4$ and to \texttt{MRA} model for $K\ge5$. It should be mentioned that, in all cases, the results follow heavily with the increase of $K$. In fact, the new formulations only look for good average dissimilarities between the pairs of paths in the solution, which may hide solutions with pairs of paths with very different dissimilarities.   

\begin{figure}[htb]
\centering
\begin{subfigure}[h]{0.4\textwidth}
	\centering
	\input{Images/midi_g_all.tex}	
\end{subfigure}\hfill
\begin{subfigure}[h]{0.6\textwidth}
\centering
{\small
\begin{tabular}{@{}lc@{ }c@{ }c@{ }c@{ }c@{ }c@{ }c@{ }c@{ }cl@{}}\toprule
 &\multicolumn{8}{c}{$K$}\\\cmidrule(lr){2-9}
 & 3 & 4 & 5 & 6 & 7 & 8 & 9 & 10\\\midrule
\texttt{MAO} & 0.919 & 0.824 & 0.592 & 0.469 & 0.408 & 0.394 & 0.356 & 0.320\\
\texttt{MRA} & 0.938 & 0.766 & 0.437 & 0.419 & 0.376 & 0.182 & 0.193 & 0.080\\
\texttt{MRO} & 0.920 & 0.805 & 0.539 & 0.501 & 0.464 & 0.410 & 0.342 & 0.334\\
\texttt{MAR} & 0.906 & 0.805 & 0.533 & 0.495 & 0.463 & 0.439 & 0.348 & 0.322\\
\texttt{IPM} & 0.911 & 0.633 & 0.524 & 0.517 & 0.483 & 0.353 & 0.314 & 0.314\\\bottomrule
\end{tabular}}
\end{subfigure}	
\vspace*{-.3cm}
\caption{Average \texttt{MiDi} of the unconstrained formulations in grid networks\label{fig:midi_g_all}}
\end{figure}

In conclusion, the dissimilarity results suggest the use of \texttt{MAO} model when dealing with grid networks. However, once again, its run times limit its application. On the other hand, \texttt{IPM} was also able of finding good solutions and has the strong advantage of running in little time. Therefore, \texttt{IPM} seems a sound approach for this specific type of networks. As to the remaining formulations, the \texttt{AvDi} values associated to \texttt{MAR} are smaller but close to the above mentioned methods (with an average difference of 5\%, to both). Moreover, the \texttt{MiDi} results favor \texttt{MAR}. This information together with the good run times associated to this method, suggest that \texttt{MAR} is also worth considering in this context.

\subsection{Constrained formulations\label{sec:8.2}}
As shown in Table~\ref{tab:MAavdi}, the models obtained by adding the constraints (\ref{eq1}) to formulations (\ref{eqf1}), (\ref{eqf2}) and (\ref{eqf4}) produce solutions with better dissimilarity results, with few exceptions. Furthermore, the increase in the dissimilarities is bigger when more paths need to be found and bigger in the grids than in the random networks. 
As expected, the most significant differences occurred to the pair \texttt{MRA}--\texttt{MRAA}, with dissimilarity improvements of up to 10\% for the random networks and up to 244\% for the grid networks. As to the pair \texttt{MRO}--\texttt{MROA}, Table~\ref{tab:MAavdi} indicates also an improvement of the average dissimilarity of the solutions. However, the differences were smaller both in the cases of the random and of the grid networks (up to 6\% in the first case and to 24\% in the later). There were no changes in the dissimilarities obtained on the square grids.
Finally, the smallest difference was found for \texttt{MAR} and \texttt{MARA}. Even so, the improvements on the rectangular grids were quite significant. Again, no differences were registered in the dissimilarity results for square grids. Detailed information about the three models can be found in Appendix~\ref{app:2}. 

\begin{table}[htb]
\caption{Average dissimilarity variation for \texttt{MRAA}, \texttt{MROA} and \texttt{MARA} (\%)\label{tab:MAavdi}}
\centering
\small{
\begin{tabular}{@{}l c@{ }c@{ }c@{ }c@{ }c@{ }c@{ }c@{ }c  c@{ }c@{ }c@{ }c@{ }c@{ }c@{ }c@{ }c  c@{ }c@{ }c@{ }c@{ }c@{ }c@{ }c@{ }c@{}}\toprule
 & \multicolumn{8}{c}{\texttt{MRAA}} & \multicolumn{8}{c}{\texttt{MROA}} & \multicolumn{8}{c}{\texttt{MARA}}\\\cmidrule(lr){2-9} \cmidrule(lr){10-17}\cmidrule(l){18-25}
 & \multicolumn{8}{c}{$K$} & \multicolumn{8}{c}{$K$} & \multicolumn{8}{c}{$K$}\\\cmidrule(lr){2-9} \cmidrule(lr){10-17}\cmidrule(l){18-25}
 $R_{n,m}$ &  3 &  4 &  5 &  6 &  7 & 8 & 9 & 10   & 3 &  4 &  5 &  6 &  7 & 8 & 9 & 10 & 3 &  4 &  5 &  6 &  7 & 8 & 9 & 10\\\midrule
 $R_{100,500}$ & 1 & 1 & 2 & 4 & 6 & 7 & 7 & 12 &
 0 & 1 & 2 & 4 & 3 & 6 & 6 & 5 &
$-1$ & 1 & 1 & 2 & 2 & 2 & 2 & 2\\
$R_{100,1000}$ & -- & 2 & 2 &  3 & 2 & 3 & 3 & 10 &
-- & 0 & 1 & 2 & 2 & 3 & 4 & 4 &
-- & 1 & 1 & 0 & 0 & 1 & 1 & 1\\\midrule
$R_{300,1500}$ &  1 & 0 & 2 &  2 & 3 & 4 & 4 & 5 &
0 & 0 & 2 & 3 & 3 & 3 & 3 & 3 &
$-2$ & 0 & 0 & 1 & 1 & 1 & 2 & 2\\
$R_{300,3000}$ &  0 & 5 & 1 & 10 & 9 & 9 & 9 & 8 &
 0 & $-1$ & 1 & 1 & 2 & 3 & 3 & 4 &
$-1$ &  1 & 1 & 1 & 1 & 1 & 1 & 1\\\midrule
$R_{500,2500}$ &  0 & 1 & 1 &  3 & 2 & 4 & 5 & 6 &
1 & 0 & 0 & 2 & 2 & 3 & 2 & 2 &
1 & 0 & 0 & 1 & 0 & 1 & 1 & 1\\
$R_{500,5000}$ &  2 & 2 & 1 &  1 & 2 & 2 & 3 & 4 &
1 & 1 & 1 & 0 & 1 & 1 & 2 & 2 &
0 & 0 & 1 & 1 & 0 & 1 & 1 & 1\\\bottomrule
\end{tabular}\hfill
\begin{tabular}{@{}l c@{ }c@{ }c@{ }c@{ }c@{ }c@{ }c@{ }c  c@{ }c@{ }c@{ }c@{ }c@{ }c@{ }c@{ }c  c@{ }c@{ }c@{ }c@{ }c@{ }c@{ }c@{ }c@{}}\toprule
    & \multicolumn{8}{c}{\texttt{MRAA}} & \multicolumn{8}{c}{\texttt{MROA}} & \multicolumn{8}{c}{\texttt{MARA}}\\\cmidrule(lr){2-9} \cmidrule(lr){10-17}\cmidrule(l){18-25}
    & \multicolumn{8}{c}{$K$} & \multicolumn{8}{c}{$K$} & \multicolumn{8}{c}{$K$}\\\cmidrule(lr){2-9} \cmidrule(lr){10-17}\cmidrule(l){18-25}
  $G_{p,q}$ &  3 &  4 &  5 &  6 &  7 & 8 & 9 & 10   & 3 &  4 &  5 &  6 &  7 & 8 & 9 & 10 & 3 &  4 &  5 &  6 &  7 & 8 & 9 & 10\\\midrule
 $G_{3,12}$ & 0 & 3 & 37 & 138 & 158 & 202 & 218 & 244 & 0 & 3 & 4 & 16 & 18 & 23 & 22 & 24 & 0 & 2 & 1 & 11 &  6 & 10 & 13 & 15 \\
 $G_{4,36}$ & 0 & 0 & 24 &  38 &  31 &  16 &  73 & 28  & 0 & 0 & 0 &  0 &  0 & 16 & 14 & 24 & 0 & 0 & 0 &  0 & $-2$ & 14 & 19 & 15 \\
  $G_{6,6}$ & 0 & 8 & 13 &   0 &   0 &  64 &  60 & 61  & 0 & 0 & 0 &  0 &  0 &  0 &  0 & 0 & 0 & 0 & 0 &  0 &  0 &  0 &  0 & 0 \\
$G_{12,12}$ & 0 & 0 &  0 &   0 &   0 &   1 &   0 & 21 & 0 & 0 & 0 &  0 &  0 &  0 &  0 & 0 & 0 & 0 & 0 &  0 &  0 &  0 &  0 & 0 \\ \bottomrule
\end{tabular}}
\end{table}

In the following, we analyze the trade-off between the improvements in the dissimilarities and the variations in the run time, identified in Section~\ref{sec:7.1}, to assess to what extent the constrained models are worth considering. The new models are compared to \texttt{IPM} and \texttt{MAO}.
The code \texttt{MROA} was ruled out of the study, because it follows closely \texttt{MARA} but always with some disadvantage.

\begin{figure}[htb]
\centering
\begin{subfigure}[h]{0.4\textwidth}
\centering
	\input{Images/avdi_ra_all.tex}	
\end{subfigure}\hfill
\begin{subfigure}[h]{0.6\textwidth}
\centering
{\small
\begin{tabular}{@{}lc@{ }c@{ }c@{ }c@{ }c@{ }c@{ }c@{ }c@{ }cl@{}}\toprule
 &\multicolumn{8}{c}{$K$}\\\cmidrule(lr){2-9}
 & 3 & 4 & 5 & 6 & 7 & 8 & 9 & 10\\\midrule
 \texttt{MAO} & 0.923 & 0.948 & 0.951 & 0.952 & 0.950 & 0.950 & 0.949 & 0.944\\
\texttt{MRAA} & 0.923 & 0.947 & 0.949 & 0.948 & 0.942 & 0.941 & 0.939 & 0.933\\
\texttt{MARA} & 0.929 & 0.954 & 0.950 & 0.952 & 0.948 & 0.945 & 0.943 & 0.940\\
\texttt{IPM} & 0.812 & 0.870 & 0.891 & 0.896 & 0.903  & 0.907 & 0.906  & 0.904\\\bottomrule
\end{tabular}}
\end{subfigure}	
\vspace*{-.3cm}
\caption{Average \texttt{AvDi} values of the constrained formulations in random networks\label{fig:avdi_ra_all}}
\end{figure}

\begin{figure}[htb]
\centering
\begin{subfigure}[h]{0.45\textwidth}
\centering
	\input{Images/avdi_ra_all2.tex}	
\end{subfigure}\hfill
\begin{subfigure}[h]{0.45\textwidth}
\centering
{\small
\begin{tabular}{@{}lccccl@{}}\toprule
	 & \texttt{MAO} & \texttt{MRAA} & \texttt{MARA} & \texttt{IPM}\\\midrule
	Min. value & 0.855 & 0.828 & 0.862  & 0.712\\
	$1^{st}$ quartile & 0.942 & 0.941 & 0.944 & 0.862\\
	Median  & 0.956 & 0.952 & 0.957  &0.904\\
	$3^{rd}$ quartile  & 0.968 & 0.967 & 0.966 & 0.918\\
	Max. value & 0.979 & 0.979 & 0.978 & 0.947\\\bottomrule
\end{tabular}}
\end{subfigure}	
\vspace*{-.3cm}
\caption{\texttt{AvDi} dispersion of the constrained formulations in random networks \label{fig:avdi_ra_all_q}}
\end{figure}

Figure~\ref{fig:avdi_ra_all} depicts the variation of the average \texttt{AvDi} for \texttt{MRAA}, \texttt{MRAA} and \texttt{IPM} in the random networks.
In spite of the improvement on the average dissimilarities associated to the \texttt{MRAA} when compared to \texttt{MRA} (see Figures \ref{fig:avdi_r_all} and \ref{fig:avdi_ra_all}), the constrained version of \texttt{MAR} is still the best, in average. The differences, however, are now very tenuous.

Comparing Figures \ref{fig:avdi_r_all_q} and \ref{fig:avdi_ra_all_q} reveals a reduction in the dispersion of the results of both \texttt{MARA} and \texttt{MRAA}. Nevertheless, \texttt{MARA} outperforms \texttt{MRAA} and performs even better than \texttt{MAO} regarding the worst cases. Again, \texttt{IPM} produces the poorest results. 

\begin{figure}[htb]
\centering
\begin{subfigure}[h]{0.4\textwidth}
	\centering
	\input{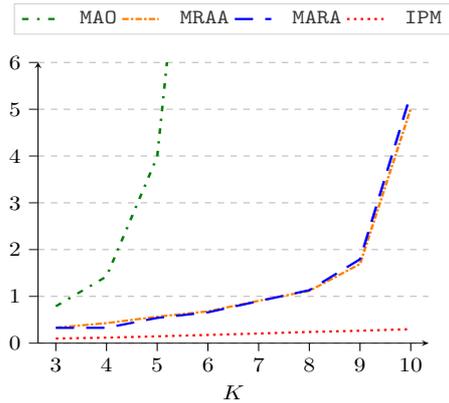}	
\end{subfigure}\hfill
\begin{subfigure}[h]{0.6\textwidth}
\centering
{\small
\begin{tabular}{@{}lc@{ }c@{ }c@{ }c@{ }c@{ }c@{ }c@{ }c@{ }cl@{}}\toprule
	&\multicolumn{8}{c}{$K$}\\\cmidrule(lr){2-9}
		& 3 & 4 & 5 & 6 & 7 & 8 & 9 & 10\\\midrule
	\texttt{MAO}  & 0.079 & 1.426 & 3.985 & 14.464 & 40.488 & 85.412 & 132.258 & 179.631\\
	\texttt{MRAA} & 0.327 & 0.329 & 0.540 & 0.659 & 0.898 & 1.133 & 1.794  & 5.320\\
	\texttt{MARA} & 0.333  & 0.429 & 0.566 & 0.681 & 0.906 & 1.119 & 1.693 & 5.014\\
	\texttt{IPM}  & 0.100 & 0.118 & 0.147 & 0.176 & 0.207 & 0.240 & 0.266 & 0.300\\\bottomrule
\end{tabular}}
\end{subfigure}	
\vspace*{-.3cm}
\caption{Run times of the constrained formulations in random networks (seconds)\label{fig:cpu_ra}}
\end{figure}

Figure \ref{fig:cpu_ra} reports the average run times of the models. As shown in Section \ref{sec:7.1}, adding the constraints (\ref{eq1}) affected \texttt{MAR} more than \texttt{MRA}. In fact, \texttt{MARA} and \texttt{MRAA} are now very close -- \texttt{MRAA} runs faster for the smaller networks and \texttt{MARA} for the larger. It is worth mentioning the increase of the run times from 1 to 5 seconds for the biggest instances. As expected, \texttt{IPM} is faster than the other codes.

\begin{figure}[htb]
\centering
\begin{subfigure}[h]{0.4\textwidth}
	\centering
	\input{Images/avdi_ga_all.tex}	
\end{subfigure}\hfill
\begin{subfigure}[h]{0.6\textwidth}
\centering
{\small
\begin{tabular}{@{}lc@{ }c@{ }c@{ }c@{ }c@{ }c@{ }c@{ }c@{ }cl@{}}\toprule
 &\multicolumn{8}{c}{$K$}\\\cmidrule(lr){2-9}
 & 3 & 4 & 5 & 6 & 7 & 8 & 9 & 10\\\midrule
 \texttt{MAO} & 0.959 & 0.930 & 0.887 & 0.874 & 0.854 & 0.845   & 0.833 & 0.827\\
\texttt{MRAA} & 0.959 & 0.930 & 0.871 & 0.838 & 0.813 & 0.808   & 0.779  & 0.775\\
\texttt{MARA} & 0.959 & 0.930 & 0.881 & 0.860 & 0.816 & 0.825   & 0.803  & 0.797\\
 \texttt{IPM} & 0.958 & 0.911 & 0.886 & 0.868 & 0.848 & 0.837 & 0.827 & 0.823\\\bottomrule
\end{tabular}}
\end{subfigure}	
\vspace*{-.3cm}
\caption{Average \texttt{AvDi} values of the constrained formulations in grid networks\label{fig:avdi_ga_all}}
\end{figure}
 
\begin{figure}[htb]
\centering
\begin{subfigure}[h]{0.45\textwidth}
\centering
	\input{Images/avdi_ga_all2.tex}	
\end{subfigure}\hfill
\begin{subfigure}[h]{0.45\textwidth}
\centering
{\small
\begin{tabular}{@{}lccccl@{}}\toprule
	 & \texttt{MAO} & \texttt{MRAA} & \texttt{MARA} & \texttt{IPM}\\\midrule
	Min. value & 0.730 & 0.641 & 0.686 & 0.728\\
	$1^{st}$ quartile & 0.830 & 0.737 & 0.783& 0.814\\
	Median  & 0.879 & 0.879 & 0.879& 0.879\\
	$3^{rd}$ quartile  & 0.936 & 0.936 & 0.936& 0.933\\
	Max. value & 0.982 & 0.982 & 0.982& 0.982\\\bottomrule
\end{tabular}}
\end{subfigure}	
\vspace*{-.5cm}
\caption{\texttt{AvDi} dispersion of the constrained formulations in grid networks\label{fig:avdi_ga_all_q}}
\end{figure}
 
In short, both \texttt{MRAA} and \texttt{MARA} seem tailored for this type of network: the average difference in the dissimilarities towards \texttt{MAO} is of 0.50\% for \texttt{MRAA} and of 0.08\% for \texttt{MARA}; as to the run times, \texttt{MRAA} is in advantage for $K\leq 7$ and \texttt{MARA} for the remaining values of $K$. However, taking into account the earlier good dissimilarities and run times of \texttt{MAR}, it not is clear whether its constrained model is the best option.

Figure~\ref{fig:avdi_ga_all} presents the average \texttt{AvDi} for \texttt{MRAA}, \texttt{MARA}, \texttt{IPM} and \texttt{MOA} in grid networks. In spite of the significant improvement for the pair \texttt{MRA}--\texttt{MRAA} (see Figure \ref{fig:avdi_gd_all}), \texttt{MARA} still produced solutions with higher average dissimilarities. In addition, it solved all the instances to optimality, which did not happen for \texttt{MRAA} (it was not possible to find 10 paths in $G_{4,36}$ within 300 seconds). Nevertheless, \texttt{MAO} provides the solutions with the best average dissimilarities, now followed closely by \texttt{IPM}.
On the other hand, there are reductions in the dispersion of the results for the pairs \texttt{MRA}--\texttt{MRAA} and \texttt{MAR}--\texttt{MARA}, when comparing Figures \ref{fig:avdi_gd_all_q} and \ref{fig:avdi_ga_all_q}.
Furthermore, as the most significant improvements occurred in the lower quartiles, it can be concluded that both \texttt{MRAA} and \texttt{MARA} are less likely to produce solutions with low dissimilarities than \texttt{MRA} and \texttt{MAR} for this type of network. Nevertheless, the best results are, again, associated to \texttt{MAO} and \texttt{IPM}.

\begin{figure}[htb]
\centering
\begin{subfigure}[h]{0.4\textwidth}
	\centering
	\input{Images/cpu_ga.tex}	
\end{subfigure}\hfill
\begin{subfigure}[h]{0.6\textwidth}
\centering
{\small
\begin{tabular}{@{}lc@{ }c@{ }c@{ }c@{ }c@{ }c@{ }c@{ }c@{ }cl@{}}\toprule
	&\multicolumn{8}{c}{$K$}\\\cmidrule(lr){2-9}
	& 3 & 4 & 5 & 6 & 7 & 8 & 9 & 10\\\midrule
	\texttt{MAO}  & 0.158 & 1.019 & 79.534 & 96.118 & 300    & 300   & 300   & 300\\
	\texttt{MRAA} & 0.055 & 0.059 & 0.667  & 5.111  & 11.194 & 1.702 & 3.270 & 76.915\\
	\texttt{MARA} & 0.043 & 0.054 & 0.101  & 0.093  & 0.233  & 0.155 & 0.260 & 0.404\\
	\texttt{IPM}  & 0.004 & 0.006 & 0.007  & 0.007  & 0.010  & 0.011 & 0.014 & 0.015\\\bottomrule
\end{tabular}}
\end{subfigure}	
\vspace*{-.3cm}
\caption{Run times for the constrained formulations in grid networks (seconds)\label{fig:cpu_g}}
\end{figure}

According to Figure \ref{fig:cpu_g}, \texttt{MRAA} is much slower than \texttt{MARA} for this type of instances, which is, in turn, much slower than \texttt{IPM}.
Thus, even though \texttt{MARA} has good run times it is clearly overtaken by \texttt{IPM}. As mentioned earlier, in general, the run times associated to \texttt{MAO} are very high. Tables~\ref{tab:MRAAcpu0} and \ref{tab:MARAcpu0} in Appendix \ref{app:2} give detailed run times of the \texttt{MRAA} and \texttt{MARA} models.

The previous discussion confirms \texttt{IPM} as the most suitable method for solving instances in grid networks. \texttt{IPM} is the fastest of the codes being compared and produced the solutions with the highest overall average dissimilarity. On the other hand, the results associated to \texttt{MARA} are also to be considered. In average the difference towards \texttt{MAO} is 2\%. This, together with the very reasonable run times, makes \texttt{MARA} also interesting for this type of network. In Section \ref{sec:8.1}  we also pointed out \texttt{MAR} as an interesting option for this type of network. However, comparing to \texttt{MARA}, we realize that \texttt{MAR} offers worse dissimilarities in the same amount of time.

\section{Concluding remarks\label{sec:9}}    
This manuscript addresses the search for $K$ dissimilar paths connecting a given pair of nodes in a directed graph.
Four integer formulations were introduced. The formulations have different motivations, but their general goal is to minimize the number of arcs that appear in more than one path or the total number of those overlaps, while searching for sets of paths with good dissimilarity.
The inclusion of an additional set of constraints to the previous formulations, with the goal of improving the solutions dissimilarity, was also proposed.
The performance of the new formulations and of a traditional method in the literature, the iterative penalty method, was tested  over random and grid networks, assessing the required run time as well as the average and the minimum dissimilarities of the solutions.
The dissimilarity of a given solution is the average of the pairwise dissimilarity between their paths, which is measured based on the ratios between the number of overlaps of the pair of paths and the length of the involved paths.
Three classes can be considered with this respect:
\begin{itemize}
\item
The code \texttt{IPM}, which not surprisingly is the fastest, given that it simply solves $K$ shortest path problems. Besides being the fastest code, \texttt{IPM} provided solutions with good dissimilarity for the grid networks. On the other hand, the solutions produced by this method for the random networks are very poor as to the dissimilarity of the $K$ paths.
\item
The code \texttt{MAO}, clearly the slowest of the five codes and often unable of finding an optimal solution within the established 300 seconds. In spite of this drawback, in general, \texttt{MAO} produced the solutions with the best dissimilarity, both for the random and the grid networks. 
\item
And the final group, composed of codes \texttt{MAR}, \texttt{MRAA} and \texttt{MARA}, which also provide solutions of good quality with regard to the dissimilarity. The three codes have similar behaviors for the random networks, however, for the grid networks, the latter unequivocally outperforms the others. Unlike \texttt{MRAA}, both \texttt{MAR} and \texttt{MARA} were able to find the optimum within the time bound for all instances. Furthermore, the run times of \texttt{MAR} and \texttt{MARA} are much smaller than the run times of the remaining codes in this group. 
\end{itemize}

It is worth noting that, the introduced formulations can still be extended to a case where the overlaps are penalized according to a given cost for the involved arcs, with no significant difference of the implementation.

In the future it would be interesting to investigate a simpler/smaller alternative formulation to \texttt{MAO}, accounting for the number of overlaps between each pair of paths, given that this is more in accordance with the dissimilarity measure that was used for assessing the solutions. 

Another important line of research, would be to clarify the relation between the dissimilarity and run time results of the formulations and some influencing characteristics of the networks, such as the density, average degree and topology. This aspect would allow a better use of the new models, particularly in the context of real world applications of the problem, for, in those cases, a wider variety of networks can arise. 
Finally, also the extension of the introduced formulations to bi-objective problems involving the minimization of a cost function besides the minimization of the arc repetitions seems to be of interest.

\paragraph{Acknowledgment}
This work was partially financially supported by the Portuguese Foundation for Science and Technology (FCT) under project grants UID/MAT/04561/2019, UID/MAT/\-00324/2020 and UID/MULTI/00308/2020.
The work was also partially supported by project P2020 SAICTPAC/0011/2015, co-financed by COMPETE 2020, Portugal 2020 -- Operational Program for Competitiveness and Internationalization (POCI), European Union's European Regional Development Fund, and FCT, and by FEDER Funds and National Funds under project CENTRO-01-0145-FEDER-029312.

\bibliographystyle{elsarticle-harv}
\bibliography{ref}

\appendix
\section{Dissimilarities for the unconstrained formulations\label{app:1}}
\begin{table}[htb]
\centering
\caption{Average \texttt{AvDi} and \texttt{MiDi} of \texttt{MAO} in random networks\label{tab:mao_AvDi_MiDi_r}}
{\scriptsize{
\begin{tabular}{@{}lcccccccccl@{}}\toprule
	\texttt{AvDi}&\multicolumn{9}{c}{$K$} \\\cmidrule(lr){2-9}
	 $R_{n,m}$ 	& 3 & 4 & 5 & 6 & 7 & 8 & 9 & 10\\\cmidrule(lr){1-9}
$R_{100,500}$   & 0.855  & 0.913  & 0.911   & 0.904  & 0.902  & 0.891  & 0.887  & 0.886  \\
$R_{100,1000}$  & --     & 0.942  & 0.965   & 0.967  & 0.968  & 0.973  & 0.970  & 0.970\\\cmidrule(lr){1-9}
$R_{300,1500}$  & 0.922  & 0.947  & 0.952   & 0.954  & 0.946  & 0.947  & 0.940  & 0.937  \\
$R_{300,3000}$  & 0.979  & 0.970  & 0.976   & 0.972  & 0.971  & 0.972  & 0.975  & 0.973\\\cmidrule(lr){1-9}
$R_{500,2500}$  & 0.957  & 0.967  & 0.961   & 0.961  & 0.956  & 0.958  & 0.952  & 0.947  \\
$R_{500,5000}$  & 0.900  & 0.948  & 0.940   & 0.953  & 0.956  & 0.957  & 0.967  & 0.958\\\midrule
Average         & 0.923  & 0.948  & 0.951   &  0.951 & 0.952  & 0.950  & 0.950  & 0.950  & \fbox{0.944}\\\bottomrule
\end{tabular}
\begin{tabular}{@{}lcccccccccl@{}}\toprule
		\texttt{MiDi}&\multicolumn{9}{c}{$K$} \\\cmidrule(lr){2-9}
	$R_{n,m}$	& 3 & 4 & 5 & 6 & 7 & 8 & 9 & 10\\\cmidrule(lr){1-9}
$R_{100,500}$   & 0.855  & 0.763  & 0.720   & 0.691  & 0.650  & 0.615   & 0.555  & 0.547  \\
$R_{100,1000}$  & --     & 0.654  & 0.750   & 0.740  & 0.700  & 0.698   & 0.682  & 0.684\\\cmidrule(lr){1-9}
$R_{300,1500}$  & 0.822  & 0.813  & 0.801   & 0.782  & 0.763  & 0.745   & 0.704  & 0.705  \\
$R_{300,3000}$  & 0.938  & 0.850  & 0.829   & 0.751  & 0.723  & 0.691   & 0.721  & 0.707\\\cmidrule(lr){1-9}
$R_{500,2500}$  & 0.887  & 0.882  & 0.822   & 0.814  & 0.774  & 0.792   & 0.730  & 0.720  \\
$R_{500,5000}$  & 0.819  & 0.862  & 0.819   & 0.794  & 0.778  & 0.760   & 0.780  & 0.756\\\midrule
Average         & 0.864  & 0.804  & 0.790   & 0.761  & 0.731  & 0.731   & 0.717  & 0.696  & \fbox{0.687}\\\bottomrule
\end{tabular}}}
\end{table}

\begin{figure}[htb]
	\centering
	\resizebox*{!}{3.5cm}{
		\input{Images/mao_avdi_r1.tex}		
		\input{Images/mao_avdi_r2.tex}		
		\input{Images/mao_avdi_g1.tex}		
		\input{Images/mao_avdi_g2.tex}		
	}
	\caption{\texttt{AvDi} dispersion of \texttt{MAO} \label{fig:mao_avdi_rg}}
\end{figure}

\begin{table}[htb]
\centering
\caption{Average \texttt{AvDi} and \texttt{MiDi} of \texttt{MRA} in random networks\label{tab:mra_AvDi_MiDi_r}}
{\scriptsize{
\begin{tabular}{@{}lcccccccccl@{}}\toprule
	\texttt{AvDi}&\multicolumn{9}{c}{$K$} \\\cmidrule(lr){2-9} 
	 $R_{n,m}$ 	& 3 & 4 & 5 & 6 & 7 & 8 & 9 & 10\\\cmidrule(lr){1-9}
$R_{100,500}$   & 0.852  & 0.895  & 0.885  & 0.862  & 0.826  & 0.803  & 0.794  & 0.741  \\
$R_{100,1000}$  & --     & 0.933  & 0.951  & 0.939  & 0.944  & 0.940  & 0.937  & 0.877\\\cmidrule(lr){1-9}
$R_{300,1500}$  & 0.915  & 0.942  & 0.933  & 0.925  & 0.913  & 0.904  & 0.901  & 0.884  \\
$R_{300,3000}$  & 0.979  & 0.924  & 0.971  & 0.886  & 0.891  & 0.892  & 0.897  & 0.896\\\cmidrule(lr){1-9}
$R_{500,2500}$  & 0.952  & 0.950  & 0.945  & 0.935  & 0.929  & 0.907  & 0.904  & 0.889  \\
$R_{500,5000}$  & 0.879  & 0.929  & 0.927  & 0.938  & 0.940  & 0.934  & 0.933  & 0.922\\\midrule
Average         & 0.915  & 0.929  & 0.935  & 0.914  & 0.907  & 0.897  & 0.894  & 0.868 & \fbox{0.909}\\\bottomrule
\end{tabular}
\begin{tabular}{@{}lcccccccccl@{}}\toprule
	\texttt{MiDi}&\multicolumn{9}{c}{$K$} \\\cmidrule(lr){2-9}
	$R_{n,m}$ & 3 & 4 & 5 & 6 & 7 & 8 & 9 & 10\\\cmidrule(lr){1-9}
$R_{100,500}$ & 0.772  & 0.735   & 0.720   & 0.666  & 0.603  & 0.532  & 0.513  & 0.416  \\
$R_{100,1000}$& --     & 0.726   & 0.746   & 0.753  & 0.743  & 0.727  & 0.723  & 0.652\\\cmidrule(lr){1-9}
$R_{300,1500}$& 0.817  & 0.824   & 0.815   & 0.766  & 0.750  & 0.729  & 0.695  & 0.642  \\
$R_{300,3000}$& 0.938  & 0.833   & 0.850   & 0.744  & 0.741  & 0.704  & 0.722  & 0.710\\\cmidrule(lr){1-9}
$R_{500,2500}$& 0.888  & 0.866   & 0.848   & 0.803  & 0.804  & 0.742  & 0.738  & 0.704  \\
$R_{500,5000}$& 0.798  & 0.831   & 0.811   & 0.806  & 0.795  & 0.791  & 0.766  & 0.734\\\midrule
Average       & 0.843  & 0.802   & 0.798   & 0.756  & 0.739  & 0.704  & 0.693  & 0.643  & \fbox{0.751}\\\bottomrule
\end{tabular}}}
\end{table}

\begin{figure}[htb]
	\centering
	\resizebox*{!}{3.5cm}{
		\input{Images/mra_avdi_r1.tex}		
		\input{Images/mra_avdi_r2.tex}		
		\input{Images/mra_avdi_g1.tex}		
		\input{Images/mra_avdi_g2.tex}		
	}
	\caption{\texttt{AvDi} dispersion of \texttt{MRA} \label{fig:mra_avdi_rg}}
\end{figure}

\begin{table}[htb]
\centering
\caption{Average \texttt{AvDi} and \texttt{MiDi} of \texttt{MRO} in random networks\label{tab:mro_AvDi_MiDi_r}}
{\scriptsize{
\begin{tabular}{@{}lcccccccccl@{}}\toprule
	\texttt{AvDi}&\multicolumn{9}{c}{$K$} \\\cmidrule(lr){2-9} 
    $R_{n,m}$ & 3 & 4 & 5 & 6 & 7 & 8 & 9 & 10\\\cmidrule(lr){1-9}
$R_{100,500}$ & 0.866  & 0.906  & 0.885  & 0.865  & 0.849  & 0.828  & 0.815  & 0.808  \\
$R_{100,1000}$& --     & 0.978  & 0.957  & 0.944  & 0.951  & 0.943  & 0.932  & 0.931\\\cmidrule(lr){1-9}
$R_{300,1500}$& 0.931  & 0.947  & 0.938  & 0.929  & 0.916  & 0.912  & 0.907  & 0.904  \\
$R_{300,3000}$& 0.963  & 0.974  & 0.971  & 0.968  & 0.956  & 0.951  & 0.941  & 0.931\\\cmidrule(lr){1-9}
$R_{500,2500}$& 0.942  & 0.965  & 0.952  & 0.942  & 0.932  & 0.929  & 0.929  & 0.923  \\
$R_{500,5000}$& 0.919  & 0.941  & 0.935  & 0.950  & 0.948  & 0.944  & 0.942  & 0.938\\\midrule
Average       & 0.924  & 0.952  & 0.940  & 0.933  & 0.925  & 0.918  & 0.911  & 0.906 & \fbox{0.926}\\\bottomrule
\end{tabular}
\begin{tabular}{@{}lcccccccccl@{}}\toprule
		\texttt{MiDi}&\multicolumn{9}{c}{$K$} \\\cmidrule(lr){2-9}
	$R_{n,m}$ & 3 & 4 & 5 & 6 & 7 & 8 & 9 & 10\\\cmidrule(lr){1-9}
$R_{100,500}$ & 0.763  & 0.758  & 0.719   & 0.674  & 0.632  & 0.577  & 0.533  & 0.525  \\
$R_{100,1000}$& --     & 0.869  & 0.757   & 0.749  & 0.755  & 0.711  & 0.703  & 0.710\\\cmidrule(lr){1-9}
$R_{300,1500}$& 0.835  & 0.836  & 0.792   & 0.789  & 0.746  & 0.740  & 0.707  & 0.687  \\
$R_{300,3000}$& 0.889  & 0.899  & 0.858   & 0.831  & 0.759  & 0.763  & 0.736  & 0.707\\\cmidrule(lr){1-9}
$R_{500,2500}$& 0.851  & 0.893  & 0.862   & 0.837  & 0.814  & 0.787  & 0.780  & 0.710  \\
$R_{500,5000}$& 0.841  & 0.849  & 0.832   & 0.833  & 0.805  & 0.809  & 0.777  & 0.793\\\midrule
Average       & 0.836  & 0.851  & 0.803   & 0.785  & 0.752  & 0.731  & 0.706  & 0.689  & \fbox{0.768}\\\bottomrule
\end{tabular}}}
\end{table}

\begin{figure}[htb]
	\centering
	\resizebox*{!}{3.5cm}{
		\input{Images/mro_avdi_r1.tex}		
		\input{Images/mro_avdi_r2.tex}		
		\input{Images/mro_avdi_g1.tex}		
		\input{Images/mro_avdi_g2.tex}		
	}
	\caption{\texttt{AvDi} dispersion of \texttt{MRO} \label{fig:mro_avdi_rg}}
\end{figure}

\begin{table}[htb]
\centering
\caption{Average \texttt{AvDi} and \texttt{MiDi} of \texttt{MAR} in random networks\label{tab:mar_AvDi_MiDi_r}}
{\scriptsize{
\begin{tabular}{@{}lcccccccccl@{}}\toprule
		\texttt{AvDi}&\multicolumn{9}{c}{$K$} \\\cmidrule(lr){2-9} 
 $R_{n,m}$ & 3 & 4 & 5 & 6 & 7 & 8 & 9 & 10\\\cmidrule(lr){1-9}
$R_{100,500}$  & 0.876  & 0.903  & 0.895  & 0.878  & 0.871  & 0.858  & 0.849  & 0.845\\
$R_{100,1000}$ & --     & 0.962  & 0.957  & 0.958  & 0.964  & 0.957  & 0.955  & 0.955\\\cmidrule(lr){1-9}
$R_{300,1500}$ & 0.950  & 0.949  & 0.950  & 0.940  & 0.936  & 0.930  & 0.920  & 0.913  \\
$R_{300,3000}$ & 0.977  & 0.969  & 0.970  & 0.973  & 0.968  & 0.965  & 0.965  & 0.960\\\cmidrule(lr){1-9}
$R_{500,2500}$ & 0.956  & 0.962  & 0.957  & 0.954  & 0.951  & 0.945  & 0.942  & 0.934  \\
$R_{500,5000}$ & 0.925  & 0.948  & 0.938  & 0.949  & 0.954  & 0.948  & 0.957  & 0.952\\\midrule
Average        & 0.937  & 0.949  & 0.945  & 0.942  & 0.941  & 0.934  & 0.931  & 0.926 & \fbox{0.938}\\\bottomrule
\end{tabular}
\begin{tabular}{@{}lcccccccccl@{}}\toprule
		\texttt{MiDi}&\multicolumn{9}{c}{$K$} \\\cmidrule(lr){2-9}
$R_{n,m}$ & 3 & 4 & 5 & 6 & 7 & 8 & 9 & 10\\\cmidrule(lr){1-9}
$R_{100,500}$  & 0.820  & 0.751  & 0.720  & 0.657  & 0.635  & 0.567  & 0.537  & 0.503  \\
$R_{100,1000}$ & --     & 0.773  & 0.739  & 0.725  & 0.727  & 0.717  & 0.732  & 0.685\\\cmidrule(lr){1-9}
$R_{300,1500}$ & 0.895  & 0.827  & 0.818  & 0.780  & 0.762  & 0.722  & 0.713  & 0.678  \\
$R_{300,3000}$ & 0.930  & 0.868  & 0.833  & 0.842  & 0.808  & 0.748  & 0.730  & 0.712\\\cmidrule(lr){1-9}
$R_{500,2500}$ & 0.891  & 0.869  & 0.841  & 0.836  & 0.806  & 0.774  & 0.779  & 0.716  \\
$R_{500,5000}$ & 0.875  & 0.866  & 0.820  & 0.813  & 0.806  & 0.777  & 0.809  & 0.754\\\midrule
Average        & 0.882  & 0.825  & 0.795  & 0.775  & 0.757  & 0.717  & 0.717  & 0.675  & \fbox{0.766}\\\bottomrule
\end{tabular}}}
\end{table}

\begin{figure}[htb]
	\centering
	\resizebox*{!}{3.5cm}{
		\input{Images/mar_avdi_r1.tex}		
		\input{Images/mar_avdi_r2.tex}		
		\input{Images/mar_avdi_g1.tex}		
		\input{Images/mar_avdi_g2.tex}		
	}
	\caption{\texttt{AvDi} dispersion of \texttt{MAR} \label{fig:mar_avdi_rg}}
\end{figure}

\begin{table}[htb]
\centering
\caption{Average \texttt{AvDi} and \texttt{MiDi} of \texttt{MAO} in grid networks\label{tab:mao_AvDi_MiDi_g_all}}
{\scriptsize{
\begin{tabular}{@{}lccccccccl@{}}\toprule
\texttt{AvDi} &\multicolumn{8}{c}{$K$}\\\cmidrule(lr){2-9}
	$G_{p,q}$	& 3 & 4 & 5 & 6 & 7 & 8 & 9 & 10\\\cmidrule(lr){1-9}
   	$G_{3,12}$  & 0.949  & 0.833  & 0.800 & 0.787  & 0.762  & 0.745  & 0.741  & 0.730  \\
   	$G_{4,36}$  & 0.982  & 0.982  & 0.892 & 0.859  & 0.848  & 0.845  & 0.819  & 0.811  \\
   	$G_{6,6}$   & 0.933  & 0.933  & 0.900 & 0.893  & 0.867  & 0.857  & 0.844  & 0.840  \\
   	$G_{12,12}$ & 0.970  & 0.970  & 0.955 & 0.952  & 0.939  & 0.935  & 0.929  & 0.927\\\midrule
   	Average     & 0.959  & 0.930  & 0.887 & 0.874  & 0.854  & 0.854  & 0.833  & 0.827 & \fbox{0.875}\\\bottomrule
\end{tabular}
\begin{tabular}{@{}lccccccccl@{}}\toprule
\texttt{MiDi} &\multicolumn{8}{c}{$K$}\\\cmidrule(lr){2-9}
	$G_{p,q}$  & 3 & 4 & 5 & 6 & 7 & 8 & 9 & 10\\\cmidrule(lr){1-9}
	$G_{3,12}$  & 0.846  & 0.538  & 0.231 & 0.154  & 0.154  & 0.154 & 0.000  & 0.000  \\
	$G_{4,36}$  & 0.974  & 0.947  & 0.526 & 0.158  & 0.105  & 0.105 & 0.105  & 0.053  \\
	$G_{6,6}$   & 0.900  & 0.900  & 0.700 & 0.700  & 0.600  & 0.500 & 0.500  & 0.500  \\
	$G_{12,12}$ & 0.955  & 0.909  & 0.909 & 0.864  & 0.773  & 0.818 & 0.818  & 0.727  \\\midrule
	Average     & 0.919  & 0.824  & 0.592 & 0.469  & 0.408  & 0.394 & 0.356  & 0.320 & \fbox{0.535}\\\bottomrule
\end{tabular}}}
\end{table}

\begin{table}[htb]
\centering
\caption{Average \texttt{AvDi} and \texttt{MiDi} of \texttt{MRA} in grid networks\label{tab:mra_AvDi_MiDi_g}}
{\scriptsize{
\begin{tabular}{@{}lcccccccccl@{}}\toprule
\texttt{AvDi} &\multicolumn{9}{c}{$K$}\\\cmidrule(lr){2-9}
	$G_{p,q}$ & 3 & 4 & 5 & 6 & 7 & 8 & 9 & 10\\\cmidrule(lr){1-9}
   	$G_{3,12}$  & 0.949  & 0.808 & 0.538  & 0.308  & 0.264  & 0.231  & 0.205  & 0.185  \\
   	$G_{4,36}$  & 0.982  & 0.982 & 0.718  & 0.584  & 0.560  & 0.570  & 0.414  & 0.506  \\
   	$G_{6,6}$   & 0.933  & 0.867 & 0.800  & 0.893  & 0.867  & 0.521  & 0.519  & 0.511  \\
   	$G_{12,12}$ & 0.970  & 0.970 & 0.955  & 0.952  & 0.939  & 0.929  & 0.924  & 0.758\\\midrule
   	Average     & 0.959  & 0.907 & 0.753  & 0.684  & 0.657  & 0.563  & 0.516  & 0.490 & \fbox{0.691}\\\bottomrule
\end{tabular}
\begin{tabular}{@{}lcccccccccl@{}}\toprule
\texttt{MiDi}&\multicolumn{9}{c}{$K$} \\\cmidrule(lr){2-9}
	$G_{p,q}$ & 3 & 4 & 5 & 6 & 7 & 8 & 9 & 10\\\cmidrule(lr){1-9}
	$G_{3,12}$  & 0.923  & 0.462 & 0.308 & 0.000  & 0.000  & 0.000  & 0.000  & 0.000  \\
	$G_{4,36}$  & 0.974  & 0.947 & 0.132 & 0.158  & 0.132  & 0.000  & 0.000  & 0.000  \\
	$G_{6,6}$   & 0.900  & 0.700 & 0.400 & 0.600  & 0.000  & 0.000  & 0.000  & 0.000  \\
	$G_{12,12}$ & 0.955  & 0.955 & 0.909 & 0.818  & 0.773  & 0.727  & 0.773  & 0.318  \\\midrule
	Average     & 0.938  & 0.766 & 0.437 & 0.419  & 0.376  & 0.182  & 0.193  & 0.080 & \fbox{0.424}\\\bottomrule
\end{tabular}}}
\end{table}

\begin{table}[htb]
\centering
\caption{Average \texttt{AvDi} and \texttt{MiDi} of \texttt{MRO} in grid networks\label{tab:mro_AvDi_MiDi_g}}
{\scriptsize{
\begin{tabular}{@{}lcccccccccl@{}}\toprule
\texttt{AvDi} &\multicolumn{9}{c}{$K$}\\\cmidrule(lr){2-9}
	$G_{p,q}$ & 3 & 4 & 5 & 6 & 7 & 8 & 9 & 10\\\cmidrule(lr){1-9}
   	$G_{3,12}$  & 0.949  & 0.808  & 0.708  & 0.641  & 0.597  & 0.569  & 0.551  & 0.542  \\
   	$G_{4,36}$  & 0.982  & 0.982  & 0.892  & 0.804  & 0.729  & 0.650  & 0.621  & 0.581   \\
   	$G_{6,6}$   & 0.933  & 0.933  & 0.900  & 0.893  & 0.867  & 0.857  & 0.833  & 0.822  \\
   	$G_{12,12}$ & 0.970  & 0.970  & 0.955  & 0.952  & 0.939  & 0.935  & 0.924  & 0.919\\\midrule
   	Average     & 0.959  & 0.927  & 0.864  & 0.822  & 0.783  & 0.753  & 0.732  & 0.716 & \fbox{0.819}\\\bottomrule
\end{tabular}
\begin{tabular}{@{}lcccccccccl@{}}\toprule
\texttt{MiDi}&\multicolumn{9}{c}{$K$} \\\cmidrule(lr){2-9}
	$G_{p,q}$ & 3 & 4 & 5 & 6 & 7 & 8 & 9 & 10\\\cmidrule(lr){1-9}
	$G_{3,12}$  & 0.923  & 0.538  & 0.462 & 0.308  & 0.308  & 0.308 & 0.308  & 0.231  \\
	$G_{4,36}$  & 0.947  & 0.974  & 0.132 & 0.132  & 0.132  & 0.158 & 0.132  & 0.132  \\
	$G_{6,6}$   & 0.900  & 0.800  & 0.700 & 0.700  & 0.600  & 0.400 & 0.200  & 0.200  \\
	$G_{12,12}$ & 0.909  & 0.909  & 0.864 & 0.864  & 0.818  & 0.773 & 0.727  & 0.773  \\\midrule
	Average     & 0.920  & 0.805  & 0.539 & 0.501  & 0.464  & 0.410 & 0.342  & 0.334 & \fbox{0.539}\\\bottomrule
\end{tabular}}}
\end{table}

\begin{table}[htb]
\centering
\caption{Average \texttt{AvDi} and \texttt{MiDi} of \texttt{MAR} in grid networks\label{tab:mar_AvDi_MiDi_g}}
{\scriptsize{
\begin{tabular}{@{}lcccccccccl@{}}\toprule
\texttt{AvDi} &\multicolumn{9}{c}{$K$}\\\cmidrule(lr){2-9}
	$G_{p,q}$	& 3 & 4 & 5 & 6 & 7 & 8 & 9 & 10\\\cmidrule(lr){1-9}
   	$G_{3,12}$ & 0.949 & 0.821 & 0.769 & 0.708 & 0.670 & 0.654 & 0.607 & 0.610  \\
   	$G_{4,36}$ & 0.982 & 0.982 & 0.892 & 0.814 & 0.759 & 0.690 & 0.643 & 0.652   \\
   	$G_{6,6}$  & 0.933 & 0.933 & 0.900 & 0.893 & 0.867 & 0.857 & 0.833 & 0.822  \\
   	$G_{12,12}$& 0.970 & 0.970 & 0.955 & 0.952 & 0.939 & 0.935 & 0.924 & 0.919\\\midrule
   	Average    & 0.959 & 0.926 & 0.879 & 0.842 & 0.809 & 0.784 & 0.752 & 0.751 & \fbox{0.839}\\\bottomrule
\end{tabular}
\begin{tabular}{@{}lcccccccccl@{}}\toprule
\texttt{MiDi}&\multicolumn{9}{c}{$K$} \\\cmidrule(lr){2-9}
	$G_{p,q}$ & 3 & 4 & 5 & 6 & 7 & 8 & 9 & 10\\\cmidrule(lr){1-9}
	$G_{3,12}$ & 0.923  & 0.538 & 0.231 & 0.385 & 0.231 & 0.308 & 0.308  & 0.231  \\
	$G_{4,36}$ & 0.947  & 0.974 & 0.237 & 0.132 & 0.158 & 0.132 & 0.158  & 0.132  \\
	$G_{6,6}$  & 0.800  & 0.800 & 0.600 & 0.600 & 0.600 & 0.500 & 0.200  & 0.200  \\
	$G_{12,12}$& 0.955  & 0.909 & 0.864 & 0.864 & 0.864 & 0.818 & 0.727  & 0.727  \\\midrule
	Average    & 0.906  & 0.805 & 0.533 & 0.495 & 0.463 & 0.439 & 0.348  & 0.322 & \fbox{0.539}\\\bottomrule
\end{tabular}}}
\end{table}

\cleardoublepage
\section{Dissimilarities and run times for the constrained formulations\label{app:2}}
\begin{table}[htb]
\centering
\caption{Average \texttt{AvDi} and \texttt{MiDi} of \texttt{MRAA} in random networks\label{tab:mraa_AvDi_MiDi_r}}
{\scriptsize{
\begin{tabular}{@{}lcccccccccl@{}}\toprule
	\texttt{AvDi}&\multicolumn{9}{c}{$K$} \\\cmidrule(lr){2-9} 
	 $R_{n,m}$ 	& 3 & 4 & 5 & 6 & 7 & 8 & 9 & 10\\\cmidrule(lr){1-9}
$R_{100,500}$   & 0.864  & 0.910  & 0.907  & 0.893  & 0.873  & 0.859  & 0.851  & 0.828  \\
$R_{100,1000}$  & --     & 0.951  & 0.968  & 0.967  & 0.968  & 0.968  & 0.968  & 0.966\\\cmidrule(lr){1-9}
$R_{300,1500}$  & 0.923  & 0.943  & 0.952  & 0.947  & 0.941  & 0.943  & 0.934  & 0.932  \\
$R_{300,3000}$  & 0.979  & 0.973  & 0.977  & 0.972  & 0.968  & 0.973  & 0.974  & 0.971\\\cmidrule(lr){1-9}
$R_{500,2500}$  & 0.952  & 0.961  & 0.955  & 0.959  & 0.949  & 0.948  & 0.948  & 0.942  \\
$R_{500,5000}$  & 0.896  & 0.944  & 0.939  & 0.952  & 0.954  & 0.954  & 0.960  & 0.958\\\midrule
Average         & 0.923  & 0.947  & 0.949  & 0.948  & 0.942  & 0.941  & 0.939  & 0.933 & \fbox{0.941}\\\bottomrule
\end{tabular}
\begin{tabular}{@{}lcccccccccl@{}}\toprule
		\texttt{MiDi}&\multicolumn{9}{c}{$K$} \\\cmidrule(lr){2-10}
	$R_{n,m}$	& 3 & 4 & 5 & 6 & 7 & 8 & 9 & 10\\\cmidrule(lr){1-9}
$R_{100,500}$   & 0.774  & 0.742  & 0.723   & 0.669  & 0.625  & 0.582  & 0.515  & 0.498  \\
$R_{100,1000}$  & --     & 0.704  & 0.754   & 0.743  & 0.720  & 0.680  & 0.675  & 0.670\\\cmidrule(lr){1-9}
$R_{300,1500}$  & 0.811  & 0.790  & 0.799   & 0.766  & 0.747  & 0.730  & 0.702  & 0.701  \\
$R_{300,3000}$  & 0.938  & 0.855  & 0.837   & 0.738  & 0.719  & 0.709  & 0.728  & 0.707\\\cmidrule(lr){1-9}
$R_{500,2500}$  & 0.888  & 0.860  & 0.819   & 0.790  & 0.769  & 0.743  & 0.739  & 0.680  \\
$R_{500,5000}$  & 0.798  & 0.841  & 0.810   & 0.802  & 0.803  & 0.766  & 0.771  & 0.761\\\midrule
Average         & 0.842  & 0.799  & 0.790   & 0.752  & 0.731  & 0.702  & 0.688  & 0.670  & \fbox{0.745}\\\bottomrule
\end{tabular}}}
\end{table}

\begin{table}[htb]
\centering
\caption{Average \texttt{AvDi} and \texttt{MiDi} of \texttt{MARA} in random networks\label{tab:mara_AvDi_MiDi_r}}
{\scriptsize{
\begin{tabular}{@{}lcccccccccl@{}}\toprule
	\texttt{AvDi}&\multicolumn{9}{c}{$K$} \\\cmidrule(lr){2-9} 
	 $R_{n,m}$ 	& 3 & 4 & 5 & 6 & 7 & 8 & 9 & 10\\\cmidrule(lr){1-9}
$R_{100,500}$   & 0.867  & 0.911  & 0.904  & 0.900  & 0.885  & 0.873  & 0.864  & 0.862\\
$R_{100,1000}$  & --     & 0.970  & 0.962  & 0.962  & 0.969  & 0.969  & 0.969  & 0.965\\\cmidrule(lr){1-9}
$R_{300,1500}$  & 0.929  & 0.951  & 0.949  & 0.951  & 0.944  & 0.941  & 0.935  & 0.933\\
$R_{300,3000}$  & 0.966  & 0.977  & 0.976  & 0.979  & 0.976  & 0.976  & 0.972  & 0.973\\\cmidrule(lr){1-9}
$R_{500,2500}$  & 0.960  & 0.967  & 0.961  & 0.963  & 0.955  & 0.955  & 0.950  & 0.947\\
$R_{500,5000}$  & 0.922  & 0.948  & 0.948  & 0.960  & 0.957  & 0.958  & 0.966  & 0.961\\\midrule
Average         & 0.929  & 0.954  & 0.950  & 0.952  & 0.948  & 0.945  & 0.943  & 0.940 & \fbox{0.945}\\\bottomrule
\end{tabular}
\begin{tabular}{@{}lcccccccccl@{}}\toprule
	\texttt{MiDi}&\multicolumn{9}{c}{$K$} \\\cmidrule(lr){2-9}
	$R_{n,m}$ & 3 & 4 & 5 & 6 & 7 & 8 & 9 & 10\\\cmidrule(lr){1-9}
$R_{100,500}$   & 0.794  & 0.752 & 0.718 & 0.671  & 0.636  & 0.582  & 0.540  & 0.497  \\
$R_{100,1000}$  & --     & 0.819 & 0.731 & 0.719  & 0.740  & 0.706  & 0.710  & 0.671\\\cmidrule(lr){1-9}
$R_{300,1500}$  & 0.832  & 0.819 & 0.793 & 0.770  & 0.763  & 0.721  & 0.714  & 0.697  \\
$R_{300,3000}$  & 0.899  & 0.883 & 0.833 & 0.822  & 0.785  & 0.752  & 0.707  & 0.705\\\cmidrule(lr){1-9}
$R_{500,2500}$  & 0.905  & 0.867 & 0.845 & 0.831  & 0.793  & 0.759  & 0.756  & 0.723  \\
$R_{500,5000}$  & 0.867  & 0.848 & 0.803 & 0.836  & 0.802  & 0.764  & 0.796  & 0.755\\\midrule
Average         & 0.859  & 0.831 & 0.787 & 0.775  & 0.753  & 0.714  & 0.714  & 0.675  & \fbox{0.760}\\\bottomrule
\end{tabular}}}
\end{table}

\begin{table}[htb]
\centering
\caption{Run times of \texttt{MRAA} (seconds)\label{tab:MRAAcpu0}}
{\scriptsize{
\begin{tabular}{@{}lccccccccl@{}}\toprule
       &\multicolumn{8}{c}{$K$} \\\cmidrule(lr){2-9}
     $R_{n,m}$ & 3 & 4 & 5 & 6 & 7 & 8 & 9 & 10\\\cmidrule(lr){1-9}
$R_{100,500}$  & 0.068 & 0.095 & 0.135 & 0.216 & 0.301 & 0.564 & 0.815 & 1.048\\
$R_{100,1000}$ & --    & 0.177 & 0.218 & 0.261 & 0.353 & 0.499 & 0.594 & 0.601\\\cmidrule(lr){1-9}
$R_{300,1500}$ & 0.204 & 0.285 & 0.372 & 0.451 & 0.637 & 0.741 & 1.028 & 1.702\\
$R_{300,3000}$ & 0.414 & 0.561 & 0.682 & 0.805 & 1.039 & 1.241 & 1.455 & 6.490\\\cmidrule(lr){1-9}
$R_{500,2500}$ & 0.382 & 0.480 & 0.661 & 0.787 & 1.131 & 1.468 & 1.625 & 2.109\\
$R_{500,5000}$ & 0.567 & 0.755 & 1.171 & 1.432 & 1.924 & 2.283 & 5.247 & 19.440\\\midrule
Average        & 0.327 & 0.329 & 0.540 & 0.659 & 0.898 & 1.133 & 1.794 & 5.320  & \fbox{1.394}\\\bottomrule
\end{tabular}
\begin{tabular}{@{}lccccccccl@{}}\toprule
       &\multicolumn{8}{c}{$K$} \\\cmidrule(lr){2-9}
  $G_{p,q}$ & 3 & 4 & 5 & 6 & 7 & 8 & 9 & 10\\\cmidrule(lr){1-9}
 $G_{3,12}$ & 0.047 & 0.024 & 0.137 & 0.253  & 0.272  & 0.514 & 0.842 & 0.905\\
 $G_{4,36}$ & 0.078 & 0.117 & 0.197 & 19.864 & 42.651 & 3.999 & 4.611 & 300\\
  $G_{6,6}$ & 0.022 & 0.012 & 0.103 & 0.097  & 0.370  & 0.493 & 2.487 & 1.093\\
$G_{12,12}$ & 0.073 & 0.082 & 0.456 & 0.227  & 1.484  & 1.802 & 5.142 & 5.582\\\midrule
Average     & 0.055 & 0.059 & 0.667 & 5.111  & 11.194 & 1.702 & 3.270 & 76.915 & \fbox{12.371}\\\bottomrule
\end{tabular}}}
\end{table}

\begin{table}[htb]
\centering
\caption{Run times of \texttt{MARA} (seconds)\label{tab:MARAcpu0}}
{\scriptsize{
\begin{tabular}{@{}lccccccccl@{}}\toprule
       &\multicolumn{8}{c}{$K$} \\\cmidrule(lr){2-9}
     $R_{n,m}$ & 3 & 4 & 5 & 6 & 7 & 8 & 9 & 10\\\cmidrule(lr){1-9}
$R_{100,500}$  & 0.073 & 0.112 & 0.156 & 0.206 & 0.313 & 0.422 & 0.508 & 0.639\\
$R_{100,1000}$ & --    & 0.178 & 0.245 & 0.267 & 0.341 & 0.395 & 0.466 & 0.588\\\cmidrule(lr){1-9}
$R_{300,1500}$ & 0.220 & 0.309 & 0.370 & 0.466 & 0.599 & 0.710 & 0.806 & 1.021\\
$R_{300,3000}$ & 0.398 & 0.517 & 0.671 & 0.787 & 1.075 & 1.496 & 1.567 & 1.814\\\cmidrule(lr){1-9}
$R_{500,2500}$ & 0.398 & 0.517 & 0.671 & 0.787 & 1.075 & 1.496 & 1.567 & 1.814\\
$R_{500,5000}$ & 0.588 & 0.891 & 1.276 & 1.540 & 2.055 & 2.447 & 5.342 & 19.563\\\midrule
Average        & 0.333 & 0.429 & 0.566 & 0.681 & 0.906 & 1.119 & 1.693 & 5.014  & \fbox{1.364}\\\bottomrule
\end{tabular}
\begin{tabular}{@{}lccccccccl@{}}\toprule
       &\multicolumn{8}{c}{$K$} \\\cmidrule(lr){2-9}
  $G_{p,q}$ & 3 & 4 & 5 & 6 & 7 & 8 & 9 & 10\\\cmidrule(lr){1-9}
 $G_{3,12}$ & 0.016 & 0.022 & 0.022 & 0.027 & 0.029 & 0.036 & 0.039 & 0.092\\
 $G_{4,36}$ & 0.074 & 0.112 & 0.153 & 0.188 & 0.231 & 0.276 & 0.357 & 0.377\\
  $G_{6,6}$ & 0.017 & 0.016 & 0.024 & 0.025 & 0.116 & 0.030 & 0.083 & 0.036\\
$G_{12,12}$ & 0.063 & 0.065 & 0.204 & 0.130 & 0.558 & 0.277 & 0.561 & 1.111\\\midrule
Average     & 0.043 & 0.054 & 0.101 & 0.093 & 0.233 & 0.155 & 0.260 & 0.404 & \fbox{1.364}\\\bottomrule
\end{tabular}}}
\end{table}

\begin{table}[htb]
\centering
\caption{Average \texttt{AvDi} and \texttt{MiDi} of \texttt{MRAA} in grid networks\label{tab:mraa_AvDi_MiDi_g}}
{\scriptsize{
\begin{tabular}{@{}lcccccccccl@{}}\toprule
\texttt{AvDi} &\multicolumn{9}{c}{$K$}\\\cmidrule(lr){2-9}
	$G_{p,q}$ & 3 & 4 & 5 & 6 & 7 & 8 & 9 & 10\\\cmidrule(lr){1-9}
   	$G_{3,12}$  & 0.949  & 0.833 & 0.738 & 0.703  & 0.714  & 0.681 & 0.656  & 0.641  \\
   	$G_{4,36}$  & 0.982  & 0.982 & 0.892 & 0.804  & 0.732  & 0.759 & 0.702  & 0.718  \\
   	$G_{6,6}$   & 0.933  & 0.933 & 0.900 & 0.893  & 0.867  & 0.857 & 0.833  & 0.822  \\
   	$G_{12,12}$ & 0.970  & 0.970 & 0.955 & 0.952  & 0.939  & 0.935 & 0.924  & 0.919\\\midrule
   	Average     & 0.959  & 0.930 & 0.871 & 0.838  & 0.813  & 0.808 & 0.779  & 0.775 & \fbox{0.847}\\\bottomrule
\end{tabular}
\begin{tabular}{@{}lcccccccccl@{}}\toprule
\texttt{MiDi}&\multicolumn{9}{c}{$K$} \\\cmidrule(lr){2-9}
	$G_{p,q}$ & 3 & 4 & 5 & 6 & 7 & 8 & 9 & 10\\\cmidrule(lr){1-9}
	$G_{3,12}$  & 0.923  & 0.538 & 0.231  & 0.231  & 0.154  & 0.154  & 0.000  & 0.000  \\
	$G_{4,36}$  & 0.974  & 0.947 & 0.237  & 0.053  & 0.105  & 0.105  & 0.079  & 0.105  \\
	$G_{6,6}$   & 0.900  & 0.800 & 0.800  & 0.700  & 0.500  & 0.500  & 0.300  & 0.200  \\
	$G_{12,12}$ & 0.955  & 0.909 & 0.818  & 0.818  & 0.727  & 0.727  & 0.727  & 0.682  \\\midrule
	Average     & 0.938  & 0.805 & 0.521  & 0.450  & 0.372  & 0.372  & 0.277  & 0.247 & \fbox{0.498}\\\bottomrule
\end{tabular}}}
\end{table}

\begin{table}[htb]
\centering
\caption{Average \texttt{AvDi} and \texttt{MiDi} of \texttt{MARA} in grid networks\label{tab:mara_AvDi_MiDi_g}}
{\scriptsize{
\begin{tabular}{@{}lcccccccccl@{}}\toprule
\texttt{AvDi} &\multicolumn{9}{c}{$K$}\\\cmidrule(lr){2-9}
	$G_{p,q}$	& 3 & 4 & 5 & 6 & 7 & 8 & 9 & 10\\\cmidrule(lr){1-9}
   	$G_{3,12}$  & 0.949  & 0.833  & 0.777 & 0.785 & 0.711 & 0.720 & 0.686 & 0.699  \\
   	$G_{4,36}$  & 0.982  & 0.982  & 0.982 & 0.812 & 0.747 & 0.787 & 0.767 & 0.748  \\
   	$G_{6,6}$   & 0.933  & 0.933  & 0.900 & 0.893 & 0.867 & 0.857 & 0.833 & 0.822  \\
   	$G_{12,12}$ & 0.970  & 0.970  & 0.955 & 0.952 & 0.939 & 0.935 & 0.924 & 0.919\\\midrule
   	Average     & 0.959  & 0.930  & 0.881 & 0.860 & 0.816 & 0.825 & 0.803 & 0.797 & \fbox{0.859}\\\bottomrule
\end{tabular}
\begin{tabular}{@{}lcccccccccl@{}}\toprule
\texttt{MiDi}&\multicolumn{9}{c}{$K$} \\\cmidrule(lr){2-9}
	$G_{p,q}$	& 3 & 4 & 5 & 6 & 7 & 8 & 9 & 10\\\cmidrule(lr){1-9}
	$G_{3,12}$  & 0.846 & 0.308 & 0.308 & 0.154  & 0.231  & 0.154  & 0.154 & 0.154  \\
	$G_{4,36}$  & 0.974 & 0.947 & 0.211 & 0.184  & 0.079  & 0.184  & 0.105 & 0.132  \\
	$G_{6,6}$   & 0.800 & 0.800 & 0.800 & 0.700  & 0.700  & 0.500  & 0.200 & 0.200  \\
	$G_{12,12}$ & 0.909 & 0.955 & 0.909 & 0.818  & 0.727  & 0.818  & 0.727 & 0.773  \\\midrule
	Average     & 0.876 & 0.752 & 0.557 & 0.464  & 0.434  & 0.414  & 0.297 & 0.315 & \fbox{0.514}\\\bottomrule
\end{tabular}}}
\end{table}

\begin{figure}[htb]
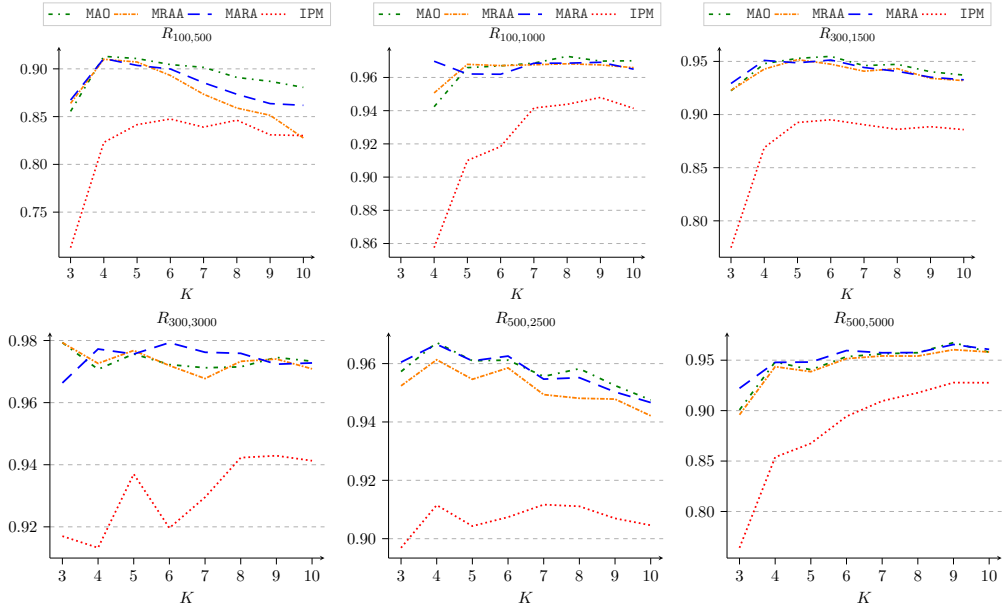

\centering
\resizebox*{!}{4cm}{
	\input{Comp/100-5-AvDi-Rn-TS2.tex}		
	\input{Comp/100-10-AvDi-Rn-TS2.tex}		
	\input{Comp/300-5-AvDi-Rn-TS2.tex}
}

\resizebox*{!}{4cm}{
	\input{Comp/300-10-AvDi-Rn-TS2.tex}
	\input{Comp/500-5-AvDi-Rn-TS2.tex}
	\input{Comp/500-10-AvDi-Rn-TS2.tex}
}
\caption{Average dissimilarity in random networks\label{fig:ipm_avdi_rn}}
\end{figure}

\begin{figure}[htb]
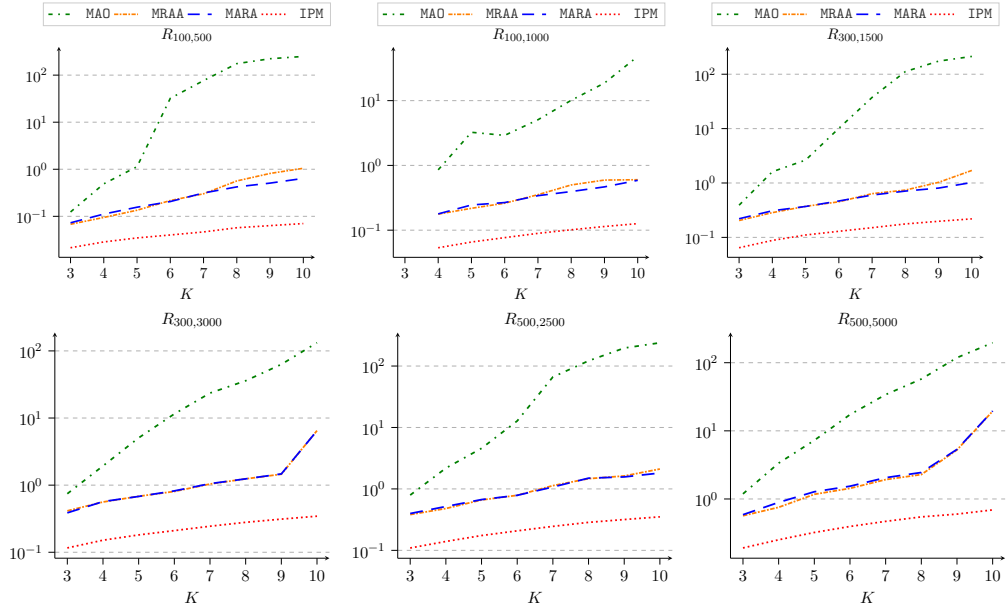

\centering
\resizebox*{!}{4cm}{
	\input{Comp/100-5-CPUS-Rn-TS2.tex}		
	\input{Comp/100-10-CPUS-Rn-TS2.tex}		
	\input{Comp/300-5-CPUS-Rn-TS2.tex}
}

\resizebox*{!}{4cm}{
	\input{Comp/300-10-CPUS-Rn-TS2.tex}
	\input{Comp/500-5-CPUS-Rn-TS2.tex}
	\input{Comp/500-10-CPUS-Rn-TS2.tex}
}
\caption{Run times in random networks (seconds -- log scale)\label{fig:ipm_cpus_rn}}
\end{figure}	

\begin{figure}[htb]
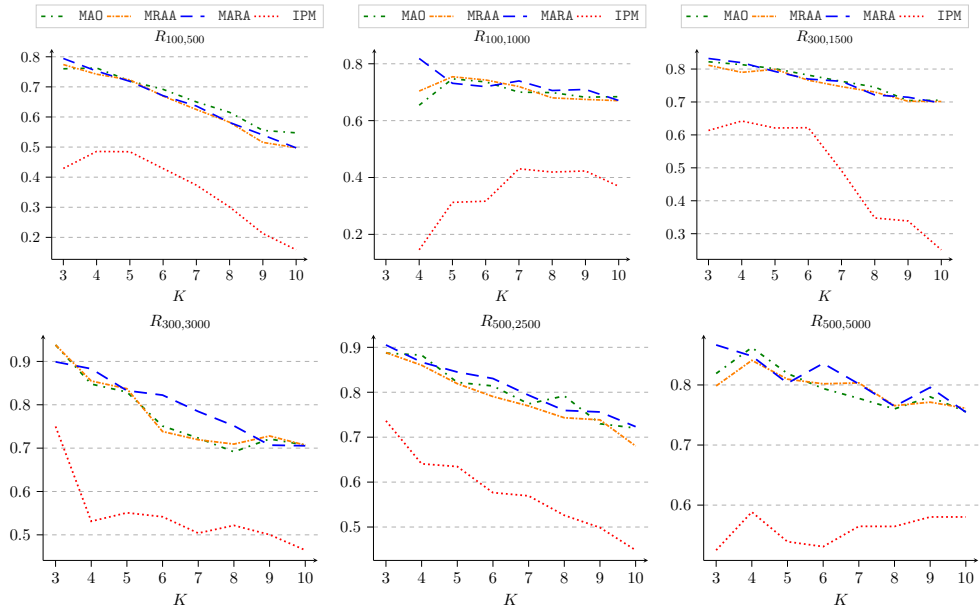

\centering
\resizebox*{!}{4cm}{
	\input{Comp/100-5-MiDi-Rn-TS2.tex}		
	\input{Comp/100-10-MiDi-Rn-TS2.tex}		
	\input{Comp/300-5-MiDi-Rn-TS2.tex}
}

\resizebox*{!}{4cm}{
	\input{Comp/300-10-MiDi-Rn-TS2.tex}
	\input{Comp/500-5-MiDi-Rn-TS2.tex}
	\input{Comp/500-10-MiDi-Rn-TS2.tex}
}
\caption{Minimum dissimilarity in random networks\label{fig:ipm_midi_rn}}
\end{figure}

\begin{figure}[htb]
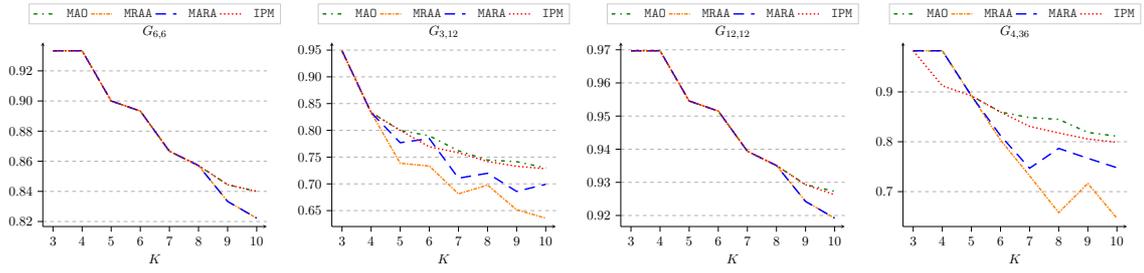

\centering
\resizebox*{!}{3.5cm}{
	\input{Comp/6-6-AvDi-Gr-TS2.tex}		
	\input{Comp/3-12-AvDi-Gr-TS2.tex}		
	\input{Comp/12-12-AvDi-Gr-TS2.tex}		
	\input{Comp/4-36-AvDi-Gr-TS2.tex}		
}
\caption{Average dissimilarity in grid networks\label{fig:ipm_avdi_gr}}
\end{figure}

\begin{figure}[htb]
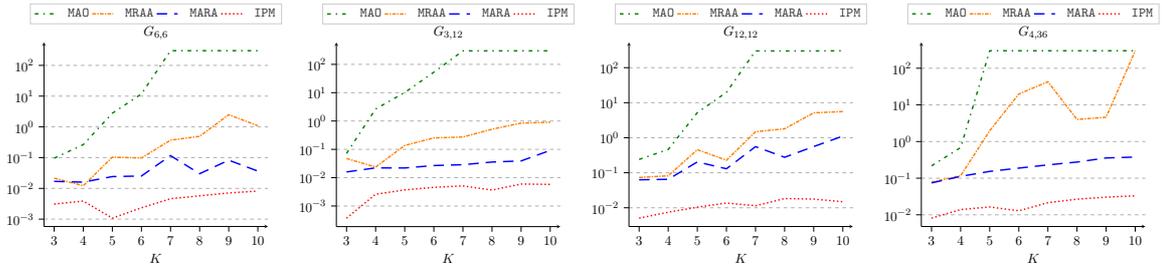

\centering
\resizebox*{!}{3.5cm}{
	\input{Comp/6-6-CPUS-Gr-TS2.tex}		
	\input{Comp/3-12-CPUS-Gr-TS2.tex}		
	\input{Comp/12-12-CPUS-Gr-TS2.tex}		
	\input{Comp/4-36-CPUS-Gr-TS2.tex}		
}
\caption{Run times in grid networks (seconds -- log scale)\label{fig:ipm_cpus_gr}}
\end{figure}

\begin{figure}[htb]
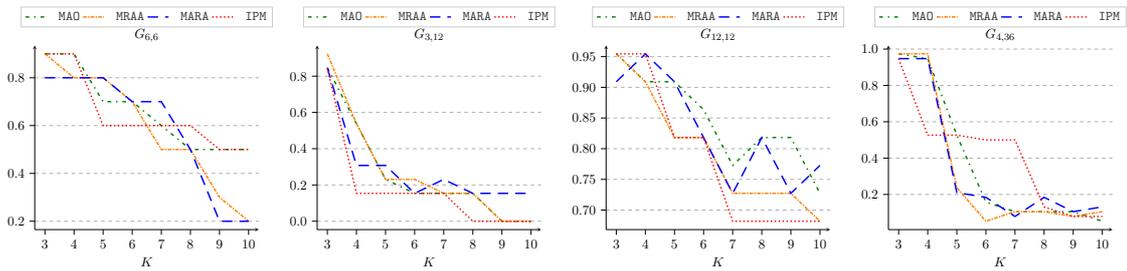

\centering
\resizebox*{!}{3.5cm}{
	\input{Comp/6-6-MiDi-Gr-TS2.tex}		
	\input{Comp/3-12-MiDi-Gr-TS2.tex}		
	\input{Comp/12-12-MiDi-Gr-TS2.tex}		
	\input{Comp/4-36-MiDi-Gr-TS2.tex}		
}
\caption{Minimum dissimilarity in grid networks\label{fig:ipm_midi_gr}}
\end{figure}

\cleardoublepage
\section{Dissimilarities and run times for the \texttt{IPM}\label{app:3}}

\begin{table}[htb]
\centering
\caption{Average \texttt{AvDi} and \texttt{MiDi} of \texttt{IPM} in random networks\label{tab:ipm_AvDi_MiDi_r}}
{\scriptsize{
\begin{tabular}{@{}lccccccccl@{}}\toprule
	\texttt{AvDi}&\multicolumn{9}{c}{$K$} \\\cmidrule(lr){2-9} 
	$R_{n,m}$ 	  & 3 & 4 & 5 & 6 & 7 & 8 & 9 & 10\\\cmidrule(lr){1-9}
	$R_{100,500}$ & 0.712 & 0.823 & 0.841 & 0.847 & 0.839 & 0.846 & 0.830 & 0.830\\
	$R_{100,1000}$ & --   & 0.857 & 0.910 & 0.918 & 0.941 & 0.943 & 0.947 & 0.941\\\cmidrule(lr){1-9}
	$R_{300,1500}$ & 0.774 & 0.868 & 0.892 & 0.895 & 0.890 & 0.886 & 0.888 & 0.885\\
	$R_{300,3000}$ & 0.917 & 0.913 & 0.937 & 0.919 & 0.929 & 0.942 & 0.942 & 0.941\\\cmidrule(lr){1-9}
	$R_{500,2500}$ & 0.896 & 0.911 & 0.904 & 0.907 & 0.911 & 0.911 & 0.906 & 0.904\\
	$R_{500,5000}$ & 0.764 & 0.853 & 0.867 & 0.894 & 0.909 & 0.917 & 0.927 & 0.927\\\midrule
	Average        & 0.812 & 0.870 & 0.891 & 0.896 & 0.903 & 0.907 & 0.906 & 0.904 & \fbox{0.886}\\\bottomrule
\end{tabular}
\begin{tabular}{@{}lccccccccl@{}}\toprule
	\texttt{MiDi}&\multicolumn{9}{c}{$K$} \\\cmidrule(lr){2-9}
	$R_{n,m}$	  & 3 & 4 & 5 & 6 & 7 & 8 & 9 & 10\\\cmidrule(lr){1-9}
	$R_{100,500}$ & 0.429 & 0.485 & 0.484 & 0.428 & 0.373 & 0.301 & 0.212 & 0.158\\
	$R_{100,1000}$ & --   & 0.145 & 0.312 & 0.316 & 0.430 & 0.419 & 0.423 & 0.369\\\cmidrule(lr){1-9}
	$R_{300,1500}$ & 0.613 & 0.642 & 0.621 & 0.621 & 0.491 & 0.347 & 0.338 & 0.251\\
	$R_{300,3000}$ & 0.750 & 0.531 & 0.550 & 0.541 & 0.503 & 0.521 & 0.500 & 0.464\\\cmidrule(lr){1-9}
	$R_{500,2500}$ & 0.736 & 0.640 & 0.634 & 0.576 & 0.569 & 0.526 & 0.498 & 0.448\\
	$R_{500,5000}$ & 0.525 & 0.588 & 0.539 & 0.530 & 0.564 & 0.564 & 0.580 & 0.580\\\midrule
	Average        & 0.610 & 0.505 & 0.523 & 0.502 & 0.488 & 0.446 & 0.425 & 0.378 & \fbox{0.484}\\\bottomrule
\end{tabular}}}
\end{table}

\begin{table}[htb]
\centering
\caption{Run times of \texttt{IPM} (seconds)\label{tab:IPMcpu0}}
{\scriptsize{
\begin{tabular}{@{}lccccccccl@{}}\toprule
	&\multicolumn{8}{c}{$K$} \\\cmidrule(lr){2-9}
	$R_{n,m}$ & 3 & 4 & 5 & 6 & 7 & 8 & 9 & 10\\\cmidrule(lr){1-9}
	$R_{100,500}$  & 0.021 & 0.028 & 0.034 & 0.040 & 0.046 & 0.057 & 0.063 & 0.070\\
	$R_{100,1000}$ & --    & 0.053 & 0.065 & 0.076 & 0.089 & 0.101 & 0.113 & 0.126\\\cmidrule(lr){1-9}
	$R_{300,1500}$ & 0.064 & 0.087 & 0.110 & 0.129 & 0.150 & 0.176 & 0.197 & 0.219\\
	$R_{300,3000}$ & 0.116 & 0.151 & 0.181 & 0.210 & 0.244 & 0.279 & 0.311 & 0.344\\\cmidrule(lr){1-9}
	$R_{500,2500}$ & 0.109 & 0.139 & 0.174 & 0.207 & 0.245 & 0.285 & 0.317 & 0.351\\
	$R_{500,5000}$ & 0.191 & 0.251 & 0.322 & 0.395 & 0.469 & 0.545 & 0.599 & 0.691\\\midrule
	Average        & 0.100 & 0.118 & 0.147 & 0.176 & 0.207 & 0.240 & 0.266 & 0.300 & \fbox{0.194}\\\bottomrule
\end{tabular}
\begin{tabular}{@{}lccccccccl@{}}\toprule	
	&\multicolumn{8}{c}{$K$} \\\cmidrule(lr){2-9}
	$G_{p,q}$ & 3 & 4 & 5 & 6 & 7 & 8 & 9 & 10\\\cmidrule(lr){1-9}
	$G_{3,12}$ & 0.000 & 0.002 & 0.003 & 0.004 & 0.005 & 0.003 & 0.005 & 0.005 \\
	$G_{4,36}$ & 0.008 & 0.013 & 0.016 & 0.012 & 0.021 & 0.026 & 0.030 & 0.033 \\
	$G_{6,6}$  & 0.003 & 0.003 & 0.001 & 0.002 & 0.004 & 0.000 & 0.007 & 0.008 \\
	$G_{12,12}$& 0.005 & 0.007 & 0.010 & 0.013 & 0.011 & 0.018 & 0.017 & 0.014 \\\midrule
	Average    & 0.004 & 0.006 & 0.007 & 0.007 & 0.010 & 0.011 & 0.014 & 0.015 & \fbox{0.009}\\\bottomrule
\end{tabular}}}
\end{table}

\begin{table}[htb]
\centering
\caption{Average \texttt{AvDi} and \texttt{MiDi} of \texttt{IPM} in grid networks\label{tab:ipm_AvDi_MiDi_g}}
{\scriptsize{
\begin{tabular}{@{}lcccccccccl@{}}\toprule
	\texttt{AvDi} &\multicolumn{9}{c}{$K$}\\\cmidrule(lr){2-9}
	$G_{p,q}$	& 3 & 4 & 5 & 6 & 7 & 8 & 9 & 10\\\cmidrule(lr){1-9}
	$G_{3,12}$  & 0.948 & 0.833 & 0.800 & 0.769 & 0.758 & 0.741 & 0.732 & 0.728 \\
	$G_{4,36}$  & 0.892 & 0.912 & 0.982 & 0.859 & 0.830 & 0.817 & 0.805 & 0.798 \\
	$G_{6,6}$   & 0.933 & 0.933 & 0.900 & 0.893 & 0.866 & 0.857 & 0.844 & 0.840 \\
	$G_{12,12}$ & 0.969 & 0.969 & 0.954 & 0.951 & 0.939 & 0.935 & 0.929 & 0.926 \\\midrule
	Average     & 0.958 & 0.911 & 0.886 & 0.868 & 0.848 & 0.837 & 0.827 & 0.823 & \fbox{0.870}\\\bottomrule
\end{tabular}
\begin{tabular}{@{}lcccccccccl@{}}\toprule
	\texttt{MiDi}&\multicolumn{9}{c}{$K$} \\\cmidrule(lr){2-9}
	$G_{p,q}$  & 3 & 4 & 5 & 6 & 7 & 8 & 9 & 10\\\cmidrule(lr){1-9}
	$G_{3,12}$ & 0.846 & 0.153 & 0.153 & 0.153 & 0.153 & 0.000 & 0.000 & 0.000 \\
	$G_{4,36}$ & 0.947 & 0.526 & 0.526 & 0.500 & 0.500 & 0.131 & 0.078 & 0.078 \\
	$G_{6,6}$  & 0.900 & 0.900 & 0.600 & 0.600 & 0.600 & 0.600 & 0.500 & 0.500 \\
	$G_{12,12}$& 0.954 & 0.954 & 0.818 & 0.818 & 0.681 & 0.681 & 0.681 & 0.681 \\\midrule
	Average    & 0.911 & 0.633 & 0.524 & 0.517 & 0.483 & 0.353 & 0.314 & 0.314 & \fbox{0.506}\\\bottomrule
\end{tabular}}}
\end{table}

\end{document}